\documentclass[11pt]{amsart}
\usepackage{amsmath, amssymb}
\usepackage{amscd}
\usepackage{url}
\usepackage{listings}

\usepackage{amssymb}
\usepackage{amsxtra}

\usepackage{mathabx}
\usepackage{color}
\usepackage[all]{xy}
\usepackage{mathrsfs}
\usepackage{stmaryrd}
\usepackage[utf8]{inputenc}
\usepackage{eucal}
\usepackage{fullpage}
\usepackage{hyperref}
\usepackage{relsize}

\def\Dbar{\leavevmode\lower.6ex\hbox to 0pt{\hskip-.23ex \accent"16\hss}D}
  \def\cftil#1{\ifmmode\setbox7\hbox{$\accent"5E#1$}\else
  \setbox7\hbox{\accent"5E#1}\penalty 10000\relax\fi\raise 1\ht7
  \hbox{\lower1.15ex\hbox to 1\wd7{\hss\accent"7E\hss}}\penalty 10000
  \hskip-1\wd7\penalty 10000\box7}
  \def\cfudot#1{\ifmmode\setbox7\hbox{$\accent"5E#1$}\else
  \setbox7\hbox{\accent"5E#1}\penalty 10000\relax\fi\raise 1\ht7
  \hbox{\raise.1ex\hbox to 1\wd7{\hss.\hss}}\penalty 10000 \hskip-1\wd7\penalty
  10000\box7} \def\cprime{$'$}

\renewcommand{\leq}{\leqslant}
\renewcommand{\geq}{\geqslant}
 
\numberwithin{equation}{section}

\newcommand{\C}{\mathbf{C}}
\newcommand{\N}{\mathbf{N}}

\newcommand{\A}{\mathbf{A}}

\newcommand{\Z}{\mathbf{Z}}
\renewcommand{\P}{\mathbf{P}}
\newcommand{\R}{\mathbf{R}}

\newcommand{\Q}{\mathbf{Q}}

\newcommand{\Tt}{\mathbf{T}}
\newcommand{\G}{\mathbf{G}}

\newcommand{\conv}{\mathrm{conv}}

\newcommand{\mods}[1]{\,(\mathrm{mod}\,{#1})}

\DeclareMathOperator{\Vol}{Vol}

\DeclareMathOperator{\rank}{rank}
\DeclareMathOperator{\res}{Res}

\DeclareMathOperator{\cl}{Cl}

\DeclareMathOperator{\vol}{Vol}
\DeclareMathOperator{\Imag}{Im}
\DeclareMathOperator{\real}{Re}

\DeclareMathOperator{\lcm}{lcm}

\DeclareMathOperator{\Stab}{Stab}

\DeclareMathOperator{\Gal}{Gal}
\DeclareMathOperator{\Ind}{Ind}

\DeclareMathOperator{\Hom}{Hom}

\DeclareMathOperator{\codim}{codim}

\DeclareMathOperator\Fr{Fr}

\DeclareMathOperator\Spec{Spec}

\newcommand{\eps}{\varepsilon}
\renewcommand{\rho}{\varrho}

\DeclareMathOperator{\GL}{GL}

\DeclareMathOperator{\sgn}{sgn}

\DeclareMathOperator\Ann{Ann}
\DeclareMathOperator\coker{coker}

\newcommand{\mx}[4]{\left( \begin{array}{cc} #1 & #2 \\ #3 & #4 \end{array} \right)}
\def\LT{{}^L T}
\newcommand{\fc}{\mathfrak{c}}
\newcommand{\fK}{\mathfrak{K}}

\DeclareMathSymbol{\gena}{\mathord}{letters}{"3C}
\DeclareMathSymbol{\genb}{\mathord}{letters}{"3E}

\newcounter{bnd}

\theoremstyle{plain}
\newtheorem{theorem}{Theorem}[section]
\newtheorem*{theorem*}{Theorem}
\newtheorem{lemma}[theorem]{Lemma}
\newtheorem{corollary}[theorem]{Corollary}
\newtheorem{conjecture}[theorem]{Conjecture}

\newtheorem{proposition}[theorem]{Proposition}

\theoremstyle{remark}

\newtheorem*{rem}{Remark}
\newtheorem*{rems}{Remarks}

\theoremstyle{definition}

\newtheorem{definition}[theorem]{Definition}
\newtheorem{propdef}[theorem]{Proposition/Definition}

\newtheorem*{question}{Question}
\newtheorem{example}[theorem]{Example}

\newcommand{\fp}{\mathfrak{p}}
\newcommand{\fP}{\mathfrak{P}}

\newcommand{\cL}{\mathcal{L}}
\newcommand{\cS}{\mathcal{S}}

\newcommand{\cF}{\mathcal{F}}

\newcommand{\cI}{\mathcal{I}}

\newcommand{\cB}{\mathcal{B}}

\def\cA{\mathcal{A}}
\def\cO{\mathcal{O}}
\def\cT{\mathcal{T}}

\renewcommand{\geq}{\geqslant}
\renewcommand{\leq}{\leqslant}

\newcommand{\authnote}[2][]{\noindent {\if!#1!  \textbf{TODO} \else \textbf{\small  #1} \fi: #2} \vspace{0.1in}}

\begin{document}
 \title{The Weyl law for algebraic tori}
 \author{Ian Petrow}
 \address{ETH Z\"urich - Departement Mathematik \\
HG G 66.4 \\
R\"amistrasse 101 \\
8092 Z\"urich \\
Switzerland }
\email{ian.petrow@math.ethz.ch}
\curraddr{UCL - Department of Mathematics \\
25 Gordon Street \\
London WC1H 0AY \\
United Kingdom}
\email{i.petrow@ucl.ac.uk}
\thanks{This work was supported by Swiss National Science Foundation grant PZ00P2\_168164.}
 
 \begin{abstract}
We give an asymptotic evaluation for the number of automorphic characters of an algebraic torus $T$ with bounded analytic conductor. The analytic conductor is defined via the local Langlands correspondence for tori by choosing a finite dimensional complex algebraic representation of the $L$-group of $T$. Our results therefore fit into a general framework of counting automorphic representations on reductive groups by analytic conductor. 
 \end{abstract} 

 \maketitle

 \tableofcontents

\section{Introduction}\label{intro}

 \subsection{Motivation and statement of results}\label{statementofresults}
A basic question in the analytic theory of automorphic forms is the following:

\begin{question}\label{question}Given a connected reductive algebraic group $G$ over a global field $k$, how many irreducible cuspidal automorphic representations of $G$ are there?\end{question}

To make sense of the Question, one needs to choose a positive real-valued invariant by which to order the representations of $G$. Sarnak, Shin and Templier \cite{SarnakShinTemplier} have proposed using the \emph{analytic conductor. }

On the groups $\GL_m$, the analytic conductor has a standard definition \cite{IS2000GAFA}, but over more general reductive groups it is less well understood. The most canonical (but not necessarily the most practical) definition is through the local Langlands conjectures. Let $r: {}^LG \to \GL_m (\C)$ be a finite dimensional algebraic representation of the complex $L$-group of $G$.  The local Langlands conjectures predict the existence of maps $$ r_{*,v} : \cA_v(G) \to \cA_v(\GL_m)$$ at every place $v$ of $k$ from the local unitary dual $\cA_v$ of $G$ to that of $\GL_m$. One then defines the analytic conductor $\fc(\pi,r)$ of an irreducible automorphic representation $\pi$ with respect to $r$ by \begin{equation}\label{acG}\fc(\pi,r) = \prod_{v} \fc_v(r_{*,v} \pi_v),\end{equation} where the conductors on the right hand side are the ``classical'' local analytic conductors on $\GL_m$.

The universal counting Question seems to be quite difficult at the level of generality in which we have stated it. Only very recently has there been progress in a few special cases. Over an arbitrary number field, the cases $G=\GL_1$, $\GL_2$, as well $\GL_n$ for $n\geq 3$ under additional hypotheses, have been resolved in a paper of Brumley and Mili\'cevi\'c \cite{BrumleyMilicevic}. The case that $G$ is a one-dimensional non-split torus over $\Q$ splitting over an imaginary quadratic extension was treated in work of Brooks and the author \cite{BrooksPetrow}, and Lesesvre has studied the case that $G$ is the units group of a quaternion algebra \cite{LesesvreCounting}.

In this paper, we present an answer to the Question for $G=T$ a torus over a number field $k$ and $r$ an arbitrary complex algebraic representation of its $L$-group. Even though the groups we are dealing with are abelian, our results are not easy, as we work with a very general notion of conductor. Indeed, the difficulties involved are already evident in the intricacy of the statement of the final result. As its reward, working with such a general notion of conductor reveals some of the richness that any general answer to the Question must exhibit.  For example, the power of $X$ in the asymptotic count of automorphic characters (see \eqref{Adef}) need not be an integer, but rather is a positive rational with denominator at most $m$. Further, we find that arbitrary integer powers of $\log X$ are possible in the asymptotic count (see example \ref{eglog}). Another interesting aspect of our results is the resemblance of the automorphic counting question to the Manin conjecture, which we present in section \ref{sec:manin}.

We make some precise definitions in order to give the statement of our result.
Let $T$ be an algebraic torus over a number field $k$. Let $\cA(T)$ denote the group of continuous unitary characters of $T(k) \backslash T(\A)$, where $\A$ is the ad\`ele ring of $k$.  We call elements of $\cA(T)$ \emph{automorphic characters}; they are the basic objects of study in this paper. 
Let $K/k$ be the minimal Galois extension over which $T$ splits, and let $G= \Gal(K/k)$. Let $X^*(T)$ and $X_*(T)$ be the algebraic character and cocharacter lattices of $T$, and $\widehat{T} = \Hom(X_*(T),\C^\times)$ the \emph{complex dual torus}.    
Each of these objects admits a natural action of $G$.  Let ${}^LT = \widehat{T} \rtimes G$ be the \emph{$L$-group} of $T$, and pick $r: {}^LT \to \GL_m (\C)$ an algebraic representation of ${}^LT$. Generally, we will write $n = \dim T$ and $m = \dim r$. Pick $\nu$ a Haar measure on the locally compact group $\cA(T)$. 

The main goal of this paper is to give an asymptotic formula for $\nu(\{\chi \in \cA(T): \fc(\chi,r)\leq X\})$, where $\fc(\chi,r)$ is the \emph{analytic conductor} (defined in section \ref{sec:globalconductor}), as $X$ tends to infinity. The statement of the result requires a few more constructions.
The restriction of $r$ to $\widehat{T}$ breaks up as a direct sum of eigenspaces \begin{equation}\label{rT}r \vert_{\widehat{T}}  = \bigoplus_{\mu \in X^*(\widehat{T})} V_\mu.\end{equation} Let $M$ be the multi-set of \emph{co-weights} $\mu$ of $r$, i.e.\ the underlying set of $M$ is $\{\mu \in X^*(\widehat{T}): V_\mu \neq 0\}$ and the multiplicity of $\mu\in M$ is $\dim V_\mu$. Let $S \subseteq M$ denote a subset of the co-weights with multiplicity and $S^c$ its complement.  For such an $S$, we define the complex group \begin{equation}\label{DS} D(S)= \bigcap_{\mu \in S^c} \ker \mu \subseteq \widehat{T}.\end{equation} The restriction $r \vert_{\widehat{T}}$ is faithful if and only if $D(\varnothing) = \{1\}$, and in that case we let \begin{equation}\label{Adef} A=A(T,r)= \max\Big\{\frac{\dim D(S) + 1}{|S|} : S \subseteq M, D(S)\neq \{1\}\Big\}. \end{equation}
\begin{theorem}\label{MT1}
Suppose that $r \vert_{\widehat{T}}$ is faithful.  Then there exists a non-zero polynomial $P = P_{T,r,\nu}$ and $c=c_{T,r}>0$ such that \begin{equation}\label{eq:MT}\nu(\{\chi \in \cA(T): \fc(\chi,r) \leq X\}) = X^A P(\log X) + O_{T,\nu,r}(X^{A}\exp(-c (\log X)^{3/5} (\log \log X)^{-1/5})).\end{equation} If the Artin conjecture holds for the finitely many Artin $L$-functions specified in Theorem \ref{unrthmwithArtin}, then the error term in \eqref{eq:MT} may be improved to $O_{T,r,\nu}(X^{A-\delta})$ for some $\delta=\delta_{T,r}>0$. If $r \vert_{\widehat{T}}$ is not faithful, then the left hand side of \eqref{eq:MT} is infinite for some finite $X$.\end{theorem}
The dependence of $P$ and the implicit constant on $\nu$ is linear, since Haar measure is unique up to scaling. Here is a simple corollary of Theorem \ref{MT1}.
\begin{corollary}\label{C}
Let $T$ be a torus of dimension $n$, $r$ an $m$-dimensional complex representation of its $L$-group, and $\nu$ a Haar measure on $\cA(T)$. We have $$\nu(\{\chi \in \cA(T):\fc(\chi,r)\leq X\}) \gg_{T,r,\nu} X^{\frac{n+1}{m}}.$$ If $r \vert_{\widehat{T}}$ is faithful, then for all $\eps>0$ we have $$\nu(\{\chi \in \cA(T):\fc(\chi,r)\leq X\}) \ll_{\eps,T,r,\nu} X^{2+\eps}.$$ \end{corollary}
\begin{proof}
By Theorem \ref{MT1} it suffices to give uniform lower and upper bounds on $A$. For the lower bound, note that $D(M) = \widehat{T}$ which gives $A \geq \frac{n+1}{m}.$ For the upper bound, observe that since $r \vert_{\widehat{T}}$ is faithful for any $S \subseteq M$ we have $\dim D(S) \leq |S|$, since $\dim D(\varnothing)=0$ and $\codim \ker \mu \leq 1$ for any $\mu \in M$. Therefore $A \leq \max \{\frac{|S|+1}{|S|}: S \neq \varnothing \} \leq 2$. \end{proof}
We can give an expression for the degree of the polynomial $P$, but this requires a few more definitions. Since $M$ was formed from the restriction of a representation of ${}^LT$, the group $G$ acts on $M$, and also on the power set $2^{M}=\{S: S \subseteq M\}.$ This action preserves $|S|$ as well as $\dim D(S)$, so $G$ also acts on the set \begin{equation}\label{Sigma}\Sigma = \{S\neq \varnothing: \frac{\dim D(S)+1}{|S|} = A\}.\end{equation}
Let \begin{equation}\label{lambda}\lambda =  \lcm_{S \in \Sigma} | \pi_0 (D(S))|,\end{equation} where $\pi_0(D(S))$ denotes the group of connected components of $D(S)$. The group $(\Z/\lambda \Z)^\times$ acts on $\pi_0 (D(S))$ for each $S \in \Sigma$ by $\ell.y = y^\ell$, $y\in \pi_0 (D(S))$, $\ell \in (\Z/\lambda \Z)^\times$. Let $\zeta_\lambda$ be a primitive $\lambda$th root of unity, and let $\widetilde{K} = K(\zeta_\lambda)$ and $\widetilde{G} = \Gal(\widetilde{K}/k)$. The enlarged group $\widetilde{G}$ acts on the fibered set with base $\Sigma$ given by \begin{equation}\label{Sigmatilde}\widetilde{\Sigma} = \{(S,y): S \in \Sigma, \quad y \in \pi_0(D(S))\}.\end{equation} Indeed, we have inclusions \begin{equation}\label{Gtildeinclusion} \widetilde{G} \hookrightarrow G \times \Gal(k(\zeta_\lambda)/k) \hookrightarrow G \times (\Z/\lambda\Z)^\times\end{equation} given by restricting automorphisms to $K$ and to $\Q(\zeta_\lambda)$. If $g \in \widetilde{G}$ restricts to $(\overline{g},\gamma) \in G \times (\Z/\lambda\Z)^\times,$ then $g$ acts on $\widetilde{\Sigma}$ by $$ g.(S,y) = (\overline{g}S,\overline{g}y^{\gamma}).$$ Finally, let \begin{equation}\label{Sigmatilde0}\widetilde{\Sigma}_0 = \widetilde{\Sigma} \smallsetminus \{(S,1): \dim D(S) = 0\}.\end{equation} Since the deleted set is preserved by the action of $\widetilde{G}$, we also have that $\widetilde{G}$ acts on $\widetilde{\Sigma}_0$. 
\begin{theorem}\label{MT2}
The polynomial $P$ appearing in Theorem \ref{MT1} satisfies $$\deg P = \left| \widetilde{G} \backslash \widetilde{\Sigma}_0 \right|-1.$$ \end{theorem}
Theorems \ref{MT1} and \ref{MT2} settle a problem of Sarnak, Shin and Templier \cite[(4)]{SarnakShinTemplier} for the universal family of automorphic characters on a torus in the greatest possible generality. 

\begin{rems}
\begin{enumerate}
\item The Weyl law for tori over number fields studied in this paper has a natural local-to-global structure. To solve the global counting problem we first address the corresponding local problem in sections \ref{nonarch_top} and \ref{archimedean_top} of the paper.  We obtain fairly complete results in the unramified non-archimedean and archimedean cases, but only need a preliminary result (see Theorem \ref{locaramthm}) in the ramified non-archimedean case for our global application. A natural further line of inquiry would be to more comprehensively investigate the local Weyl law for tori that have a ramified non-archimedean splitting field.  Similarly, it would be interesting to investigate the counting problem over positive characteristic global fields. 
\item We would also like to have an interpretation of the leading constant in the asymptotic formula in Theorem \ref{MT1} in terms of the geometry or arithmetic of $T$. While in principle our method yields an expression for the leading constant, it is not so easy to write down in explicit form.  One difficulty is that we cannot exclude the possibility that there are non-identity global units of $T$ that contribute to the leading term of the polynomial $P$ (see example \ref{product}). A second complication is our soft treatment of the local counting problem at ramified primes, as indicated in the previous remark. 
\item The invariant $A$ in Theorem \ref{MT1} and the power of $\log X$ in Theorem \ref{MT2} are sufficiently complicated as to suggest that any general answer to the Question at the beginning of this paper would be quite onerous to state in full generality.
\item A key tool in our proof of Thereom \ref{MT1} for general $r$ is a Brascamp-Lieb inequality due to Barthe \cite{BartheBLinequality}. To the author's knowledge, this theorem from analysis has not been used before in analytic number theory. We use the Brascamp-Lieb inequality in two places: first in the local archimedean counting problem in section \ref{archimedean} and then again in the global counting problem in section \ref{conclusion}. In the first instance, the linear forms correspond to the co-weights $\mu \in M$, and in the second instance the linear forms are rows of the regulator matrix of $T$. 
The use of the Brascamp-Lieb inequality suggests that the counting problem for a general reductive group is difficult indeed, as already in the case of tori one needs to go much beyond an explicit understanding of the local Langlands correspondence. 
\item \label{autrem}Another interpretation of the families of automorphic characters studied in this paper is the following. Let $T$, $K/k$ be as above, $T^\vee=\Hom (X_*(T),\G_m)$ be the \emph{algebraic dual torus}, and $S = \res_{K/k} \G_m$.  Given a faithful irreducible algebraic representation $r$ of $\LT$, one obtains an injective morphism $i: T^\vee \to S$ by restriction of $r$. Such an injective map $i$ gives rise to an $L$-homomorphism $\LT \to {}^LS$, and so Langlands predicts that there exists a transfer of automorphic characters $i_*: \cA(T/k) \to \cA(\GL_1/K)$. Conversely, given $i:T^\vee \to S$, there exists a faithful irreducible algebraic representation $r$ of $\LT$ extending $i$ such that $L(s, \chi,r ) = L(s,i_*\chi)$ for all $\chi \in \cA(T)$, where the left hand side is the Langlands $L$-function and right hand side is the Hecke $L$-function. 
\item The automorphic counting problem outlined at the beginning of this paper may have applications to the Ramanujan conjecture on general reductive groups (see the surveys \cite{SarnakNotesonRamanujan} and \cite{ShahidiRamanujanConj}). Outside the case $G=\GL_n$, the naive Ramanujan conjecture is known to be false, but all automorphic forms for which it fails are expected to arise as functorial transfers from lower rank groups. For analytic applications, one would like to show that the Ramanujan conjecture cannot fail ``too often'' in a quantitative sense in terms of analytic conductor.  
One way to do so would be to estimate the sizes of subfamilies of $\cA(G)$ coming from functorial transfers of automorphic characters of tori, and the so present paper paves the way for putting the above program into action.
\item The shape of the error term in Theorem \ref{MT1} comes from Vinogradov-Korobov-strength zero-free regions for $L$-functions of Hecke characters due to Coleman \cite{Coleman}, but these are not essential to our method. This zero-free region is merely the best currently available result in the literature, and e.g.\ using instead the classical zero-free region for Hecke $L$-functions one obtains an asymptotic formula in Theorem \ref{MT1} with the weaker error term $O_{T,r,\nu}(X^{A}\exp(-c\sqrt{\log X}))$. 
\end{enumerate}
\end{rems}

\subsection{Examples}\label{examples}
\begin{example} Let $T = \GL_1 = \G_m$. Then $\widehat{T} = \C^\times$ and $G$ is trivial. We choose $r={\rm id}= z:\C^\times \to \C^\times$ as representation of the $L$-group. Then $\cA(T)$ is the set of primitive Hecke characters over $k$, and $\fc(\chi,r)$ is the standard notion of analytic conductor of a Hecke character, which we denote by $C(\chi)$ in all of the examples that follow.  The multiset of co-weights is the singleton $M=\{z\}$, and $2^{\{z\}}= \{\varnothing, \{z\}\}.$ We have $D(\varnothing) = \{1\}$ and $D(\{z\}) = \C^\times,$ so we have $A = 2$, and $\deg P = 0$. Therefore there are $\sim c_k X^2$ primitive Hecke characters of analytic conductor bounded by $X$ for some constant $c_k>0$. An inspection of Theorem \ref{unrthmwithArtin} shows that the only Artin $L$-function relevant to Theorem \ref{MT1} when $T=\G_m$ and $r={\rm id}$ is the Dedekind zeta function of $k$, so we obtain a power saving error term in the asymptotic count. For general number fields $k$, already this result seems to be new. A similar result with a modified notion of analytic conductor has been given recently in a preprint of  Brumley, Lesesvre, and Mili\'cevi\'c \cite{BrumleyLesesvreMilicevic}.
\end{example}
\begin{example}
Let $T = \GL_1 = \G_m$ as above, but take as representation the $1001$-dimensional representation $r=z^{\oplus 1001}: \C^\times \to \GL_{1001} (\C)$.  
The set $\cA(T)$ consists of Hecke characters $\chi$ as above, whereas $r$ assigns to $\chi$ the conductor $C(\chi)^{1001}$. The multiset of co-weights is $\{z,\ldots,z\}$, where $z$ is repeated 1001 times, and only the full set has $D(S)\neq \{1\}.$ Thus, $A=2/1001$, and one recovers that there are $\sim c_{k,r} X^{2/1001}$ Hecke characters of $r$-conductor less than $X$. This shows that the power of $X$ in Theorem \ref{MT1} can be arbitrarily small.
\end{example}
\begin{example}
Keep $T = \GL_1 = \G_m$ as above, but take as representation $r= z^2 \oplus z^3: \C^\times \to \GL_2 (\C)$.  This is a $2$-dimensional faithful representation of the $L$-group.  
The set $\cA(T)$ is as in the previous two examples, but now $r$ assigns to $\chi$ the conductor $C(\chi^2)C(\chi^3)$. 
The set of co-weights is $\{z^2,z^3\}$, and the subsets $S$ and groups $D(S)$ are $$S = \varnothing, \quad \{z^2\}, \quad \{z^3\}, \quad \{z^2,z^3\},$$ $$ D(\varnothing) = \{1\}, \quad D(\{z^2\}) = \mu_3,\quad D(\{z^3\}) = \pm 1, \quad D(\{z^2,z^3\}) = \C^\times.$$
Therefore $A=1$, and the maximum is attained on all $S \neq \varnothing$. We have the Galois group $G = 1$, but the enlarged Galois group is $\widetilde{G} \simeq (\Z/6\Z)^\times$. We have $$\widetilde{\Sigma} = \{ (\{z^2\},1),  (\{z^2\},\zeta_3), (\{z^2\},\overline{\zeta_3}),  (\{z^3\},1),  (\{z^3\},-1), (\{z^2, z^3\},1)\},$$ $$ \widetilde{\Sigma}_0 = \{ (\{z^2\},\zeta_3), (\{z^2\},\overline{\zeta_3}),    (\{z^3\},-1), (\{z^2, z^3\},1)\}.$$ The group $\widetilde{G}$ acts on $\widetilde{\Sigma}_0$ by swapping $(\{z^2\},\zeta_3)$ and $(\{z^2\},\overline{\zeta_3})$ and fixing $(\{z^3\},-1)$ and $ (\{z^2, z^3\},1)$. Therefore there exists a constant $c_{23}$ so that $$\nu(\{\chi \in \cA(\GL_1) : C(\chi^2)C(\chi^3)\leq X \}) \sim c_{23} X (\log X)^2.$$
\end{example}
\begin{example}\label{product}
Let $T = \G_m \times \G_m$.  This torus has $L$-group equal to $\C^\times \times \C^\times$.  Take the faithful $2$-dimensional representation $z_1 \oplus z_2$.  Classically, this corresponds to counting pairs of primitive Hecke characters $(\chi_1,\chi_2)$ with the conductor $C(\chi_1)C(\chi_2)$.  There are $\sim c_{12} X^2 \log X$ pairs of Hecke characters of conductor bounded by $X$, for some $c_{12}>0$.

For a general torus $T$, we will later see in Proposition \ref{localtoglobal} that there is a term which potentially contributes to the main term of $\nu(\{\chi \in \cA(T): \fc(\chi,r)\leq X\})$ as $X \to \infty$ for each global unit of $T$. In the example $T = \G_m \times \G_m$ over $\Q$, there are four global units: $(1,1)$, $(1,-1)$, $(-1,1)$ and $(-1,-1)$. It is interesting to note that the contribution of $(1,1)$ is of size $X^2 \log X$, each of $(1,-1)$ and $(-1,1)$ make a contribution of size $X^2$, and $(-1,-1)$ contributes a smaller a power of $X$.  
\end{example}
\begin{example}\label{eglog}
 Let $f$ be an irreducible separable polynomial of degree $m$ over a general field $k$. Let $\alpha$ be a root of $f$ and let $K$ be the splitting field of $f$. Then $k(\alpha)/k$ is a degree $m$ extension with Galois closure $K$. Let $G=\Gal(K/k)$ and
$T= (\res_{k(\alpha)/k} \G_m)/ \G_m,$ where $d: \G_m\to \res_{k(\alpha)/k} \G_m$ is the diagonal embedding.  Write $H=\Hom_k(k(\alpha),k^{\rm{sep}})$ for the set of embeddings of $k(\alpha)$ in its separable closure that fix $k$. We have in particular that $|H|=m$.
 One has an identification \begin{equation}\label{Res_char_lattice}X^*(\res_{k(\alpha)/k} \G_m) \simeq \Z^{H}\end{equation} sending $(a_\sigma)_{\sigma \in H}$ to the morphism $\chi:\res_{k(\alpha)/k} \G_m \to \G_m$ that is given on $k^{\rm sep}$ points by $$\chi: (k(\alpha) \otimes_k k^{\rm sep} )^\times \to (k^{\rm sep})^\times$$ $$c \otimes r \mapsto \left( \prod_{\sigma \in H}\sigma(c)^{a_\sigma}\right)r,$$ see \cite[Lem.\ 12.61]{MilneAGS}. Here, $G$ acts on $\Z^{H}$ by permuting embeddings, i.e.\ by permuting coordinates. Write $\Z_0^H :=\{ (a_\sigma) \in \Z^H: \sum_{\sigma\in H} a_\sigma =0\}$. We have the following commutative diagram with exact rows and vertical maps given by \eqref{Res_char_lattice} 
 $$ \xymatrix{ 0 \ar[r] & X^*(T) \ar[r] & X^*(\res_{k(\alpha)/k} \G_m) \ar[d]^{\simeq} \ar[r]^-{d^*} & X^*(\G_m) \ar[d]^\simeq \ar[r] & 0\\
 0 \ar[r] & \Z_0^H \ar[r] & \Z^H \ar[r]^{\sum_{\sigma}} & \Z \ar[r] & 0.}$$
Thus the isomorphism \eqref{Res_char_lattice} restricts to $X^*(T)\simeq  \Z_0^H$. 
 The evaluation map gives a perfect pairing $X^*(T) \otimes_\Z X_*(T) \to \Z$ between character and cocharacter lattices, so that $\Hom(X_*(T),\Z) \simeq X^*(T)$ and $$\widehat{T} = \Hom(X_*(T),\Z) \otimes_{\Z} \C^\times \simeq X^*(T) \otimes_{\Z} \C^\times\simeq \{(z_\sigma) \in \left( \C^\times \right)^{H}: \prod_{\sigma \in H} z_\sigma = 1\} .$$ See section \ref{background:tori} for more on these standard isomorphisms.

Now let us choose an enumeration $\sigma_i$, $i=1,\ldots, m$ of the embeddings $H$. Define $r:\LT \to \GL_{m}(\C)$ by setting $r((1,\cdots,1)\rtimes \sigma)$ to be the permutation matrix in $\GL_m(\C)$ defined by $\sigma \in G\subseteq S_m$, and  $$r(z \rtimes 1) = \left( \begin{array}{ccc} z_1 & & \\ & \ddots & \\ & & z_m \end{array}\right).$$ The set of co-weights $M$ is $\{z_i\}_{i=1,\ldots,m}$, where $z\mapsto z_i$ represents projection onto the $i$-th coordinate. For any $S \in 2^{\{z_i\}}$ we have $$D(S) = \{ z\in \widehat{T}: z_i =1 \text{ if } i \in S^c \text{ and } \prod_{i \in S} z_i=1\},$$
so that each $D(S)$ is connected and $\dim D(S) = |S|-1$ if $S\neq \varnothing$. Therefore, $A=1$, $\Sigma = \{S \subseteq M: S\neq \varnothing\}, \lambda =1$, and so $\widetilde{G} = G, \widetilde{\Sigma}=\Sigma$, and $\widetilde{\Sigma}_0= \{ S\subseteq M: |S|\geq 2\}$. 

Suppose now that $k$ is a number field and $G\simeq S_m$.  Then $$ \widetilde{G} \backslash \widetilde{\Sigma}_0 = \big\{ \{S:|S|=2\}, \ldots, \{S: |S|=m\}\big\}$$ and so we have $$\nu(\{ \chi \in \cA(T) : c_r(\chi) \leq X\}) \sim c_{r,T} X (\log X)^{m-2}.$$ This is an example of logarithms arising in the asymptotic formula for natural reasons, and answers a question of Sarnak \cite[Question 1]{SarnakFamilyofLfcns}. Also note that if e.g. $[k(\alpha):k]=4$ and $G \not \simeq S_4$ the asymptotic is $\sim c_G X (\log X)^3$ for some constants $c_G$, whereas if $G\simeq S_4$ it is $\sim c_{S_4} X (\log X)^2$. This shows that the power of $\log X$ in the asymptotic formula is sensitive to the arithmetic of the torus.
\end{example}

\subsection{Relation to the Manin conjecture}\label{sec:manin}
The automorphic counting Question introduced at the outset of this paper is reminiscent of the Manin conjecture on the number of rational points of bounded height on a Fano variety. We briefly review the latter to point out a few of its features. For a more developed survey of the Manin conjecture, see e.g.\ \cite{PeyreSurvey} or \cite{TschinkelSurvey}.

Let $V$ be a Fano variety over $k$, and $\cL$ a very ample line bundle. Let $s_0, \ldots, s_m$ be global sections of $\cL$ with no common zeros, and $\phi = \phi_{\cL,s_0, \ldots, s_m}: V \to \P^m$ be the natural morphism associated to these data. Let $H(x)$ be the absolute exponential Weil height on $\P^m(k)$. Then $h_\phi(x)= h(\phi(x))$ is a height function on $V(k)$ relative to $\cL, s_0,\ldots, s_m$. If $s_0',\ldots,s_m'$ is another choice of global sections for the same $\cL$ with $\phi'=\phi_{\cL, s_0',\ldots,s_m'}: V \to \P^m$, then $h_\phi(x) = h_{\phi'}(x) + O(1)$ as $x \in V(\overline{k})$ varies \cite[Thm. 3.1]{SilvermanHeights}. Following Batyrev-Manin \cite{BatyrevManin}, for $U\subseteq V$ a Zariski open let $$N_{U}(\cL,X) = \#\{x \in U(k): h_\phi(x)\leq X\}.$$ Let $N_{\rm eff}^1(V)$ be the closed cone of effective divisors. 
\begin{conjecture}[Batyrev-Manin Conj.~C$'$] 
Let $V$ be a Fano variety with canonical bundle $\omega_V$ not effective. If $U$ is sufficiently small, we have $$N_U(\cL,X) \sim c X^{\alpha(\cL)} (\log X)^{t(\cL)-1}$$ as $X \to \infty$ for some positive constant $c$. Here, $$\alpha(\cL) =  \inf \{\lambda \in \R: \lambda [ \cL] + [\omega_V] \in N_{{\rm eff}}^1\},$$ and $t(\cL)$ is the codimension of the minimal face of $\partial N_{\rm eff}^1$ containing $\alpha(\cL)[ \cL] + [\omega_V]$.
\end{conjecture}

The analogy between the automorphic counting question and the Manin conjecture is as follows, and should be viewed as an expression of the deep conjectures of Langlands. The role of the ambient space is played by $\P^m(k) \leftrightarrow \cA(\GL_m)$, into which $V(k) \leftrightarrow \cA(G)$ embeds. The embedding is given by the data $\cL,s_0,\ldots,s_m$ on the Manin side, and (conjecturally) on the automorphic side by $r:{}^LG \to \GL_m(\C)$. Indeed, $\cL,s_0,\ldots,s_m$ determine a morphism $V \to \P^m$ whereas the representation $r$ (conjecturally) determines $r_*: \cA(G) \to \cA(\GL_m)$. The absolute exponential Weil height $H(x)$ for $x \in \P^m(k)$ on the Manin side corresponds to the analytic conductor $\fc(\pi)$, $\pi \in \cA(\GL_m)$ on the automorphic side. The height function $h_\phi(x)$ relative to $\phi$ corresponds to the analytic conductor $\fc( \pi,r)$ relative to $r$ as in \eqref{acG}. 

The invariant $\alpha(\cL)$ appearing in the Manin conjecture and the invariant $A$ appearing in Theorem \ref{MT1} both are expressible in terms of combinatorial geometry problems, see the computations with matroids in section \ref{archimedean} of this paper. 

At least in the special case of tori, Theorems \ref{MT1} and \ref{MT2} suggest that $t(\cL)$ on the Manin side corresponds to the set of orbits $\widetilde{G} \backslash \widetilde{\Sigma}_0$. In both cases, the power of log comes from the possible embeddings of $V$ or $\cA(T)$ in ambient space that are ``extremal'' in the combinatorial geometry problem defining $\alpha(\cL) $ or $ A$. 

The leading constant in Manin's conjecture has been given a conjectural interpretation in terms of ad\`elic volumes by Peyre \cite{Peyre} and Chambert-Loir and Tschinkel \cite{CLTsch}. For a discussion of the significance of the leading constant in the automorphic counting problem, see \cite[\S1.5]{BrumleyMilicevic}.  

While the analogy presented here is striking, it only goes so far. In Manin's conjecture, there is a canonical choice of $\cL$, that is, one takes $\cL = - \omega_V$, the anti-canonical bundle. In the automorphic setting, there is apparently no canonical choice of complex representation $r$ of the $L$-group of $G$. Moreover, in the setting of Manin's conjecture the set of possible height functions corresponds to the ample cone of $V$, whereas in the automorphic setting, the possible height functions correspond to the set of faithful finite-dimensional complex representation of ${}^LG$. The later takes into account both the finite-dimensional representation theory of complex connected reductive groups and of global Galois groups, so seems to afford a more intricate set of height functions. Lastly, we remark that the invariant $t(\cL)$ in Manin's conjecture is an essentially global invariant of $V$. On the other hand, $\widetilde{G} \backslash \widetilde{\Sigma}_0$ has a somewhat more local nature, as we shall see in section \ref{sec:localtoglobal} of this paper.

\subsection{Outline of the proof}
In order to make direct use of the local to global nature of the counting problem we work with the global \emph{conductor zeta function} $Z(s)$ of $T,r$, that is we define $$Z(s) := \int_{\cA(T)} \frac{1}{\fc(\chi,r)^s}\,d\nu(\chi).$$ 
For $c>0$ let $\mathcal{R}=\mathcal{R}(c)=\mathcal{R}(A,c)\subset \C$ be the region \begin{equation}\label{calRdef}\mathcal{R}=\mathcal{R}(c)= \mathcal{R}(A,c)= \Big\{ \sigma+it \in \C : \sigma > A- \frac{c}{(\log (|t|+3))^{2/3}(\log \log (|t|+16))^{1/3}}\Big\}.\end{equation} The main goal of this paper is to prove the following theorem.
\begin{theorem}\label{Zsthm}
Suppose that $r\vert_{\widehat{T}}$ is faithful. The generating series $Z(s)$ converges absolutely for $\real(s)>A$ and extends to a meromorphic function in the open half plane $\real(s)>A-\min(2^{-1}, m^{-2})$. There exists $c=c(T,r)>0$ such that the function $Z(s)$\begin{itemize} \item has a pole at $s=A$ of order $|\widetilde{G}\backslash \widetilde{\Sigma}_0|$ and no other poles in $\mathcal{R}(c)$ (respectively, the half-plane $\real(s)>A-\min(2^{-1}, m^{-2})$ if the  Artin conjecture holds), \item grows slowly in $\mathcal{R}$; i.e.\ there exists $J=J(T,r)>0$ and $0<c'=c'(T,r)\leq c$ such that for any $s=\sigma+it \in \mathcal{R}(c')$ avoiding any small neighborhood $U$ of $A$ we have $$Z(\sigma + it) \ll_{T,r,U} (\log (|t|+3))^J,$$ 
and \item has moderate growth in a vertical strip if the  Artin conjecture holds, i.e.\ there exists $K=K(T,r)>0$ such that for any $s=\sigma+it$ with $\sigma > A-\min(2^{-1},m^{-2})$ avoiding any small neighborhood $U$ of $A$ we have $Z(\sigma+it) \ll_{T,r,\sigma, U} (1+|t|)^K.$ 
\end{itemize}
\end{theorem}
 Here and throughout the paper, when we ask that the Artin conjecture holds, we mean for the finitely many representations appearing in Theorem \ref{unrthmwithArtin}. Theorems \ref{MT1} and \ref{MT2} follow from Theorem \ref{Zsthm} by an appropriate Tauberian argument, see \cite[Ch.\ III \S 11]{InghamPrimeNumbers} or \cite[Thm.\ A.1]{CLTsch}. 

An application of Poisson summation (Lemma \ref{Psum}) decomposes $Z(s)$ as a sum over global units $x \in U(T)$ of generating series $Z(s,x)$ that factor over places $v$ of $k$, i.e.\  \begin{equation}\label{ZZsx} Z(s) = \sum_{x  \in U(T)}Z(s,x) \quad \text{ with } \quad Z(s,x) = \prod_{v\mid \infty} A_{k_v}(s,x) \prod_{v \nmid \infty} N_{k_v}(s,x)\end{equation}
for certain local archimedean and non-archimedean generating series $A_{k_v}= A_{k_v}(s,x)$ and $N_{k_v}=N_{k_v}(s,x)$. See Proposition \ref{localtoglobal} for the precise statement, which is a minor modification of \eqref{ZZsx}. 

The location and order of the rightmost pole of each $Z(s,x)$ only depends on all but finitely many of the $N_{k_v}(s,x)$, in particular, only on those $v<\infty$ which are unramified in the extension $K/k$ splitting $T$. A main idea of this paper is to perform the analysis of $A_{k_v}$ and $N_{k_v}$ on the Galois side of the local Langlands correspondence for tori \cite{LanglandsAbelian}, which is particularly simple when the torus splits over an unramified extension of non-archimedean local fields, see Proposition \ref{PairingProp}.  

An outline of the argument in this paper is then as follows. 
\begin{enumerate}
\item\label{outline1} Show for each place $v$ of $k$ that the local series $N_{k_v}$ and $A_{k_v}$ converge absolutely for $\real(s)> A-m^{-2}$, see sections \ref{ramified} and \ref{archimedean}. 
\item\label{outline2} Compute $N_{k_v}$ for $v$ unramified finely enough to obtain a group-theoretic description of its leading terms as a local Dirichlet series, see section \ref{unramified}.
\item\label{outline3} Compare the product over unramified places $\prod_{v \not \in B} N_{k_v}$ with a finite product of global Artin $L$-functions and apply the Brauer induction theorem or Artin conjecture, see section \ref{unramifiedcounting}.
\item\label{outline4} Show that the leading Laurent series coefficients of $Z(s,x)$ at $s=A$ are positive for all $x \in U(T)$, see Theorem \ref{archthm}\eqref{archthm3}, Theorem \ref{unrthm}\eqref{unrthm2}\eqref{unrthm3}, and Lemma \ref{RamLemma}. 
\item\label{outline5} Show that the sum over $U(T)$ in \eqref{ZZsx} converges absolutely, see section \ref{conclusion}.
\end{enumerate}
As previously remarked, the Brascamp-Lieb inequality (Theorem \ref{BL}) enters the picture in step \eqref{outline1} for archimedean places, and leads to a problem in combinatorial optimization of convex polyhedra. To resolve this problem, we use the theory of matroids, see section \ref{matroidbackground}, especially Theorem \ref{polyinter}. 

Step \eqref{outline2} is the heart of the paper. Here we use the detailed conductor analysis from section \ref{sec:non-arch_background}, group theory, some algebraic geometry, and Lang-Weil bounds. The group theoretic description of the unramified terms thus obtained is crucial in their collection into Artin $L$-functions in step \eqref{outline3}.

The most difficult part of step \eqref{outline4} again turns out to be the archimedean places, see section \ref{positivity}.  It is important for our method that the function $x \mapsto (1+|x|)^{-\sigma}$ on $\R$ (among others) has a non-negative Fourier transform, where the $1+|x|$ here arises from the Iwaniec-Sarnak definition of the archimedean analytic conductor, see Lemmas \ref{positivein1dim}, \ref{positivein1dim2} and \ref{positivein1dim3}. 

The Brascamp-Lieb inequality is used a second time in step \eqref{outline5} of the proof, where it is applied in a global context with respect to the regulator matrix of $T$. 

\subsection{Index of notation}
\begin{center}
 \begin{tabular}{@{} | p{3cm} |  p{9.7cm}  |  p{2.7cm} |} 
\hline
Notation & Definition & Location \\
\hline\hline
$k$ & a general field in \S \ref{background}, a number field in \S \ref{intro}, \ref{sec:localtoglobal} and \ref{final} &  \S \ref{statementofresults}, \ref{background:tori}\\
$\A, F$ & the ring of ad\`eles of a number field $k$, a local field & \S \ref{statementofresults}, \ref{sec:LLC}\\
$T$, $n$ & an algebraic torus over $k$ or $F$ of dimension $n$  & \S \ref{statementofresults}, \S \ref{background:tori}\\
$\cA(T), \nu$ & the Pontryagin dual of $T(k) \backslash T(\A)$, a Haar measure on it & \S \ref{statementofresults}\\
$X^*(T), X_*(T)$ & groups of algebraic characters and coharacters of $T$ & \S\ref{statementofresults}, \S \ref{background:tori}\\
$K$, $L$, $G$ & the splitting field of $T$ over $k$, or over $F$, its Galois group & \S\ref{statementofresults}, \S \ref{background:tori}\\
$\widehat{T}$ & the complex dual torus: $\Hom(X_*(T),\C^\times)$ &\S\ref{statementofresults}, \eqref{Thatdef}\\
$\LT$ & the $L$-group of $T$: $\widehat{T} \rtimes G$ & \S\ref{statementofresults}, Def \ref{LTdef}\\
$r, m$ & an $m$-dimensional complex algebraic representation of $\LT$ & \S\ref{statementofresults}, \S \ref{background:tori}\\
\hline
\end{tabular}

 \begin{tabular}{@{} | p{3cm} |  p{9.7cm}  |  p{2.7cm} |} 
\hline
Notation & Definition & Location \\
\hline\hline
$\fc(\chi,r)$ & the (local or global) analytic conductor with respect to $r$ & Defs \ref{artinc1}, \ref{archdef}, \ref{GlobalAC}\\
$\mu, M$ & a co-weight of $r$, the multi-set of co-weights of $r$ & \eqref{rT}, Def \ref{def_coweights}\\
$S, 2^M$ & a subset of $M$,  the set of all subsets of $M$ & \S \ref{statementofresults}, \S \ref{sec:Lgroupsandreps}\\
$D(S)$ & a complex diagonalizable subgroup of $\widehat{T}$ & \eqref{DS}, \S \ref{sec:Lgroupsandreps}\\
$A$ & the power of $X$ in the main theorem & \eqref{Adef}\\
$\Sigma$ & set of nonzero $S\subseteq M$ which attain $A$ & \eqref{Sigma}\\
$\lambda$ & $\lcm_{S \in \Sigma} |\pi_0(D(S))|$ & \eqref{lambda}\\
$\widetilde{K}$, $\widetilde{G}$ & $K$ adjoin the $\lambda$th roots of unity, $\Gal(\widetilde{K}/k)$ & \S \ref{statementofresults}\\
$\widetilde{\Sigma}$ & a fibered set with base $\Sigma$ and fiber $\pi_0(D(S))$ & \eqref{Sigmatilde}\\
$\widetilde{\Sigma}_0$ & $\widetilde{\Sigma}$ with the subset $\{(S,1): \dim D(S)=0\}$ deleted & \eqref{Sigmatilde0}\\
$C(\chi)$ & the analytic conductor of a Hecke character $\chi$ & \S \ref{examples}\\
$\res_{K/k}$ & restriction of scalars, i.e.\ ``Weil restriction'' &  \S \ref{examples}\\
$\mathcal{R}(c)=\mathcal{R}(A,c)$ & a region of analytic continuation for $Z(s)$, $Y(s)$ or $U(s,x)$ & \eqref{calRdef}\\
$k^{\rm sep}$, $G_k$ & a separable closure of $k$, the absolute Galois group of $k$ & \S \ref{background:tori}\\
$X_1 \times_{S} X_2$ & the fiber product of schemes $X_1,X_2$ over a base $S$ & \S \ref{background:tori}\\
$A^\wedge$ & the Pontryagin dual of a locally compact abelian group $A$ & \S \ref{sec:LLC}\\
$W_F$, $W_{L/F}$ & Weil group and relative Weil group of a local field & \S \ref{sec:LLC}, \eqref{relative_weil_group}\\
$\sigma$ & the canonical map $W_{L/F} \to \Gal(L/F)$ & \eqref{WLFes}\\
$\xi$, $\varphi$ & a cohomology class and a Langlands parameter $\varphi = \xi \rtimes \sigma$ & \S \ref{sec:LLC}\\
$\Phi(T)$ & the set of Langlands parameters of $T$ &  \S \ref{sec:LLC}\\
$H_M$ & a polytope in $\R^m_{\geq 0}$ given by an $m\times n$ matrix $M$ & \eqref{CLLcondition1}, \eqref{CLLcondition2}\\
$q_F$ & the cardinality of the residue field of $F$ & \S \ref{ArtinC}\\
$c(\rho)$, $\tilde{c}(\rho)$ & the Artin and abelian conductors of a representation $\rho$ & Def \ref{propdef}, \ref{modifiedArtin}\\
$c_\cF(\rho)$ & the conductor of a representation $\rho$ wrt a filtration $\cF$ & Def \ref{filtcond}\\
$\tilde{\fc}(\chi,r)$ & ``abelian'' local analytic conductor & Def \ref{modifiedAC}\\
$\cO_F$, $\cO_L$, $\fp$, $\fP$, $f$, $\ell$ & integers, maximal ideals, and residue fields of $F$ and $L$ & \S \ref{norm}\\
$G^v$, $W_{L/F}^v$ & higher ramification groups with upper-numbering & \S \ref{ArtinC}\\
$\mathcal{U}$, $\cO_L^{(v)}$ & the standard filtration on $\cO_L^\times$ & \S \ref{ArtinC}\\
$\psi_{E/F},\phi_{E/F}$ & the Hasse-Herbrand functions, see \cite[Ch.\ IV \S 3]{SerreLF}& \S \ref{ArtinC}\\
$\cT$ & the canonical integral model of a torus $T$ over a local field & \S \ref{norm}\\
$N$ & the norm map, i.e. the product of Galois conjugates & \eqref{eq:norm}\\
$\widehat{H}^n(G,M)$ & Tate cohomology groups of a $G$-module $M$ & \S \ref{norm}\\
$R$ & the restriction to $\cO_L^\times$ map out of $H^1(W_{L/F},\widehat{T})$ & \eqref{eq:R}\\
$H_1(G,M)$ & $1$st group homology group of a $G$-module $M$ & \eqref{H_1def}\\ 
$N_F(s,x)$ & local generating series for a non-archimedian field & \eqref{NFsxdef}\\
$\N$ & the non-negative integers & \S \ref{unramified}\\
$P_{\leq}(c)$ & a finite subgroup of $\Hom_G(\cO_L^\times,\widehat{T})$ & \eqref{Pleqc}\\
$P_{=}(c)$ & a ``sharp'' subset of $P_\leq (c)$ & \eqref{P=c}\\
$\Pi_\leq(c,x)$, $\Pi_=(c,x)$ & character sums over $P_\leq(c)$ and $P_=(c)$ & \eqref{PileqPi=}\\
$D_k(c)$ & a generalization of $D(S)$ & \eqref{Dkc}\\
$p(V)$ & the set of geometric components of a variety $V$ & \S \ref{unramified}\\
$a(S,x)$ & the number of Frobenius-fixed components of $\alpha^{-1}(x)$ & \eqref{aSx}\\
$a(S)$ & $a(S,1)$; the number of Frobenius-fixed points of $\pi_0(D(S))$ & \eqref{aS}\\
$S_{\rm red}$ & the maximal Galois-stable subset of $S$ & \eqref{Sred}\\
$\Tt$, $n_1$, $n_2$, $n_3$ & $T(F) \simeq \Tt= (\R^\times)^{n_1}\times (S^1)^{n_2} \times (\C^\times)^{n_3}$, $F$ arch.\ local & \eqref{TRisom}\\
$\Tt^\wedge$ & $\Tt^\wedge = (i\R^{n_1} \times (\Z/2\Z)^{n_1}) \times \Z^{n_2} \times (i\R^{n_3} \times \Z^{n_3})$ & \eqref{Ttwedgedef}\\
$((w,\epsilon), \alpha, (w',\alpha'))$ & a typical element of $\Tt^\wedge$ & \S \ref{archimedeanC}\\
\hline
\end{tabular}

 \begin{tabular}{@{} | p{3cm} |  p{9.7cm}  |  p{2.7cm} |} 
\hline
Notation & Definition & Location \\
\hline\hline
$(a,c,(b,b'))$ & an element of $X_*(\G_m)^{n_1} \times X_*(S^1)^{n_2} \times X_*(\res_{\C/\R} \G_m)^{n_3}$ & \S \ref{archimedeanC}\\
$M$ & a matrix with entries in $\Z$ encoding the co-weights of $r$ & \S \ref{archimedeanC}\\
$A_i, C, B_i,B_3^\pm$ & sub-block matrices of the co-weight matrix $M$ & \S \ref{archimedeanC}\\
$A_F(s,x)$ & local generating series for an archimedian field & \eqref{AFdef}\\
$B_\infty$, $B_{\infty,1/2}$ & $\inf\{\|x \|_\infty: x \in H_M\}$, a variant involving a factor of $1/2$ & \eqref{B0def}, \eqref{B0halfdef}\\
$(N,\cI)$, $r(S)$  & a matroid, its rank function & Def \ref{def:matroid}, \ref{def:rank}\\
$P_\cI$, $P_\cB$ & the matroid polytope and matroid base polytope & Def \ref{def:polytopes}\\
$\cF_G(f)$ & Fourier transform of a measure $f$ on an abelian group $G$ & \S \ref{positivity} \\
$v$, $w$ & a valuation of $k$ and a unique valuation of $K$ extending it & \S \ref{sec:globalconductor}\\
$T_\infty$ & $ \prod_{v\mid \infty} T(k_v)$ & \S \ref{locglobandconclusion} \\ 
$NT_f$ & $\prod_{v \nmid \infty} NT(\cO_w)$ & \S \ref{locglobandconclusion}\\
$T_{N,\A}$ & $T_\infty \times NT_f$ & \eqref{TNAdef} \\
$U_N(T)$ & $T(k) \cap T_{N,\A}$, called the global norm-units of $T$ & \eqref{UNTdef}\\
$\cl_N(T)$ & $T(k)T_{N,\A} \backslash T(\A)$, called the norm-class group of $T$ & \eqref{beginninges}\\
$V^\wedge$ & $ \{\chi \in T_{N,\A}^\wedge: \chi(x) = 1 \text{ for all } x \in U_N(T)\}$ & \eqref{Vwedge}\\ 
$V_\infty^\wedge$ & $\{\chi_\infty \in T_\infty^\wedge: \chi_\infty(x) = 1 \text{ for all } x \in U_N(T)\}$ & \eqref{Vinftywedge}\\ 
$B$ & set of places of $k$ with $(q_{k_v},\lambda)\neq 1$ or $v$ ramified & \S \ref{subsec:localtoglobal}\\
$T_S$ & an auxiliary torus attached to $S \in 2^M$ & \S \ref{unramified}, \ref{unramifiedcounting}\\
$\alpha^{-1}(x)$ & fibers of a map of tori $\alpha:T_S \to T$, see also Lem.\ \ref{agbnd} & \S \ref{unramified}, \ref{unramifiedcounting}\\
$K'$, $L'$, $\Gamma$ & field of def.\ of components of $\alpha^{-1}(x)$ for all $x$, $\Gal(K'/k)$ & \S \ref{unramifiedcounting}\\
$\fP'$, $D_\fP$, $D_{\fP'}$ & a prime of $K'$ above $\fP$, decomposition groups of $\fP$, $\fP'$ & \S \ref{unramifiedcounting}\\
$\approx$ & equal up to an absolutely convergent Euler product & Def \ref{approxdef}\\
$C$ & a conjugacy class of $\Gamma$ & \eqref{splitovercc}\\
${\Sigma^\fP}$ & subset of $\Sigma$ fixed by $D_\fP$ & \eqref{SigmaP}\\
$a_C(S,x)$ & the number of $C$-fixed components of $\alpha^{-1}(x)$ & \eqref{aCSx}\\
$\Sigma_{a,b}$ & set of $S \in 2^M$ such that $\dim D(S)=a$ and $|S|=b$ & \eqref{Sigmaab}\\
$\widetilde{\Sigma}_{a,b}$ & similar to $\widetilde{\Sigma}_0$, but with respect to $\Sigma_{a,b}$ & \eqref{Sigmaabtilde}\\
$V_\cO= \oplus_i V_i^{\oplus m_{\cO,i}}$ & permutation rep.\ of an orbit $\cO$ of $\Gamma$ acting on $\widetilde{\Sigma}_{a,b}$ & \eqref{Videf}\\
\hline
\end{tabular}
\end{center}

\subsection{Acknowledgements}
This paper was began as joint work with Hunter Brooks, who, soon after it was started, left academic mathematics and went into industry. I learned a great deal of mathematics from Hunter, especially the background for section \ref{sec:non-arch_background} of this paper. I would also like to express my gratitude to Will Sawin. This paper grew out of discussions of Hunter, Will and I in Z\"urich in fall 2016, and would not have been possible without Will's generosity with his ideas. Lastly, thanks are due to Rico Zenklusen, who pointed out to me that the polytope minimization problem in section \ref{archimedean} was best studied in the language of matroids, and introduced me to the polymatroid intersection theorem which plays a key role in that section. I would like to thank Farrell Brumley for interesting discussions on the automorphic counting problem, Emmanuel Kowalski for a careful reading of an earlier draft of this paper, Peter Sarnak for his encouragement, and the anonymous referee whose extraordinarily detailed report greatly improved the paper.  I am grateful for the financial support of the Swiss National Science Foundation, and would like to thank the people of Switzerland for their commitment to funding fundamental research.

\section{Background and notation}\label{background}
\subsection{Tori and groups of multiplicative type over a field}\label{background:tori}
In this subsection we let $k$ denote an arbitrary field. We take an algebraic $k$-group to be as in \cite[Def.\ 1.1]{MilneAGS}. 
\begin{definition}
An algebraic $k$-group $T$ is called a \emph{torus} if there exists a field $K \supseteq k$ such that the base change $T \times_k \Spec K$ of $T$ is isomorphic to a finite product of copies of $\G_m$, i.e.\ $T \times_k \Spec K \simeq \G_{m,K}^n$ for some non-negative integer $n$.
\end{definition}

\begin{definition} An algebraic $k$-group is called \emph{diagonalizable} if it is isomorphic over $k$ to a finite product of copies of $\G_m$ and groups of roots of unity $\mu_r$. More generally, an algebraic $k$-group $U$ is said to be \emph{of multiplicative type} if there exists a field $K \supseteq k$ such that the base change $U \times_k \Spec K$ is isomorphic to a diagonalizable group over $K$. \end{definition} Tori are the smooth connected groups of multiplicative type \cite[\S 12.f]{MilneAGS}. If a field $K \supseteq k$ is such that $U \times_k \Spec K$ is diagonalizable over $K$, then we say that $U$ \emph{splits} over $K$. In fact, any $k$-group of multiplicative type splits over a finite separable extension of $k$ \cite[Cor.\ 12.19]{MilneAGS} and we call the minimal Galois extension of $k$ over which $U$ splits \emph{the splitting field} of $U$. 

Let $k^{\rm sep}$ be a separable closure of $k$ and $G_{k'}= \Gal(k^{\rm sep}/k')$ for any $k \subseteq k' \subseteq k^{\rm sep}$. 
For any group of multiplicative type $U$ over $k$, let $X^*(U) = \Hom(U,\G_m)$ be the \emph{group of algebraic characters} of $U$ and $X_*(U) = \Hom(\G_m,U)$ be \emph{group of algebraic cocharacters} of $U$. They are finitely generated abelian groups admitting continuous actions of $G_k$. These actions on $X^*(U)$ and $X_*(U)$ factor through the action of the finite group $G= \Gal(K/k)$, where $K$ is the splitting field of $U$. A group of multiplicative type $U$ is an affine scheme with coordinate ring $k^{\rm sep}[X^*(U)]^{G_k} = K[X^*(U)]^G$. 
\begin{lemma}\label{equiv_of_cat_multi}
The functor $X^*$ is a contravariant equivalence of categories from the category of algebraic $k$-groups of multiplicative type to the category of finitely generated abelian groups equipped with a continuous action of $G_k$. The functor $X^*$ is exact, i.e.\ it sends short exact sequences to short exact sequences. 
\end{lemma}
\begin{proof} See \cite[Thm.\ 12.23]{MilneAGS}. \end{proof}
The equivalence of categories from Lemma \ref{equiv_of_cat_multi} given by the exact functor $X^*$ restricts to an equivalence of categories from the category of $k$-tori to the category of finitely-generated free $\Z$-modules equipped with a continuous action of $G_k$, see \cite[Rem.\ 12.5]{MilneAGS}.  

For any $k \subseteq k' \subseteq k^{\rm sep}$ and any $k$-group $U$ of multiplicative type we have 
\begin{equation}\label{stdisom1}
U(k') \simeq \Hom_{G_{k'}} (X^*(U), (k^{\rm sep})^\times),
\end{equation}
see \cite[Rem.\ 12.26]{MilneAGS}. The map in \eqref{stdisom1} takes a continuous $G_{k'}$-equivariant homomorphism $\ell: X^*(U) \to (k^{\rm sep})^\times$ and extends it to a $k$-algebra homomorphism $u^*: \cO(U) = (k^{\rm sep})[X^*(U)]^{G_{k'}} \to k'$, which defines the $k'$-point $u: \Spec k' \to U$ of $U$. 

 Let $T$ be a $k$-torus. The \emph{evaluation pairing} \begin{equation}\label{char_cochar_PP}\langle \cdot,\cdot \rangle: X^*(T) \otimes_{\Z} X_*(T) \to \Z\end{equation} given by $\chi \circ \lambda : z \mapsto z^{\langle \chi,\lambda \rangle}$ is a perfect pairing between the character and cocharacter lattices. 

The perfect pairing \eqref{char_cochar_PP} gives us another description for the $k$-rational points of a torus. Indeed, we have $X_*(T)\simeq \Hom(X^*(T),\Z)$ and so $X_*(T) \otimes_\Z K^\times \simeq \Hom(X^*(T),\Z) \otimes_\Z K^\times$. 
  There is an isomorphism  
\begin{equation}\label{stdisom2}
\phi: \Hom(X^*(T),\Z) \otimes_\Z K^\times \simeq \Hom(X^*(T),K^\times)
\end{equation}
given on pure tensors by $\phi(\lambda \otimes z) = (\chi \mapsto z^{\lambda(\chi)})$ for $\chi \in X^*(T)$.  (More generally, $\Hom(P,R) \otimes M \simeq \Hom(P,M)$ for any $R$-module $M$ and finitely-generated projective $R$-module $P$.) Combining these maps with \eqref{stdisom1} we have 
 \begin{equation}\label{rational_points_as_tensor} T(K) \simeq X_*(T) \otimes_\Z K^\times \quad \text{ and } \quad T(k) \simeq (X_*(T) \otimes_\Z K^\times)^G.\end{equation}

\subsection{$L$-groups of tori and representations}\label{sec:Lgroupsandreps}
For any $k$-torus $T$ the group \begin{equation}\label{Thatdef}
\widehat{T} := \Hom(X_*(T),\C^\times)\simeq X^*(T)\otimes_\Z \C^\times
\end{equation} is called the \emph{complex dual torus} of $T$. As a group, $\widehat{T} \simeq (\C^\times)^{n}$ and carries an action of $G=\Gal(K/k)$ through the Galois action on the cocharacter lattice $X_*(T)$. We shall also use the unit complex dual torus $$\widehat{T}_u = \Hom(X_*(T),S^1) \simeq X^*(T)\otimes_\Z S^1 .$$  

The affine $k$-scheme $T^\vee = \Spec K[X_*(T)]^{G}$ is called the \emph{algebraic dual torus} of $T$. There is a natural isomorphism $X_*(T) \simeq X^*(T^\vee)$ sending $\lambda \in X_*(T)$ to the $\chi \in X^*(T^\vee)$ defined by $\chi^*= (X \mapsto \lambda)$ on coordinate rings $\chi^{*} : k^{\rm sep}[X,X^{-1}] \to k^{\rm sep}[X_*(T)]$ (see e.g.\ \cite[\S 12.a, Lem.\ 12.4]{MilneAGS}). 
If $k$ is a subfield of $\C$, then $\widehat{T} = T^\vee(\C)$ and $$X^*(\widehat{T}) := \Hom_{\rm cts}(\widehat{T},\C^\times) = \Hom(T^\vee, \G_m) = X^*(T^\vee),$$ so in this case we obtain a natural identification $X^*(\widehat{T}) \simeq X_*(T)$.

\begin{definition}\label{LTdef}
Let $T$ be a $k$-torus with splitting field $K$. We call the external semi-direct product 
$ \LT= \widehat{T} \rtimes G$ the $L$\emph{-group} of $T$, where $G=\Gal(K/k)$.\
\end{definition} The $L$-group of $T$ is a complex algebraic group. We call the subgroup $\LT_u :=  \widehat{T}_u \rtimes G$ of $\LT$ the \emph{unit $L$-group} of $T$. Caution: more commonly in the literature on the Langlands correspondence the $L$-group is defined using the absolute Galois group $G_k$ in lieu of the finite group $G$. While these two definitions are ultimately equivalent, we work with finite Galois groups mainly because Langlands does in his paper on the correspondence for tori \cite{LanglandsAbelian}, and some computations in group co-/homology become simpler when we work with finite groups. Of course, the cost of working with finite $G$ is having to keep track of the splitting field of $T$.   

Let $r: \LT \to \GL(V)$ be a finite-dimensional complex algebraic representation of $\LT$. The restriction of $r$ to $\widehat{T}$ admits a weight space decomposition 
\begin{equation}\label{rT2}
r \vert_{\widehat{T}} = \bigoplus_{\mu \in X^*(\widehat{T})} V_\mu,
\end{equation}
where $V_\mu$ is the eigenspace of $V$ with character $\mu$. 
\begin{definition}\label{def_coweights}
The multi-set $M= M_r$ with underlying set $\{\mu \in X^*(\widehat{T}): V_\mu \neq 0\}$ and multiplicity of $\mu\in M$ equal to $\dim V_\mu$ is called the set of \emph{co-weights} of $r$. 
\end{definition}
Let $S\subseteq M$ be a subset of co-weights with multiplicity. Recall the definitions of $D(S)$ and $A$ from \eqref{DS} and \eqref{Adef}, which make sense for general base fields. 

The group $G$ acts on the set of co-weights $M$ via its action on $X_*(T)$ (or $\widehat{T})$. In the case that $k$ is an archimedean local field, we will, after choosing coordinates on $T(k)$ and $r$, associate to $M$ an $m \times n$ matrix (where $m=\dim r$ and $n=\dim T$), which we also write $M$. 

In this paper a $k$-variety is a reduced, separated $k$-scheme of finite type. In particular, we do not assume that varieties are irreducible. 

 \begin{lemma}\label{agbnd}
Let $\alpha:T_1\to T_2$ be a map of tori over a field $k$ of characteristic 0. The number of geometric components of the fiber $\alpha^{-1}(x)$ is constant on $\{x \in T_2(k): \alpha^{-1}(x) \text{ is non-empty}\}$. For any finitely-generated subgroup $\Lambda$ of $T_2(k)$, every component of $\alpha^{-1}(x)$ for all $x \in \Lambda$ is defined over a single finite extension of $k$. 
\end{lemma}
\begin{proof}  
For the first assertion, if $\alpha^{-1}(x)$ is empty, there is nothing to show, so suppose otherwise.

Recall \cite[Rem.\ 12.5]{MilneAGS} that the tori $T_1$, $T_2$ as well as the group of multiplicative type $\ker \alpha$ are all reduced $k$-schemes, since $k$ has characteristic $0$. If $k$ is a perfect field and $A$ and $B$ are reduced $k$-algebras, then $A \otimes_k B$ is a reduced $k$-algebra, see \cite[Ch.\ V \S 15 5.\ Thm.\ 3(c) and 2.\ Prop.\ 5]{BourbakiAlgebra}. Since $T_1$ and $T_2$ are affine, it follows that $\alpha^{-1}(x)$ is reduced, hence a closed subvariety of $T_1$. As we have already remarked, the algebraic group $\ker \alpha$ is reduced, thus $y (\ker \alpha)$ is also a closed subvariety of $T_1$ for any fixed geometric point $y \in \alpha^{-1}(x)(K)$ for any finite extension $K/k$. 

For any fixed algebraic closure $\overline{k}/k$ it is easy to check that $\alpha^{-1}(\overline{k}) = y (\ker \alpha) (\overline{k})$, so that by e.g.\ \cite[Cor.\ 1.18]{MilneAGS}, $\alpha^{-1}(x) = y(\ker \alpha)$ as closed subschemes of $T_1$. Then, since $y (\ker \alpha) \simeq_K \ker \alpha$ (see e.g.\ \cite[Prop.\ 5.24]{MilneAGS}) and the formation of the group of connected (equivalently, irreducible) components $\pi_0$ commutes with base change \cite[Prop.\ 2.37(c)]{MilneAGS}, we have that the number of components of $\alpha^{-1}(x)$ is independent of $x$. 

For the second assertion, say $x_1, \ldots, x_r$ are generators of $\Lambda$. Since $T_1$ is abelian, all of the irreducible components of $\alpha^{-1}(x)$ for $x \in \Lambda$ are defined over the finite extension of $k$ obtained by adjoining the coordinates of the $y_i$ corresponding to $x_i$ (if they exist) from the previous paragraph to the field of definition of $(\ker \alpha)^\circ$. 
\end{proof}
 
\subsection{Local Langlands correspondence}\label{sec:LLC}
For a topological abelian group $A$, we henceforth denote by $\Hom(A,\C^\times)$ the group of continuous complex characters of $A$, and by $A^\wedge$ the subgroup of \emph{unitary} characters, that is to say the \emph{Pontryagin dual}. For $M$ a $G$-module, $H^1(G,M)$ denotes the first group cohomology group defined using continuous cocycles. 

In this section we suppose that $T$ is a torus over a local field $F$ with splitting field $L$ and Galois group $G= \Gal(L/F)$. 

Following \cite{BorelAutomorphicLfcns}, we define the local analytic conductors $\fc(\chi,r)$ associated to a character $\chi:T(F) \to \C^\times$ and representation $r$ of $\LT$ by passing through the local Langlands correspondence and taking the conductors from the Galois representation associated to $\chi$ and $r$. To that end, we now review the local Langlands correspondence for tori.  

Recall the Weil group of a local field \cite[\S1.1]{tateNTB}, which is a triple $(W_F, \varphi, \{r_E\})$, and the relative Weil group \begin{equation}\label{relative_weil_group}W_{L/F} := \frac{W_F}{[W_L,W_L]}. \end{equation} The group $W_{L/F}$ has $W_L^\text{ab}$ as a subgroup and so can be thought of as a group extension of $G$ by $L^\times$, i.e. there is a short exact sequence \begin{equation}\label{WLFes}\xymatrix{1 \ar[r] & L^\times \ar[r]^-{r_L}& W_{L/F} \ar[r]^-\sigma & G \ar[r] & 1},\end{equation} see \cite[\S1.2]{tateNTB}. If $L,F$ are non-archimedean local fields, then the map $r_L$ is the Artin reciprocity map of class field theory \cite[(1.4.1)]{tateNTB}.

Following \cite{LanglandsAbelian}, we define a \emph{Langlands parameter} to be a continuous group homomorphism $\varphi: W_{L/F} \to \LT$ for which the diagram $$\xymatrix{W_{L/F} \ar[r]^\varphi \ar[dr]_\sigma & \LT \ar[d] \\ & G}$$ is commutative. Two Langlands parameters are said to be equivalent if they are $\widehat{T}$-conjugate. We write $\Phi(T)$ for the set of equivalence classes of Langlands parameters of $T$ as in \cite[\S 8]{BorelAutomorphicLfcns}. The local Langlands correspondence (LLC) for tori asserts that there is a canonical bijection 
\begin{equation}\label{LLCforTori} \Hom(T(F),\C^\times) \longleftrightarrow \Phi(T).\end{equation} 
Given a Langlands parameter $\varphi$ corresponding to $\chi \in \Hom(T(F),\C^\times)$, the composition $$r \circ \varphi: W_{L/F} \to \GL(V)$$ only depends on the equivalence class of $\varphi$. We have thus associated a complex Galois representation to the character $\chi$ and the $L$-group representation $r$. We define the local $L$ and $\eps$-factors associated to finite dimensional complex Galois representations as in \cite[\S3]{tateNTB}. Later, we will define local analytic conductor $\fc(\chi,r)$ in terms of the $\eps$-factor of $r \circ \varphi$ if $F$ is non-archimedean (see Definition \ref{artinc1}), and in terms of the $L$-factor of $r \circ \varphi$ if $F$ is archimedean (see Definition \ref{archdef}).

To make \eqref{LLCforTori} more explicit, we recall the the cohomological interpretation of the LLC for tori. Given a Langlands parameter $\varphi$, we write $\varphi(z) = \xi(z) \rtimes \sigma(z)$ for $z \in W_{L/F}$, $\xi(z) \in \widehat{T}$, and $\sigma(z) \in G$. One sees that $\varphi$ and $\varphi'$ are equivalent if and only if $\xi$ and $\xi'$ are cohomologous, i.e.\ we have a bijection 
\begin{equation}\label{bij_LLP_to_coh}\Phi(T) \longleftrightarrow H^1(W_{L/F}, \widehat{T}),\end{equation} where $W_{L/F}$ acts on $\widehat{T}$ via the map $\sigma:W_{L/F} \to G$ of \eqref{WLFes}. Langlands proved \cite[Thm.\ 1]{LanglandsAbelian} that there is an isomorphism \begin{equation}\label{langlands} \Hom(T(F), \C^\times) \simeq H^1(W_{L/F},\widehat{T})\end{equation} and moreover  \eqref{langlands} restricts to \begin{equation}\label{langlandsunitary}T(F)^\wedge \simeq H^1(W_{L/F},\widehat{T}_u),\end{equation} where $\widehat{T}_u = X_*(T)^\wedge$. We will use \eqref{langlandsunitary} when $F$ is an archimedean local field. 

Let $dx$ be a Haar measure on $F$, $\psi$ a non-trivial additive character of $F$, and $dx'$ the dual Haar measure relative to $\psi$.  Given a finite dimensional complex representation $(\rho,V)$ of $W_F$, Tate \cite[\S 3]{tateNTB} defines the $\eps$-factor $\eps(V,\psi,dx)= \eps(\rho,\psi,dx)$ attached to these data. When we give the definition of the local analytic conductors in sections \ref{sec:non-arch_background} and \ref{archimedeanC}, we will encounter the factor $(\delta(\psi) dx / dx')^{\dim (V)}$. This factor is explained in \cite[\S 3.4]{tateNTB} and we do not need to elaborate on it for the purposes of this paper.

\subsection{Some tools}\label{tools}

Let $\overline{f}$ denote an algebraic closure of a finite field $f$, and let $V \subseteq \overline{f}^n$ be a variety over $\overline{f}$ of dimension $r$ and degree $d$. The following is \cite[Thm.\ 1]{LangWeil}. \begin{theorem}[Lang-Weil]\label{LWthm} If $V$ is defined over $f$ and irreducible as a variety over $\overline{f}$, then 
$$|V(f)| = |f|^{r} + O_{n, d,r}(|f|^{r-1/2}).$$
\end{theorem} 
In fact, Lang and Weil give a more explicit bound on the implied constant in their Theorem 1, but we do not need this. We also have the following result of Lang and Weil under weaker hypotheses \cite[Lem.\ 1]{LangWeil}. 
\begin{lemma}[Lang-Weil]\label{LWlemma}
If $V$ is defined over $\overline{f}$, then 
$$ |V(f)| \ll_{n,d,r} |f|^{r}.$$
\end{lemma}
We need the following standard variant of the Lang-Weil bound that relaxes the geometric irreducibility, definability over $f$, and reducedness hypotheses.   
\begin{corollary}[Lang-Weil, alternate form]\label{LWcor}
Let $V$ be a separated $\overline{f}$-scheme of finite type. Then one has $$ |V(f)| = (p(V) + O_{n,d,r}(|f|^{-1/2})) |f|^{r}$$
where $p(V)$ is the number of geometrically irreducible components of $V$ of dimension $r=\dim(V)$ that are invariant with respect to the Frobenius endomorphism $x \mapsto x^{|f|}$ associated to $f$.
\end{corollary}
If in fact the dimension of $V$ is zero, it is not hard to see that $ |V(f)| = p(V).$  
\begin{proof}[Proof sketch following the blog post ``The Lang-Weil bound'' of T.\ Tao.]
We may assume without loss of generality that $V$ is an $\overline{f}$-variety by working with the underlying reduced closed subscheme $V_{\rm red}$, which is an $\overline{f}$-variety and satisfies $V_{\rm red}(f) = V(f)$, since $f$ has no nilpotents.  

Decompose the variety $V$ into geometrically irreducible components $\cup_i V_i$. For any $V_i$ of dimension $<r$ we apply Lemma \ref{LWlemma} to subsume these components into the error term. If $V_i$ is geometrically irreducible but not defined over $f$, then it is not fixed by the Frobenius endomorphism $\Fr$. In this case, $V_i \cap \Fr(V_i)$ is a proper closed subvariety of the irreducible $V_i$, so is of strictly lower dimension. Since all the $f$-points of $V_i$ are contained in $(V_i \cap \Fr(V_i))(f)$, we again use Lemma \ref{LWlemma} to subsume these components into the error term. Lastly, each of the components $V_i$ that remain are geometrically irreducible and defined over $f$, to which we apply Theorem \ref{LWthm} to conclude the proof. 
\end{proof}

 Let $M\in M_{m\times n}(\R)$ be an $m \times n$ matrix with real entries, $m\geq n$. Let us write $a_i,$ $i= 1, \ldots, m$ for the rows of $M$. We define a convex polytope $H_M\subseteq \R_{\geq 0}^m$ by the the following inequalities: \begin{equation}\label{CLLcondition1} \sum_{i=1}^{m} x_i = n\end{equation} and \begin{equation}\label{CLLcondition2}\sum_{i \in S} x_i \leq \dim(\mathrm{span}(\{a_i: i \in S\}))\end{equation} for every subset $S \subseteq \{1,\ldots, m\}$. Note that $H_M$ is non-empty if and only if $M$ is full-rank.
 
The following Brascamp-Lieb inequality is due to Barthe \cite{BartheBLinequality} and was re-stated in the form below by \cite[\S4]{CLLinequality}. We use it in a crucial way in section \ref{analysis} and then again in section \ref{conclusion}.
\begin{theorem}[Brascamp-Lieb Inequality]\label{BL} Let $a_1,\ldots, a_m$ be non-zero vectors in $\R^n$ which span $\R^n$, and let $M$ be the $m \times n$ matrix whose rows are $a_i$.  Let $\overline{p}=(p_1^{-1},\ldots,p_m^{-1})\in\R^m_{\geq 0}.$ Let $\overline{f} = (f_i)_{i=1,\ldots, m}$ be an $m$-tuple of non-negative measurable functions $f_i: \R \to \R_{\geq 0}$. Then $$\int_{\R^n} \prod_{i=1}^m f_i(\langle a_i,x\rangle) \,dx \ll_{m,n,M,\overline{p}} \prod_{i=1}^m \|f_i\|_{L^{p_i}(\R)} $$ if and only if $\overline{p} \in H_M \subset \R_{\geq 0}^m$. Here the implied constant depends on $m,n,M,\overline{p},$ but not on $\overline{f}$. \end{theorem}
 We also have the following convenient version of the Brascamp-Lieb inequality on finitely generated abelian groups due to Bennett, Carbery, Christ and Tao \cite[Thm.~2.4]{BCCTintegral}.  
\begin{theorem}[Discrete Brascamp-Lieb Inequality]\label{BCCT10}
Let $G$ and $\{G_i: 1\leq i \leq m\}$ be finitely generated abelian groups. Let $\varphi_i: G \to G_i$ be homomorphisms. Let $p_i \in [1,\infty]$. Then \begin{equation} \rank(H ) \leq \sum_i p_i^{-1} \rank(\varphi_i(H)) \quad \text{for every subgroup } H \text{ of } G \end{equation}if and only if there exists a constant $C<\infty$ such that 
\begin{equation}
\sum_{y \in G} \prod_{i=1}^m (f_i \circ \varphi_i)(y) \leq C \prod_{i} \| f_i \|_{\ell^{p_i}(G_{i})} \quad \text{ for all } f_i : G_i \to [0,\infty).\end{equation}\end{theorem}

 We next recall some analytic results on $L$-functions of Hecke characters in $t$-aspect.  
 \begin{lemma}\label{ZFHforHeckeL}
 For any Hecke character $\chi$ over a number field there exists an effective constant $c(\chi)>0$ such that $L(s, \chi) \neq 0$ for all $s \in \mathcal{R}(1,c(\chi))$. 
 \end{lemma}
 \begin{proof}
 A more general result of Coleman \cite[Thm.\ 2]{Coleman} asserts that the lemma holds apart from a possible exceptional real zero when the archimedean component of $\chi$ is trivial and $\chi^2=1$. However,  Stark's effective lower bounds on $L(1,\chi)$ \cite[Thm.\ 1']{StarkEffectiveBrauerSiegel} bound such a potential exceptional zero away from $s=1$ in terms of the discriminant of the field of definition of $\chi$, so that by adjusting the value of $c(\chi)$ accordingly one obtains the lemma without exceptions. 
 \end{proof}
 \begin{lemma}\label{HeckeLowerBound}
 For any Hecke character $\chi$ over a number field 
 \begin{equation} \frac{1}{L(\sigma+it,\chi)} \ll_{\chi} (\log (|t|+3))^{2/3} (\log \log (|t|+16))^{1/3}\end{equation}
 uniformly for $\sigma+it \in \mathcal{R}(1,c'(\chi))$ with an effective $0<c'(\chi)< c(\chi)$.
 \end{lemma}
 \begin{proof}
 Given \cite[Thm.\ 1 and \S5]{Coleman} the proof essentially follows that of \cite[Thm.\ 3.11]{TOTRZF} with $\phi(t) = c_2(\chi) \log \log (|t|+16)$ where $c_2(\chi)$ is as in \cite[Thm.\ 1]{Coleman} and $\theta(t)= c'(\chi) \left(\frac{\log \log (|t|+16)}{\log (|t|+3)}\right)^{2/3}$ (cf.\ \cite[Lem.\ 11 and Rem.\ 3]{HuangSupNormDihedral}). 
 \end{proof}

\begin{lemma}[Poisson Summation]\label{Psum}
Let $H \leq G$ be locally compact commutative groups such that the quotient $G/H$ is compact. Let $f \in L^1(G)$ and write $\widehat{f}$ for its Fourier transform $$\widehat{f}(\psi) = \int_G f(g) \overline{\psi}(g)\,dg.$$ If \begin{enumerate}
\item\label{Psum1} the restriction of $\widehat{f}$ to $(G/H)^\wedge$ is integrable, 
\item\label{Psum2} for all $x \in G$ the function $y \mapsto f(xy)$ is integrable on $H$, and 
\item\label{Psum3} the map $x \mapsto \int_H f(xy)\,dy$ is continuous on $G$,
\end{enumerate}
then for all $x \in G$ we have $$ \int_H f(xh)\,dh = \frac{1}{\vol(G/H)} \sum_{\substack{\psi \in (G/H)^\wedge }} \widehat{f}(\psi) \psi(x).$$
\end{lemma}
\begin{proof} See \cite[Ch.II \S1 7.\ Cor.]{BourbakiTheoriesSpectrales}, which does not assume that $G/H$ is compact. This latter hypothesis is only used to write the integral over $(G/H)^\wedge$ as a sum above.
\end{proof}
\begin{rem}There is also a more general version of Lemma \ref{Psum} without the hypothesis that $G/H$ be compact or the hypotheses \eqref{Psum2} and \eqref{Psum3}, but in which the conclusion only holds for almost every $x \in G$, see \cite[Ch.\ II \S1 7.\ Prop.\ 15]{BourbakiTheoriesSpectrales}.
\end{rem}

\section{Local non-archimedean theory}\label{nonarch_top}
\subsection{Local Langlands correspondence, local conductors}\label{sec:non-arch_background}
We now restrict our attention to non-archimedean local fields $F$. 
\subsubsection{The Artin conductor}\label{ArtinC}
Let $T$ be an $F$-torus and $(W_F, \varphi, \{r_E\})$ a Weil group for $F$. Let $r: \LT \to \GL(V)$ be a finite-dimensional complex representation of the $L$-group of $T$ as in section \ref{sec:Lgroupsandreps}. Recall the $\eps$-factors attached to finite-dimensional complex Galois representations of $W_F$ from section \ref{sec:LLC}.
\begin{definition}\label{artinc1} If $\varphi \in \Phi(T)$ corresponds to $\chi \in \Hom(T(F), \C^\times)$ under the local Langlands correspondence \eqref{LLCforTori}, then the quantity 
$$\fc (\chi, r) = |\eps(r\circ \varphi, \psi,dx)|^2$$ is called the local analytic conductor of $\chi$ with respect to $r$.\end{definition}  Tate \cite[\S3.4.2]{tateNTB} shows that $\eps(V,\psi,dx)$ is additive, and in particular only depends on the isomorphism class of $V$.  If $(\rho,V)$ is a unitary representation we have (see \cite[\S 3.4.7]{tateNTB}) that \begin{equation}\label{usesunitary} |\eps(V, \psi,dx)|^2 = q_F^{c(\rho)} (\delta(\psi) dx/dx')^{\dim(V)}.\end{equation} In particular, since $\eps(V,\psi,dx)$ only depends on the isomorphism class of $(\rho,V)$, it suffices for \eqref{usesunitary} to hold that $(\rho,V)$ be unitarizable. Here $q_F$ is the cardinality of the residue field of $F$ and $c(\rho)$ is the Artin conductor of the representation $(\rho,V)$.

In light of \eqref{usesunitary} we next review the definition of the Artin conductor $c(\rho)$ of a finite-dimensional complex representation $\rho:W_F \to  \GL(V)$ of the Weil group of a non-archimedean local field.  
The classical Artin conductor is an invariant of a finite dimensional complex representation of a finite Galois group $\Gal(E/F)$.  For more discussion of the classical Artin conductor see \cite[Ch.VI]{SerreLF} or \cite[\S4]{UlmerConductors}.  We give a slightly nonstandard definition of the Artin conductor of a finite-dimensional complex representation of $W_{L/F}$ following \cite{UlmerConductors} (this goes back at least to \cite{DarmonDiamondTaylor}).

Let $\varphi:W_F \hookrightarrow G_F$ be the inclusion given as part of the data of a Weil group (see \cite[\S 1.4.1]{tateNTB}).  For any $v \in [-1,\infty)$, let $W_F^v$ be the inverse image of the (upper-numbering) higher ramification group $G_F^v$ by $\varphi$ (see Serre \cite{SerreLF}  for definitions, especially Ch.~IV, \S3, Remark 1).  Let $L/F$ be a finite extension and $W^v_{L/F}$ be the image of $W_F^v$ by the canonical projection $W_F \twoheadrightarrow W_{L/F}$.  
The groups $W^v_{L/F}$ therefore define a descending filtration of $W_{L/F}$ with index set $[-1,\infty)$. 

\begin{propdef}\label{propdef} For a finite dimensional complex representation $\rho:W_{L/F}\to \GL(V)$, the number $$c(\rho) = \int_{-1}^\infty \codim (V^{\rho(W^v_{L/F})})\,dv$$ is called the \emph{Artin conductor} of $(\rho,V)$. The value of $c(\rho)$ only depends on $\rho\vert_{W^0_{L/F}}$, and extends the notion of Artin conductor for complex representations of finite Galois groups.
\end{propdef}
\begin{proof} 
Since there are no breaks in the upper-numbering filtration between $-1$ and $0$, and the upper-numbering is left-continuous (see \cite[Ch.4 \S3]{SerreLF}), it follows that the Artin conductor $c(\rho)$ only depends on the restriction of $\rho$ to $W_{L/F}^0$.  

Since $W_{L/F}^0$ is compact and profinite and $\GL(V)$ has no small subgroups, it follows that $H =\ker \rho \vert_{W_{L/F}^0}$ is a finite index open subgroup of $W_{L/F}^0$.  We have that $\rho \vert_{W_{L/F}^0}$ then factors through the finite quotient $W_{L/F}^0/H$. From this, the inverse image of $H$ in $W_F^0$ also has finite index, and contains $[W_L,W_L]$, thus (see Tate \cite[\S 1.4.5]{tateNTB}) we have that $H = W_{L/E}^0$ for some finite extension $L^\text{ab}/E/F$.  By Serre \cite[Ch.~IV, Prop. 14]{SerreLF} we have that $$ \frac{W_{L/F}^0}{H} = \frac{W_{L/F}^0 H}{H} = \left( \frac{W_{L/F}}{W_{L/E}}\right)^0 \simeq \Gal(E/F)^0$$ and indeed, for all $v \in (-1,\infty)$ that $$\frac{W^v_{L/F}}{H \cap W_{L/F}^v} \simeq \frac{W_{L/F}^v H}{H} =  \left( \frac{W_{L/F}}{W_{L/E}}\right)^v \simeq \Gal(E/F)^v.$$ Therefore to $\rho\vert_{W_{L/F}^0}$ there is associated a finite extension $E/F$ with $E\subseteq L^\text{ab}$, and $\rho$ factors through the representation $\rho':\Gal(E/F)\to \GL(V)$ given by composing with the isomorphisms above.  It is shown in \cite[\S 4]{UlmerConductors} for finite dimensional complex representations of finite Galois groups that the standard definition of the Artin conductor matches the one given in the Proposition/Definition with the higher ramification groups $G^v(E/F)$ in place of the Weil group and $\rho'$ in place of $\rho$. \end{proof}

The Artin conductor is a special case of the following more general notion of conductor. 
\begin{definition}\label{filtcond} For $G$ a group endowed with a descending filtration $\cF = (G^v)_{v \in (-1, \infty)}$ and $\rho:G \to \GL(V)$ a finite dimensional complex representation, we call $$c_\cF(\rho) = \int_{-1}^\infty \codim V^{\rho(G^v)}\,dv$$ the \emph{conductor of $(\rho,V)$ with respect to} $\cF$.
\end{definition}
With $G = W_{L/F}$ and $\cF$ given by the upper-numbering filtration, $c_\cF(\rho)$ the Artin conductor of $(\rho,V)$. 

Next, we introduce an ``abelian'' conductor $\tilde{c}(\rho)$.  The Artin conductor of a representation $(\rho,V)$ is controlled by the abelian conductor, and in the case that the representation factors through $W_{L/F}$ for $L/F$ an unramified extension, the abelian conductor is identical to the Artin conductor.

We denote \emph{for any} $v \in (-1,\infty)$ the groups $$\cO_L^{(v)} = \begin{cases}  1+\pi_L^{\lceil v \rceil} \cO_L & \text{ if } v >0 \\ \cO_L^\times & \text{ if }0 \geq v>-1.\end{cases}$$  In particular the function $v \mapsto \cO_L^{(v)}$ is locally constant on $(-1,\infty) - \Z$, and satisfies $\lim_{v \to n^-} \cO_L^{(v)}= \cO_{L}^{(n)}$ for $n \in \Z$, where $\lim_{v \to n^-}$ denotes the one-sided limit from below.

Let $\psi_{E/F}$ and $\phi_{E/F}$ be the Hasse-Herbrand functions for an extension $E/F$ of non-archimedean local fields, see \cite[Ch.\ IV \S 3]{SerreLF}.  Recall the short exact sequence \eqref{WLFes} expressing $W_{L/F}$ as a group extension of $G=\Gal(L/F)$. 
\begin{lemma}\label{bigdiagram}
For any real number $v>-1$, the following diagram commutes and the horizontal rows are short exact sequences: $$\xymatrix{
1 \ar[r] & \cO_L^{(\psi_{L/F}(v))} \ar[r] \ar[d] & W^{v}_{L/F} \ar[r] \ar[d] & G^{v} \ar[d] \ar[r] & 1 \\   
1 \ar[r] & L^\times \ar[r]^{r_L} & W_{L/F} \ar[r]^{\sigma} &G\ar[r] & 1  .}$$
\end{lemma}
\begin{proof}
The Artin reciprocity homomorphism $r_L$ maps the subgroup $\cO_L^{(v)} \subseteq L^\times$ onto the $v$th higher ramification group $W_L^{\text{ab},v}$ of $W_L^\text{ab}$ in the upper-numbering (see Serre \cite[Ch.~XV, Thm. 2]{SerreLF}). We shall need an analogue of \cite[Ch.~IV, Prop.~2]{SerreLF} for the upper-numbering filtration of $W_L^{\text{ab}}$, so we work with the Hasse-Herbrand functions. By definition of the upper-numbering filtration and the transitivity of the function $\psi$ under field extensions (see \cite[Ch.\ IV Prop.\ 15]{SerreLF}), we have
$$W_L^{\text{ab},\psi_{L/F}(v)} = W^{\text{ab}}_{L, \psi_{L^{\text{ab}}/L} \circ \psi_{L/F}(v)} = W^{\text{ab}}_{L, \psi_{L^{\text{ab}}/F}(v)}.$$
By \cite[Ch.~IV, Prop.~2]{SerreLF} and converting back to the upper-numbering filtration, we have 
$$W^{\text{ab}}_{L, \psi_{L^{\text{ab}}/F}(v)} = W_{L/F, \psi_{L^{\text{ab}}/F}(v)}\cap W^{\text{ab}}_L = W^v_{L/F}\cap W^{\text{ab}}_L.$$ 
Therefore, $\cO_L^{(\psi_{L/F}(v))} \simeq W_L^{\text{ab},\psi_{L/F}(v)} \subseteq W^v_{L/F}$, where the first $\simeq$ is the Artin map, so the left hand square in the statement of the lemma commutes and all four maps are injections.

Next we compute the cokernel. By the foregoing, $$\frac{W^v_{L/F}}{W^{\text{ab},\psi_{L/F}(v)}_{L}} = \frac{W^v_{L/F}}{W^v_{L/F} \cap W_L^\text{ab}} \simeq \frac{W^v_{L/F}W_L^\text{ab} }{W_L^\text{ab}}.  $$ We apply \cite[Ch.~IV, Prop. 14]{SerreLF} with $G=W_{L/F}$ and $H= W_L^\text{ab}$ to see that  $$\frac{W^v_{L/F}W_L^\text{ab} }{W_L^\text{ab}}\simeq \left( W_{L/F} / W_L^\text{ab} \right)^v.$$ Finally, by the third group isomorphism theorem and \cite[\S1.1]{tateNTB} we conclude that $$\left( W_{L/F} / W_L^\text{ab} \right)^v \simeq G^v,$$ since the canonical inclusion $\varphi:W_F \hookrightarrow G_F$ has dense image. \end{proof}

We are ready to give a definition of the abelian conductor $\tilde{c}(\rho)$. \begin{definition}\label{modifiedArtin} Let $(\rho,V)$ be a finite-dimensional complex representation of $W_{L/F},W_{L/F}^0,L^\times$, or $\cO_L^\times$, where the latter two groups are viewed as subgroups of $W_{L/F}$ as in Lemma \ref{bigdiagram}.  Let $\mathcal{U}$ be the descending filtration defined by $(\cO_L^{(v)})_{v \in (-1,\infty)}$. 
Then $$\tilde{c}(\rho) = c_\mathcal{U}(\rho)$$ is called the \emph{abelian conductor} of $(\rho,V)$.
\end{definition}

Remark: The abelian conductor $\tilde{c}$ is additive in the sense that if $\rho=\rho_1 \oplus \rho_2$, then $\tilde{c}(\rho)=\tilde{c}(\rho_1)+\tilde{c}(\rho_2)$.

 \begin{definition}\label{modifiedAC} The \emph{abelian local analytic conductor} $\tilde{\fc}(\chi,r)$ attached to $\chi,r$ is the complex number $$\tilde{\fc}(\chi,r) = q_F^{\tilde{c}(r\circ \varphi)} (\delta(\psi) dx/dx')^{\dim(r)},$$ where $\varphi \in \Phi(T)$ corresponds to $\chi \in \Hom(T(F), \C^\times)$ under the LLC for tori \eqref{LLCforTori}. \end{definition}

Finally, we note that the abelian conductor controls the Artin conductor and vice-versa. Let $v_0 = \inf \{ v : G^v(L/F) = \{1\}\}.$ For example, if $L/F$ is unramified then $v_0=-1$.  
\begin{lemma}\label{conductorbound}
We have $$\frac{1}{e_{L/F}}\tilde{c}(\rho) \leq c(\rho) \leq \tilde{c}(\rho) +  (v_0+1)\dim V .$$ In particular, if $L/F$ is unramified, then $\tilde{c}(\rho)=c(\rho).$\end{lemma}
\begin{proof} For the first inequality, we have from Lemma \ref{bigdiagram} that $$\codim(V^{\rho(\cO_L^{(\psi_{L/F}(v))})}) \leq \codim(V^{\rho(W_{L/F}^v)}).$$
Then, by \cite[Ch.\ IV, Props.\ 12, 13]{SerreLF} we have 
\begin{multline*}
\frac{1}{e_{L/F}}\tilde{c}(\rho) \leq \int_{-1}^\infty \frac{\codim(V^{\rho(\cO_L^{(u)})})}{(G_0:G_u)}\,du =  \int_{-1}^\infty \codim(V^{\rho(\cO_L^{(u)})}) \phi'_{L/F}(u)\,du \\ = \int_{-1}^\infty \codim(V^{\rho(\cO_L^{(\psi_{L/F}(v))})}) \,dv \leq c(\rho). 
\end{multline*}
For the second inequality, we use the fact that if $v>v_0$ then $\codim(V^{\rho(\cO_L^{(\psi_{L/F}(v))})}) = \codim(V^{\rho(W_{L/F}^v)})$. We have 
\begin{align*} 
c(\rho) & =   \int_{-1}^{v_0} \codim(V^{\rho(W_{L/F}^v)})\,dv + \int_{v_0}^{\infty} \codim(V^{\rho(W_{L/F}^v)})\,dv  \\
 & \leq  (v_0+1)\codim (V^{\rho(W^0_{L/F})})  +  \int_{v_0}^{\infty} \codim(V^{\rho(\cO_L^{(\psi_{L/F}(v))})})\,dv \\
 & \leq   (v_0+1)\codim (V^{\rho(W^0_{L/F})}) +  \int_{-1}^{\infty} \codim(V^{\rho(\cO_L^{(\psi_{L/F}(v))})})\,dv \\ 
 & = (v_0+1)\codim (V^{\rho(W^0_{L/F})}) +  \int_{-1}^{\infty} \frac{\codim(V^{\rho(\cO_L^{(u)})})}{(G_0:G_u)}\,du \\ 
 & \leq   (v_0+1) \dim V +  \tilde{c}(\rho) .
\end{align*}
\end{proof}
\begin{corollary}\label{MTcond}
We have for $r,\chi,v_0$ as above and $m=\dim r$ that $$\tilde{\fc}(\chi,r)^{1/e_{L/F}} \mid \fc(\chi,r) \mid q_F^{(v_0+1)m}\tilde{\fc}(\chi,r).$$ In particular, if $L/F$ is unramified then 
$\tilde{\fc}(\chi,r) = \fc(\chi,r)$.
\end{corollary}
The reason we prefer the abelian conductor is the following. Recall from section \ref{sec:Lgroupsandreps} the set of co-weights $M$ associated to the representation $r:\LT \to \GL(V)$ of dimension $m$. 
\begin{proposition}\label{MP}
Let $\varphi$ be a Langlands parameter of $T$. Then, $\varphi \vert_{\cO_L^\times} \in \Hom_G(\cO_L^\times , \widehat{T})$ depends only on the equivalence class of $\varphi$ and writing $\xi = \varphi \vert_{\cO_L^\times}$, we have 
$$ \tilde{c}(r \circ \varphi ) = \sum_{\mu \in M} c(\mu \circ \xi),$$
where on the right hand side $c=c_{\mathcal{U}}$ as in Definition \ref{filtcond} with $\mathcal{U}$ being the standard filtration of $\cO_L^\times$. 
\end{proposition}
\begin{proof}
We have that $\cO_L^\times \subset L^\times$ maps to the trivial element of $G$ (see \eqref{WLFes}), so that $\cO_L^\times$ acts trivially on $\widehat{T}$. By definition, of a Langlands parameter $\varphi \vert_{\cO_L^\times}$ takes values in $\widehat{T} \rtimes 1 \subseteq \LT$, i.e.\ $\varphi \vert_{\cO_L^\times} = \xi \vert_{\cO_L^\times}$ for the $\xi \in H^1(W_{L/F},\widehat{T})$ corresponding to $\varphi$ across the bijection \eqref{bij_LLP_to_coh}. Since $\cO_L^\times$ acts trivially on $\widehat{T}$, we have $H^1(\cO_L^\times,\widehat{T})=\Hom_G(\cO_L^\times, \widehat{T})$ and the first assertion of the proposition follows. 

By definition, the abelian local analytic conductor of $r \circ \varphi$ only depends on the restriction to $\cO_L^\times$. 
We have that $$ r \circ \varphi \vert_{\cO_L^\times} = r \vert_{\widehat{T}} \circ \varphi \vert_{\cO_L^\times} = \bigoplus_{\mu \in M} \mu \circ \xi \vert_{\cO_L^\times},$$ so that the second assertion of the proposition follows by the additivity of the abelian conductor.
\end{proof}

\subsubsection{The canonical integral model of a torus and the norm map}\label{norm}
We begin with a brief discussion of the work of Voskresenskii \cite[\S 10.3]{Voskresenskii} on the canonical integral model of a  torus $T$ over a non-archimedean local field $F$. Let $L$ be the splitting field of $T$ with $G=\Gal(L/F)$. Let $\cO_F,\cO_L, f,\ell$ be the rings of integers and residue fields in $F,L$, respectively. 
\begin{lemma}[Voskresenskii Theorem 1]\label{Vosk1}
Given an $F$-torus $T$, there exists a faithfully flat $\cO_F$-algebra $A$ of finite type endowed with a Hopf algebra structure such that $\cT := \Spec A$ is an $\cO_F$-integral model for $T$, i.e.\ $\cT \times_{\cO_F} \Spec F \simeq T$. \end{lemma}
The affine group scheme $\cT$ of Lemma \ref{Vosk1} is called the \emph{canonical integral model} of $T$. Voskresenskii showed that $\cT$ is the unique integral model of $T$ that may be constructed by choosing an $F$-linear embedding $T \hookrightarrow \GL(V)$ and an $\cO_F$-lattice in the finite-dimensional $F$-vector space $V$ that is stable by the action of the unique maximal compact subgroup of $T(F)$.  Moreover, for any finite extension $E/F$ one has that $\cT(\cO_E)$ is the unique maximal compact subgroup of $T(E)$, in particular 
\begin{equation}\label{stdiso3}
\cT(\cO_L) \simeq \Hom( X^*(T), \cO_L) \quad \text{ and } \quad \cT(\cO_F) \simeq \Hom_G( X^*(T), \cO_L).
\end{equation} Compare \eqref{stdiso3} with \eqref{stdisom1}. We introduce the following abuse of notation: define \begin{equation}\label{abuse_of_notation}T(\cO_L):= \cT(\cO_L) \quad \text{ and } \quad T(\cO_F):= \cT(\cO_F).\end{equation}

It is not in general true that the special fiber $\cT \times_{\cO_F} \Spec f$ is itself a torus over $f$. Nonetheless, we do have the following result.
\begin{lemma}[Voskresenskii Theorem 2]\label{Vosk2}
If $T$ splits over an unramified extension $L/F$, then the canonical integral model $\cT= \Spec A$ of $T$ is given by $A =  \cO_{L}[X^*(T)]^G$ and $\cT_f:= \cT \times_{\cO_F} \Spec f$ is a torus over $f$. 
\end{lemma} 
If $T$ splits over an unramified extension, we commit the abuse of notation $T(f) := \cT_f (f)= \cT(f).$ 

  \begin{lemma}\label{lem:fibermodel}
 Let $T_1,T_2$ be two tori over a non-archimedean local field $F$ both splitting over some common unramified extension $L/F$ with Galois group $G$. Let $\alpha:T_1\to T_2$ and $x \in T_2(\cO_F)$ be an integer-valued point of $T_2$. Then, the fiber $\alpha^{-1}(x)\subseteq T_1$ of $\alpha$ over $x$ admits an $\cO_F$-integral model $\alpha^{-1}(x)_0$.  
 \end{lemma}
    \begin{proof}
    The $G$-equivariant map $\alpha^{*}: X^*(T_2) \to X^*(T_1)$ extends to a map of $\cO_F$-algebras $$\cO_L[X^*(T_2)]^G \to \cO_L[X^*(T_1)]^G.$$ Since $L/F$ is unramified, by the explicit description of the canonical model in Lemma \ref{Vosk2}, the map $\alpha : T_1 \to T_2$ extends to a map of the integral models $\alpha:  \cT_1 \to \cT_2$ over $\cO_F$.  The scheme-theoretic fiber $\alpha^{-1}(x)_0$ of $\alpha:\cT_1 \to \cT_2$ over $x \in \cT_2(\cO_F)$ is an integral model for $\alpha^{-1}(x)$. 
\end{proof}

The local to global decomposition of $Z(s)$ in section \ref{locglobandconclusion} will lead us to 
restrict the local Langlands correspondence \eqref{langlands} to a compact subgroup of $T(F)$ which has finite index in the maximal compact subgroup $T(\cO_F)$. As we will soon see in section \ref{pairingsec}, the natural choice is to restrict the Langlands correspondence to the image of the \emph{norm map} \begin{equation}\label{eq:norm}N:T(\cO_L) \to T(\cO_F)\end{equation} defined by the product of Galois conjugates. 

We will next prove an important lemma describing the image of $N$, which we write $NT(\cO_L)$. 
We begin with a preliminary but crucial result. \begin{lemma}\label{Amano}
Suppose $L/F$ is unramified.  The map $N:T(\cO_L)\to T(\cO_F)$ is surjective.
\end{lemma}
\begin{proof} See \cite[Cor.~of Thm.~1]{Amano}.\end{proof} 
Lemma \ref{Amano} will be used to deduce the last assertion of Proposition \ref{PairingProp} from the previous ones.
More generally, we have the following result. 
\begin{lemma}\label{keylemma}
We have $$\left| \frac{T(\cO_F)}{NT(\cO_L)}\right| \leq e_{L/F}^{\dim T} \left| \frac{T(F)}{NT(L)} \right|,$$ where $e_{L/F}$ is the ramification index of $L/F$. Both quotients are finite groups. 
\end{lemma}
\begin{proof}  
Applying the functor $\Hom(X^*(T), -)$ to the valuation exact sequence $$\xymatrix{1 \ar[r] & \cO_L^\times \ar[r] & L^\times \ar[r] &  \Z \ar[r] & 0}$$ yields  \begin{equation}\label{Ext1}\xymatrix{1 \ar[r] & T(\cO_L) \ar[r] &  T(L) \ar[r] &  \Hom(X^*(T),\Z) \ar[r] &  0
},\end{equation}
since $X^*(T)$ is a free abelian group, i.e.\ a projective $\Z$-module (see e.g.\ \cite[Lemma 2.2.3]{Weibel}).

Taking $G$-invariants in \eqref{Ext1} and using \eqref{stdiso3} we have $$\xymatrix{1 \ar[r] &  T(\cO_F) \ar[r] &  T(F) \ar[r] &  \Hom_G(X^*(T),\Z) \ar[r] &  H^1(G,T(\cO_L))}$$ and a commutative diagram $$ \xymatrix{ 1 \ar[r] & T(\cO_L) \ar[r] \ar[d]^{N} & T(L) \ar[r] \ar[d]^{N} & \Hom(X^*(T), \Z) \ar[r] \ar[d]^N & 0 \\ 
1 \ar[r] & T(\cO_F) \ar[r]  & T(F) \ar[r]  &\Hom_G(X^*(T), \Z).   &  }$$
The rightmost map $N$ above is given by $$N: \Hom(X^*(T),\Z) \to \Hom(X^*(T),\Z)$$ $$ \ell \mapsto \sum_{\sigma \in G} \ell^\sigma$$ where $ \ell^\sigma (\chi) = \ell(\chi^{\sigma^{-1}})$. 
Let $$\fK = \ker(N:  \Hom(X^*(T),\Z) \to  \Hom(X^*(T),\Z)) = \{\ell \in  \Hom(X^*(T),\Z) : \sum_{\sigma \in G} \ell^\sigma =0\}.$$ 
The middle map $N$ (recall \eqref{stdisom1}) is given by $$N: \Hom(X^*(T),L^\times) \to \Hom_G(X^*(T),L^\times)$$ $$ \psi \mapsto \prod_{\sigma \in G} \psi^\sigma,$$ where $\psi^\sigma$ is given by $\psi^\sigma(\chi) = \psi(\chi^{\sigma^{-1}})^{\sigma}.$ Let $$T_1 =\ker(N:T(L) \to T(F) ) = \{ \psi \in  \Hom(X^*(T),L^\times) : \prod_{\sigma \in G} \psi^\sigma = 1\}.$$
Define the valuation map $$v: T_1 \to \fK$$ $$ \psi \mapsto v(\psi)$$ where $v(\psi)(\chi)=v(\psi(\chi)).$
The snake lemma gives us the exact sequence
\begin{equation}\label{SNes} \xymatrix{T_1 \ar[r]^-v & \fK \ar[r]^-\delta & \frac{T(\cO_F)}{NT(\cO_L)} \ar[r] & \frac{T(F)}{NT(L)}}.\end{equation} 

We claim that $$\left| \frac{\fK}{v(T_1)} \right| \leq e_{L/F}^{\dim T}.$$ Indeed, let $\ell \in \fK$ be arbitrary.  We claim that $e_{L/F} \ell \in v(T_1)$. The first claim follows from this second claim on letting $\ell$ run through a $\Z$-basis for $\fK$, so it suffices to show this.  Now we show the second claim. Choose $\pi_F$ a uniformizer for $F$.  For $\ell \in \fK$ let $$ \psi_\ell \in \Hom(X^*(T),L^\times)$$ be defined by $$\psi_\ell(\chi) =\pi_F^{\ell(\chi)}.$$ Note that $$ \prod_{\sigma \in G} \psi_\ell^\sigma ( \chi) = \prod_{\sigma \in G} \psi_\ell(\chi^{\sigma^{-1}})^\sigma = \prod_{\sigma \in G} \pi_F^{\ell(\chi^{\sigma^{-1}})} = \pi_F^{\sum_\sigma \ell^{\sigma^{-1}}(\chi) }= \pi_F^0 =1,$$ since $\ell \in \fK$.  Therefore $\psi_\ell \in T_1$.  Note also that $$v(\psi_\ell)(\chi) =v(\psi(\chi)) = v(\pi_F^{\ell(\chi)}) = e_{L/F} \cdot \ell(\chi).$$ Thus we have shown that for all $\ell \in \fK$ we have $e_{L/F} \ell \in v(T_1)$, as claimed.

By the exact sequence \eqref{SNes} we have that $$\left| \frac{T(\cO_F)}{NT(\cO_L)} \right| \leq \left| \frac{\fK}{v(T_1)}\right| \cdot \left| \frac{T(F)}{NT(L)} \right| \leq e_{L/F}^{\dim T} \left| \frac{T(F)}{NT(L)} \right|.$$

For the second assertion, recall the definition of the Tate cohomology groups $\widehat{H}^n$ from e.g.~\cite[Ch.VIII]{SerreLF}. We have $$\frac{T(F)}{NT(L)} = \widehat{H}^0(G,T),$$ and by the Nakayama-Tate theorem (see e.g.~\cite[Thm.~6.2]{PlatonovRapinchuk}) $$\widehat{H}^0(G,T) \simeq \widehat{H}^2(G,X^*(T)).$$ Since $X^*(T)$ is a finitely generated abelian group, we have by e.g. \cite[\S 6 Cor.~2]{AtiyahWall} that $$|\widehat{H}^2(G,X^*(T)) | < \infty,$$ and so it follows that $|\frac{T(F)}{NT(L)}|$ is finite.
\end{proof}

\subsubsection{The Langlands pairing}\label{pairingsec}
The goal of this section is to restrict \eqref{langlands} to the compact subgroup $NT(\cO_L)$ of $T(F)$. To do this, we re-formulate the Langlands correspondence \cite{LanglandsAbelian} as a perfect pairing 
 \begin{equation}\label{LanglandsPP}T(F) \otimes H^1(W_{L/F}, \widehat{T}) \to \C^\times,\end{equation} which we call the Langlands pairing. We write $G^0$ for the inertia subgroup of $G$. The following is the main result of this section.
 \begin{proposition}\label{PairingProp} Write $R(H^1)$ for the image of the restriction to $\cO_L^\times$ map \begin{equation}\label{eq:R}R: H^1(W_{L/F},\widehat{T}) \to \Hom_G(\cO_L^\times,\widehat{T}).\end{equation}  The subgroup $R(H^1)$ of $\Hom_G(\cO_L^\times, \widehat{T})$ is of index at most $\leq |H^1(G^0,\widehat{T})|\cdot |H^2(G^0,\widehat{T})|$.  The Langlands pairing restricts to a perfect pairing $$NT(\cO_L) \otimes R(H^1) \to \C^\times.$$ In particular, if $L/F$ is unramified, the Langlands pairing restricts to a perfect pairing \begin{equation}\label{unrampairing}T(\cO_F) \otimes \Hom_G(\cO_L^\times, \widehat{T}) \to \C^\times.\end{equation} \end{proposition}

\begin{corollary}\label{ACcor}
The abelian local analytic conductor $\tilde{\fc}(\chi,r)$ only depends on $\chi \vert_{NT(\cO_L)}$.  \end{corollary}
\begin{lemma}\label{HG0lem} We have that $H^i(G^0,\widehat{T})$ is a finite group for all $i\geq 1$. \end{lemma}
 Lemma \ref{HG0lem} follows from a result of Cartan and Eilenberg, which we recall now since it will also be useful for other purposes later. Recall the Tate cohomology groups $\widehat{H}_n$, see e.g.\ \cite[Ch.VIII]{SerreLF}.

\begin{theorem}[Duality Theorem]\label{CartanEilenberg} 
Let $G$ be a finite group, $A$ a $G$-module and $C$ a divisible abelian group. For any $i \in \Z$ there exists a perfect pairing $$\cup : \widehat{H}^i(G,A) \otimes \widehat{H}^{-i-1}(G, \Hom(A,C)) \to C.$$
 \end{theorem}
 \begin{proof} See \cite{CartanEilenberg} chapter XII, Theorems 4.1 and 6.4. \end{proof}
 \begin{proof}[Proof of Lemma \ref{HG0lem}]
 In Theorem \ref{CartanEilenberg} we take $G=G^0$, $A=\widehat{T}$, and $C=\C^\times$ to obtain for $i\geq 1$ that \begin{equation}\label{Hipp} H^i(G^0,\widehat{T}) \otimes H_i(G^0,X_*(T)) \to \C^\times \end{equation} is a perfect pairing.  Now, by \cite[\S 6 Cor.~1]{AtiyahWall} we have that $H^i(G^0,\widehat{T})$ is a group of finite exponent.  Furthermore, $X_*(T)$ is a finitely-generated $G^0$-module, so by \cite[\S 6 Cor.~2]{AtiyahWall} we have that $H_i(G^0,X_*(T))$ is a finite group. The result now follows from the duality theorem.
 \end{proof}
In the unramified case, we also have the following version of the Langlands correspondence over finite fields.  
 \begin{proposition}\label{LanglandsFF}
 If $L/F$ is unramified, then the Langlands pairing restricts to a perfect pairing 
 $$ T(f) \otimes \Hom_G(\ell^\times, \widehat{T}) \to \C^\times.$$
 \end{proposition}
 We give the proof of Proposition \ref{LanglandsFF} at the end of this section after first proving Proposition \ref{PairingProp}.
  
 To prepare for the proof of Proposition \ref{PairingProp}, we review the proof of the Langlands correspondence \eqref{langlands}.   To do so, we recall the following explicit descriptions of group cohomology and homology (the same exposition appeared in the appendix of \cite{BrooksPetrow}).  
 
 For this paragraph, let $G$ be a group and $M$ a left $G$-module. 
 Computing via the inhomogeneous resolution gives the usual description of group cohomology
$$
H^1(G, M) = \frac{\{  \xi: G \to M \,|\, \xi(gh) = \xi(g) + g\xi(h) \}}{ \{  \xi: G \to M  \,|\, \xi(g) = gm - m \text{ for some } m \in M\}}.
$$
If $\oplus_S N$ is a direct sum of copies of an abelian group $N$ indexed by a set $S$, let $\delta_s(n)\in \oplus_S N$ be the element which is $n$ in the $s$th entry and $0$ elsewhere. Computing via the inhomogeneous resolution then gives the following description of group homology
\begin{equation}\label{H_1def}
H_1(G, M) = \frac{\{(m_g)_{g \in G} \,|\, \sum_g (g^{-1}m_g - m_g) = 0\}}{ d(\oplus_{G \times G} M)},
\end{equation}
where $d(\delta_{g, h}(m)) = \delta_h(g^{-1}m) - \delta_{gh}(m) + \delta_g(m)$. If $G$ is abelian and acts trivially on $M$, then we have $H_1(G,M) \simeq G \otimes_\Z M$.

If $G' < G$ is a finite index normal subgroup, there is an action of $G/G'$ on $H_1(G', M)$ by the rule $g * \delta_{g'}(m) = \delta_{gg'g^{-1}}(gm)$.  There also exists a natural map 
$$
\operatorname{Trace}:H_1(G, M) \to H_1(G', M)^{G/G'},
$$
which may be computed as follows: pick coset representatives $g_1, g_2, \ldots, g_n$ for $G/G'$. Then any $g \in G$ determines a permutation $\tau \in S_n$ by the rule $g_i g = g' g_{\tau(i)}$ (where $g' \in G'$), and
$$
\operatorname{Trace} (\delta_g(m)) = \sum_i \delta_{g_i g g_{\tau(i)}^{-1}}(g_i m).
$$
  
  Now we return to our review of the Langlands correspondence \eqref{langlands}. In particular, $G = \Gal(L/F)$ again. Since $L^\times$ is abelian and acts trivially (recall \eqref{WLFes}) on $X_*(T)$, we have from the ``standard isomorphisms'' \eqref{stdisom1} and \eqref{stdisom2} and the above explicit description of group homology that 
  \begin{equation}\label{TFtoH1}
   T(L) \simeq  H_1(L^\times , X_*(T)) \quad \text{ and } \quad T(F) \simeq   H_1(L^\times , X_*(T))^G.
   \end{equation}

 Langlands proves the following mild extension of the Duality Theorem \ref{CartanEilenberg}.  Let $\alpha \in W_{L/F},$ $\chi \in X_*(T)$ such that $\delta_\alpha(\chi)$ is a cycle representing a class in $H_1(W_{L/F}, X_*(T))$. Let $\xi$ be a cocycle representing a class in $H^1(W_{L/F}, \widehat{T})$. Langlands \cite[p.\ 233-234]{LanglandsAbelian} shows that the pairing \begin{equation}\label{langlandscup}\cup:H_1(W_{L/F}, X_*(T)) \otimes H^1(W_{L/F}, \widehat{T}) \to \C^\times\end{equation} defined by $$\cup: \delta_\alpha(\chi) \otimes \xi \mapsto \chi(\xi(\alpha))$$ is a perfect pairing. 
 
The difficult part of Langlands's proof of his correspondence \eqref{langlands} is that the map  \begin{equation}\label{langlandstrace}\operatorname{Trace}: H_1(W_{L/F},X_*(T)) \to H_1(L^\times , X_*(T))^{G}\end{equation} is an isomorphism. Combining \eqref{TFtoH1}, \eqref{langlandstrace} and  \eqref{langlandscup}, we obtain the Langlands pairing \eqref{LanglandsPP}. 

 We now discuss the connection between the Langlands pairing and the norm map. The Artin map (see \eqref{WLFes}) induces a map \begin{equation}\label{rL*}r_{L, *}: H_1(L^\times, X_*(T)) \to H_1(W_{L/F},X_*(T))\end{equation} so that the triangle $$ \xymatrix{H_1(L^\times,X_*(T)) \ar[r]^-N \ar[rd]_{r_{L, *}} & H_1(L^\times , X_*(T))^G \\ 
& H_1(W_{L/F},X_*(T))  \ar[u]_\simeq}$$ commutes. Here $N$ is the norm map defined as a product of Galois conjugates, and the vertical map is the trace map \eqref{langlandstrace}. Composing with the isomorphisms \eqref{TFtoH1}, we have that the norm map $N:T(L) \to T(F)$ factors through the homology group $H_1(W_{L/F},X_*(T))$:  \begin{equation}\label{magictriangle}\xymatrix{ T(L) \ar[r]^-N \ar[rd] & T(F) \\
& H_1(W_{L/F},X_*(T)). \ar[u]_{\simeq}}\end{equation}
 
 \begin{proof}[Proof of Proposition \ref{PairingProp}]
 Let $$\Ann(NT(\cO_L)) \subseteq H^1(W_{L/F}, \widehat{T})$$ be the annihilator of $NT(\cO_L)$ with respect to the Langlands pairing \eqref{LanglandsPP}. To prove the proposition, it suffices to compute $\Ann(NT(\cO_L))$, and show that $$\frac{H^1(W_{L/F}, \widehat{T})}{\Ann(NT(\cO_L))} \simeq R(H^1),$$ is as described in the statement of the proposition. 
 
 \begin{lemma}\label{kerreslemma}
 Let \begin{equation}\label{ResDef} R: H^1(W_{L/F},\widehat{T}) \to \Hom_G(\cO_L^\times, \widehat{T})\end{equation} be the restriction to $\cO_L^\times$ map (recall $\cO_L^\times$ acts trivially on $\widehat{T}$). Then we have $$\Ann(NT(\cO_L))= \ker(R).$$
 \end{lemma}
\begin{proof}Let $$t:H_1(W_{L/F},X_*(T)) \to T(F)$$ be the isomorphism obtained by composing the trace map \eqref{langlandstrace} with the isomorphism \eqref{TFtoH1}. The first step is to give an explicit description for the inverse image $t^{-1}(NT(\cO_L))\leq H_1(W_{L/F},X_*(T))$. Later, we use the explicit description for the cup product pairing \eqref{langlandscup} to compute $\Ann(NT(\cO_L))$.
 
 The main trick to compute $t^{-1}(NT(\cO_L))$ is to use the commuting triangle \eqref{magictriangle}, as the trace map is difficult to work with directly. Restricting \eqref{magictriangle} to the maximal compact of $T(L)$ we obtain  $$\xymatrix{ T(\cO_L) \ar@{->>}[r]^-N \ar@{->>}[rd] & NT(\cO_L) \\
& t^{-1}(NT(\cO_L)), 
\ar[u]_\simeq}$$
 where all arrows are surjective. Since we understand the diagonal arrow much better than the vertical one, this yields a description for $t^{-1}(NT(\cO_L))$. In the above explicit description for group homology, it is the subgroup of $H_1(W_{L/F},X_*(T))$ generated by sums of all possible homology classes $\delta_\alpha(\chi)$ as $\alpha$ runs over $\alpha \in \cO_L^\times \subset W_{L/F}$. 
 
 We now use the description $ t^{-1}(NT(\cO_L))= \langle \delta_\alpha(\chi)\rangle_{\alpha \in \cO_L^\times}$ and compute the annihilator $$\Ann(NT(\cO_L))= \Ann(t^{-1}(NT(\cO_L))) \subseteq H^1(W_{L/F},\widehat{T})$$ across \eqref{langlandscup}. 
 
 First we prove $\Ann(NT(\cO_L)) \supseteq \ker(R)$. Let $\xi$ represent a class in $\ker(R)$. Then $\xi$ vanishes on $\cO_L^\times$ by definition, and we have $\xi \cup \delta_\alpha(\chi)= 1$ for all $\delta_\alpha(\chi)$ with $\alpha \in \cO_L^\times$ by the definition \eqref{langlandscup} of $\cup$. It follows from the description $ t^{-1}(NT(\cO_L))= \langle \delta_\alpha(\chi)\rangle_{\alpha \in \cO_L^\times}$ that $\xi \cup x=1$ for all $1$-cycles $x\in t^{-1}(NT(\cO_L))$.
  Therefore $\ker(R) \subseteq \Ann(t^{-1}(NT(\cO_L))).$ 
 
 Now we prove $\Ann(NT(\cO_L)) \subseteq \ker(R)$. Suppose $\xi \in H^1(W_{L/F},\widehat{T})$ does \emph{not} represent any class in $\ker(R)$. Then there exists a $\beta \in \cO_L^\times$ for which $\xi (\beta) \neq 1$. Since $\xi(\beta) \neq 1$ there exists $\chi\in X_*(T)$ not vanishing on $\xi(\beta)\in \widehat{T}$.  Since $\cO_L^\times$ acts trivially on $X_*(T)$, we have that $\chi_\beta$ is a cycle, and thus represents a homology class.  Thus $\delta_\beta(\chi) \in t^{-1}(NT(\cO_L)) $ and $\xi \cup \delta_\beta(\chi) \neq 1$, so $\xi \not \in \Ann(t^{-1}(NT(\cO_L)))$.  Therefore $\ker(R)^c \subseteq \Ann(t^{-1}(NT(\cO_L)))^c$, so we have $\ker(R) = \Ann(t^{-1}(NT(\cO_L))$.  \end{proof}
 By Lemma \ref{kerreslemma} we have shown that $$NT(\cO_L) \otimes \frac{H^1(W_{L/F},\widehat{T})}{\ker(R)} \to \C^\times$$ is a perfect pairing. It now suffices to show that $$ \frac{H^1(W_{L/F},\widehat{T})}{\ker(R)} \simeq R(H^1) \leq \Hom_G(\cO_L^\times, \widehat{T})$$ is of index at most $\leq |H^1(G^0,\widehat{T})|\cdot |H^2(G^0,\widehat{T})|$, as in the statement of Proposition \ref{PairingProp}.
 
 Consider the inertia group $W_{L/F}^0$ acting on $\widehat{T}$, and the exact sequence $$\xymatrix{1 \ar[r] &  \cO_L^\times \ar[r] &  W_{L/F}^0 \ar[r] & G^0 \ar[r] & 1}$$ as in Lemma \ref{bigdiagram}.  We take the inflation-restriction-transgression exact sequence (see e.g.\ \cite[(1.6.7) Prop.]{NSW}) attached to these data 
 \begin{equation}\label{irseq2}\xymatrix{1 \ar[r] & H^1(G^0,\widehat{T}) \ar[r]^-i & H^1(W_{L/F}^0, \widehat{T}) \ar[r]^-r & \Hom_{G^0}(\cO_L^\times,\widehat{T}) \ar[r]^-g & H^2(G^0,\widehat{T}) .}\end{equation}
 These give $$ \xymatrix{1 \ar[r] & H^1(G^0,\widehat{T}) \ar[r]^-i & H^1(W_{L/F}^0, \widehat{T}) \ar[r]^-r & \ker(g) \ar[r] & 1 ,}$$ where $\ker (g)$ is a subgroup of $\Hom_{G^0}(\cO_L^\times,\widehat{T})$ of index at most $|H^2(G^0,\widehat{T})|$. We take Frobenius invariants of this to obtain a sequence $$\xymatrix{1 \ar[r] & H^1(G^0,\widehat{T})^{\Z} \ar[r]^-{i'} & H^1(W_{L/F}^0, \widehat{T})^{\Z} \ar[r]^-{r'} & \ker(g)^{\Z} \ar[r] & H^1(\Z,H^1(G^0,\widehat{T})),}$$ where $1 \in \Z$ acts by arithmetic Frobenius on $\widehat{T}$. Since a cocycle is determined by its value on a generator, we have $$|H^1(\Z,H^1(G^0,\widehat{T}))|\leq |H^1(G^0,\widehat{T})|. $$  Therefore $r'(H^1(W_{L/F}^0, \widehat{T})^{\Z})$ has index at most $|H^1(G^0,\widehat{T})|$ in $ \ker(g)^{\Z}$, and $\ker(g)$ has index at most $|H^2(G^0,\widehat{T})|$ in $\Hom_{G^0}(\cO_L^\times,\widehat{T})$, so $ \ker(g)^{\Z}$ has index at most $|H^2(G^0,\widehat{T})|$ in $\Hom_G(\cO_L^\times,\widehat{T})$. Thus $r'(H^1(W_{L/F}^0, \widehat{T})^{\Z})$ has index at most $|H^1(G^0,\widehat{T})|\cdot |H^2(G^0,\widehat{T})|$ in $\Hom_G(\cO_L^\times,\widehat{T})$. 
 
 Consider again $W_{L/F}$ acting on $\widehat{T}$, and take the exact sequence $$\xymatrix{1 \ar[r] & W_{L/F}^0 \ar[r] & W_{L/F} \ar[r] & \Z \ar[r] & 1}.$$ Taking the inflation-restriction exact sequence associated to these we have $$\xymatrix{1 \ar[r] & H^1(\Z,\widehat{T}^{G^0}) \ar[r]^{i''} & H^1(W_{L/F}, \widehat{T}) \ar[r]^{r''} & H^1(W_{L/F}^0,\widehat{T})^{\Z} \ar[r] & 1. }$$ Here the term $H^2(\Z,\widehat{T}^{G^0})$ vanishes because the cohomological dimension of $\Z$ is one (see \cite[Ch. VIII, \S2]{BrownCohofGroups}). 
 We have that $R(H^1) = (r' \circ r'')(H^1(W_{L/F}, \widehat{T})),$ and by the above remarks we conclude that $R(H^1)$ has index at most $|H^1(G^0,\widehat{T})|\cdot |H^2(G^0,\widehat{T})|$ in $\Hom_G(\cO_L^\times,\widehat{T})$, as was to be shown.
 
 In the case that $L/F$ is an unramified extension, we have $G^0=\{1\}$, so that the first part of Proposition \ref{PairingProp} gives us that $R(H^1)= \Hom_G(\cO_L^\times, \widehat{T})$. We have $NT(\cO_L) = T(\cO_F)$ by Lemma \ref{Amano}, so that the Langlands pairing restricts to the perfect pairing \eqref{unrampairing}.
 \end{proof}
 
 \begin{proof}[Proof of Proposition \ref{LanglandsFF}]
 We again use the description of the local Langlands correspondence in terms of group homology described above. First, recall \cite[Prop.\ 2.3.1]{OnoTori} that we have an exact sequence 
 \begin{equation*}
 \xymatrix{ 1 \ar[r] & \Hom_G(X^*(T), 1+\fP_L) \ar[r] & T(\cO_F) \ar[r] & T(f) \ar[r] & 1.}
 \end{equation*}
 By the standard isomorphisms \eqref{stdiso3}, the left half of this exact sequence can be re-interpreted in terms of group homology. That is, we have the following commutative diagram
 \begin{equation*}
 \xymatrix{ \Hom_G(X^*(T), 1+\fP_L) \ar@{^{(}->}[r] \ar[d]^\simeq & T(\cO_F) \ar[d]^\simeq \\ 
 H_1(1+\fP_L,X_*(T))^G \ar@{^{(}->}[r]  & H_1(\cO_L^\times,X_*(T))^G.}
 \end{equation*}
 Recall in the course of the proof of Lemma \ref{kerreslemma}, we showed that there is a commuting triangle 
 \begin{equation}\label{triangle_diagram_3}
 \xymatrix{H_1(\cO_L^\times, X_*(T)) \ar@{->>}[rd]_{r_{L,*}} \ar[r]^-N & H_1(\cO_L^\times, X_*(T))^G  \\
 & \langle \delta_\alpha(\chi) \rangle_{\alpha \in \cO_L^\times} \ar[u]_{\operatorname{Trace}}^{\simeq},}
 \end{equation}
 where  $\langle \delta_\alpha(\chi) \rangle_{\alpha \in \cO_L^\times}$ is the subgroup of $H_1(W_{L/F},X_*(T))$ generated by sums of all possible homology classes $\delta_\alpha(\chi)$ as $\alpha$ runs over $\cO_L^\times$, $N$ is the norm map, $\operatorname{Trace}$ is the map defined in \eqref{langlandstrace}, and $r_{L,*}$ is induced by the Artin reciprocity map $r_L$ (see \eqref{rL*}). Since $L/F$ is unramified by hypothesis, $N$ is surjective by Lemma \ref{Amano}.  
 
 We determine the inverse image of the subgroup $H_1(1+\fP_L,X_*(T))^G$ by the trace map. Set $$H_1(T):=\langle \delta_\alpha(\chi)\rangle_{\alpha \in 1+\fP_L} \subseteq H_1(W_{L/F},X_*(T))$$ to be the subgroup of $H_1(W_{L/F},X_*(T))$ generated by sums of all possible homology classes $\delta_\alpha(\chi)$ as $\alpha$ runs over $1+\fP_L$. We have that $N^{-1}(H_1(1+\fP_L,X_*(T))^G) = H_1(1+\fP_L,X_*(T))$ since $L/F$ is unramified \cite[Prop.\ 1]{Amano}, so that  $$\operatorname{Trace}: H_1(T) \to H_1(1+\fP_L,X_*(T))^G$$ is an isomorphism by chasing the diagram \eqref{triangle_diagram_3}.  
 
In summary, we have a short exact sequence
 \begin{equation}
 \xymatrix{1 \ar[r] &  H_1(T) \ar[r] & T(\cO_F) \ar[r] & T(f) \ar[r] & 1.}
 \end{equation}
 By exactness of the dual functor, we have the short exact sequence ``on the automorphic side'' of the local Langlands correspondence
 $$\xymatrix{ 1 \ar[r] & T(f)^\wedge \ar[r] & T(\cO_F)^\wedge \ar[r] & \Hom(H_1(T),\C^\times) \ar[r] & 1.}$$

The strategy of the proof is now to write down another exact sequence ``on the Galois side'' of the local Langlands correspondence, and by the five lemma, conclude the local Langlands correspondence over finite fields. 
 We start with the short exact sequence 
 $$\xymatrix{1 \ar[r] & 1+\fP_L \ar[r] & \cO_L^\times \ar[r] & \ell^\times \ar[r] & 1}.$$
Since $\widehat{T}$ is a divisible group, the functor $\Hom(-, \widehat{T})$ is exact \cite[Cor.\ 2.3.2, Lem.\ 2.3.4]{Weibel} and we obtain the short exact sequence
  $$\xymatrix{ 1 \ar[r] & \Hom(\ell^\times, \widehat{T}) \ar[r] & \Hom(\cO_L^\times, \widehat{T}) \ar[r] & \Hom(1+\fP_L,\widehat{T}) \ar[r] & 1.}$$ 
 Taking the long exact sequence in cohomology we get 
  $$\xymatrix{ 1 \ar[r] & \Hom_G(\ell^\times, \widehat{T}) \ar[r] & \Hom_G(\cO_L^\times, \widehat{T}) \ar[r]^-{R} & \Hom_G(1+ \fP_L,\widehat{T}).}$$
 
Our goal now is to show that the image of the map $R$ is isomorphic to $\Hom(H_1(T),\C^\times)$. Recall the annihilator $$\Ann(H_1(T)):= \{ \xi \in \Hom_G(\cO_L^\times, \widehat{T}): x \cup \xi = 1 \text{ for all } x \in H_1(T)\},$$
 where $\cup$ is Langlands's cup product pairing \eqref{langlandscup}. 
 We claim that 
 $$\Ann(H_1(T)) = \ker (R ).$$  
 Indeed, if $\xi$ is in the kernel of $R$, then $\xi$ is trivial on $1+\fP_L$, and then for any cycle of the form $\delta_\alpha(\chi) $ with $\alpha \in 1+\fP_L$ we have $\xi(\alpha)=1$, so of course $\chi(\xi(\alpha))=1$ for all $\chi \in X_*(T)$, i.e.\ $\xi \in \Ann(H_1(T))$.   On the other hand, if $\xi \in \Ann(H_1(T))$, then we have that $x \cup \xi = 1$ for all $x \in H_1(T)$, so in particular, $\chi(\xi(\alpha))=1$ for all $\alpha \in 1+\fP_L$ and $\chi \in X^*(\widehat{T})$. This can only be the case if $\xi(\alpha) =1$ for all $\alpha \in 1+\fP_L$, i.e.\ $\xi \in \ker(R)$. 
 
By the perfect pairing $T(\cO_F) \otimes \Hom_G(\cO_L^\times, \widehat{T}) \to \C^\times$ of Proposition \ref{PairingProp}, we have  $$\Hom(H_1(T),\C^\times) \simeq \frac{\Hom_G(\cO_L^\times, \widehat{T})}{\Ann(H_1(T))} \simeq  R( \Hom_G(\cO_L^\times, \widehat{T})) \subseteq  \Hom_G(1+\fP_L,\widehat{T}).$$ Thus, we have a commutative diagram
$$ \xymatrix{1 \ar[r] & T(f)^\wedge \ar[r] & T(\cO_F)^\wedge \ar[r] \ar[d]^\simeq & \Hom(H_1,\C^\times)\ar[d]^\simeq \ar[r] & 1 \\ 
1 \ar[r] & \Hom_G(\ell^\times, \widehat{T}) \ar[r] & \Hom_G(\cO_L^\times, \widehat{T}) \ar[r] & R( \Hom_G(\cO_L^\times, \widehat{T})) \ar[r] & 1}$$
with short exact rows, 
where the first vertical $\simeq$ is Proposition \ref{PairingProp} and the second $\simeq$ is the one just established. The Proposition now follows from the five lemma of homological algebra. 
 \end{proof}

\subsection{Local conductor zeta function, unramified case}\label{unramified}

In this section $T$ is a torus over a non-archimedean local field $F$ with splitting field $L$ and Galois group $G = \Gal(L/F)$.  Let $\fP$ be the prime ideal of $\cO_L$, $\ell$ the residue field of $L$, and $\mathop{\rm char}(\ell)$ its characteristic. We assume throughout this section that the representation $r \vert_{\widehat{T}}$ is faithful.

Let $\theta$ be a character of $T(F)$ that is trivial on the subgroup $NT(\cO_L)$ and $x \in NT(\cO_L)$. In this section and the next, we consider the (twisted) generating series 
\begin{equation}\label{NFsxdef}
N_F(s,x) := \sum_{\chi \in NT(\cO_L)^\wedge} \frac{\chi(x)}{\fc(\chi \theta,r)^s}.
\end{equation}

Now, and for the rest of section \ref{unramified} we assume that the extension $L/F$ is \emph{unramified}. 
Thus, the results of section \ref{sec:non-arch_background} afford us several immediate reductions. We have $NT(\cO_L)=T(\cO_F)$ (Lemma \ref{Amano}), $\fc(\chi \theta, r)= \tilde{\fc}(\chi \theta, r)$ (Corollary \ref{MTcond}), and $\tilde{\fc}(\chi \theta, r)=\tilde{\fc}(\chi , r)$ (Corollary \ref{ACcor}), so that
$$ N_F(s,x)=\sum_{\chi\in T(\cO_F)^\wedge}\frac{\chi(x)}{\tilde{\fc}(\chi,r)^s }.$$  
Recall the local Langlands isomorphism on integral points  from Proposition \ref{PairingProp} \begin{equation}\label{integralLLCcounting_sec}\Hom_G(\cO_L^\times, \widehat{T}) \simeq T(\cO_F)^\wedge\end{equation} $$\xi \mapsto \chi_\xi.$$ Proposition \ref{MP} then gives us a more hands-on way of working with the abelian conductor $\tilde{\fc}(\chi , r)$ in terms of characters on the Galois side of the integral local Langlands isomorphism. Changing variables by \eqref{integralLLCcounting_sec} we have \begin{equation}\label{ur1}N_F(s,x)= (\delta(\psi) dx/dx')^m \sum_{ \xi \in \Hom_G(\cO_L^\times,\widehat{T})} \chi_\xi(x)q_F^{-s \sum_\mu c(\mu \circ \xi)}.\end{equation}

Recall the set $M$ of co-weights of $r$ from Definition \ref{def_coweights}. Recall the set of non-negative integers $\N$, and let us index the coordinates of $\N^M$ by $\mu \in M$. For each $c=(c_\mu)_{\mu \in M} \in \N^M$, consider the following sets of Langlands parameters (restricted to $\cO_L^\times$): \begin{equation}\label{Pleqc}P_\leq( c) = \{\xi \in \Hom_G(\cO_L^\times,\widehat{T}) : c(\mu \circ \xi)\leq c_\mu,  \text{ for all }\mu \in M\}\end{equation} and \begin{equation}\label{P=c}P_=(c) = \{\xi \in \Hom_G(\cO_L^\times,\widehat{T}) : c(\mu \circ \xi)= c_\mu,  \text{ for all }\mu \in M\}.\end{equation} 
Since $\mu: \widehat{T} \to \C^\times$ is a group homomorphism $\mu \circ (\xi_1.\xi_2) = (\mu \circ \xi_1). (\mu\circ \xi_2)$ and $P_{\leq}(c)$ is an abelian group. Since $r \vert_{\widehat{T}}$ is faithful the abelian group $P_\leq( c)$ is finite and so the subset $P_{=}(c)$ is finite as well. Indeed, if $\xi \in P_{\leq}(c)$ then $\mu \circ \xi(1+\fP^{\max c_\mu}) = 1$ for all $\mu \in M$, so that $$ r \circ \xi( 1+\fP^{\max c_\mu}) = \prod_{\mu \in M}  \mu \circ \xi( 1+\fP^{\max c_\mu}) =1.$$
Since $r \vert_{\widehat{T}}$ is faithful, we must have $\xi(1+\fP^{\max c_\mu})=1$ for all $\xi \in P_{\leq}(c)$. Then, $P_{\leq}(c)$ is finite as $(\cO_L/\fP^{\max c_\mu})^\times$ is finite and $\widehat{T}$ has only finitely many elements of order dividing $| (\cO_L/\fP^{\max c_\mu})^\times|$. 

We consider the character sums over $P_\leq( c)$ and $P_=( c)$ \begin{equation}\label{PileqPi=}\Pi_\leq( c, x) = \sum_{\xi \in P_\leq( c)} \chi_\xi(x) \quad \text{ and } \quad \Pi_=(c,x) = \sum_{\xi \in P_=(c)} \chi_\xi(x).\end{equation} For example, $\Pi_\leq( c,1)= |P_\leq (c)|$ and $\Pi_=(c,1)= |P_=(c)|$. 

Writing $|c|=\sum_\mu c_\mu$, the sums $\Pi_=(c,x)$ are the coefficients of $N_F(s,x)$ as a local Dirichlet series, i.e. \begin{equation}\label{rhotoc} N_F(s,x)=(\delta(\psi) dx/dx')^m \sum_{c \in \N^M} \frac{\Pi_=(c,x)}{q_F^{s|c|}}.\end{equation}

We begin our analysis with the sums $\Pi_\leq( c,x)$ in order to make use of the group structure of $P_\leq( c)$.  The two functions $\Pi_\leq( c,x)$ and $\Pi_=(c,x)$ are related by inclusion-exclusion: \begin{equation}\label{protoincexc} \Pi_=(c,x) = \sum_{b \in \{0,1\}^M} (-1)^{|b|} \Pi_\leq((c-b),x).\end{equation}

Recall that the set of co-weights $M$ admits an action of $G$. We also let $G$ act on $\N^M$ by permuting coordinates and let $D_k(c)$ be the complex diagonalizable group defined by \begin{equation}\label{Dkc}D_k(c) = \bigcap_{\substack{ \mu \in M \\ c_\mu \leq k}}\ker \mu \subseteq \widehat{T},\end{equation} i.e.\ $D_k(c) = D(S)$ for $S=\{\mu \in M : c_\mu >k\}$. If $c$ is $G$-fixed then $D_k(c)$ admits an action of $G$. Note that $D_k(c)$ is monotonic in $c$, i.e.~if $c'\leq c$ coordinate-wise then for any $k\geq 0$ we have $$D_{k}(c')\subseteq D_k(c).$$   
The main result of this section of the paper is the following.
\begin{proposition}\label{MurP}
Suppose $r\vert_{\widehat{T}}$ is faithful, $c\in \N^M$ is $G$-fixed, $L/F$ is unramified, and $(q_F,\lambda)=1$.  If $\chi_\xi(x)=1$ for all $\xi \in P_\leq( c),$ then $$\Pi_\leq( c,x)= \left| \Hom_{G} \left(\ell^\times, D_0(c)\right)\right| \prod_{k=1}^\infty  \mathop{\rm char}(\ell)^{\dim D_k(c)},$$ and if there exists $\xi \in P_\leq( c)$ such that $\chi_\xi(x) \neq 1$ then $\Pi_\leq( c,x) = 0$.  
\end{proposition}
\begin{proof}
Suppose that $\chi_\xi(x) = 1$ for all $\xi \in P_\leq (c)$. Then $\Pi_\leq (c,x) = |P_\leq (c)|$ and it suffices to count the latter set. Since $r \vert_{\widehat{T}}$ is faithful, we have $D_k(c) = \{1\}$ for sufficiently large $k \in \N$, and so there exists $$k_0 = k_0(c) = \min \{ k \in \N: D_k(c)=\{1\}\}.$$ Thus the product in the statement of the proposition is finite, running up to $k_0-1$. 

We have 
\begin{align*} P_{\leq}(c) & = \bigcap_{\mu \in M} \{ \xi \in \Hom_G(\cO_L^\times,\widehat{T}) : c(\mu \circ \xi) \leq c_\mu \} \\
& = \bigcap_{\mu  \in M } \bigcap_{\substack{k=0 \\ c_\mu \leq k}}^\infty  \{ \xi \in \Hom_G(\cO_L^\times,\widehat{T}) : \mu \circ \xi (1+\fP^k) = 1 \} \\
& = \bigcap_{k=0}^\infty \bigcap_{\substack{\mu  \in M  \\ c_\mu \leq k}}   \{ \xi \in \Hom_G(\cO_L^\times,\widehat{T}) : \xi (1+\fP^k) \subseteq \ker \mu \} \\
& = \bigcap_{k=0}^\infty  \{ \xi \in \Hom_G(\cO_L^\times,\widehat{T}) : \xi(1+\fP^k) \subseteq D_k(c)\}.
\end{align*}
That is to say, a parameter $\xi \in P_\leq (c)$ if and only if $\xi(1+\fP^{k}) \subseteq D_k(c)$ for all $k\in \N $.
In particular, every $\xi \in P_\leq (c)$ is trivial on $1+\fP^{k_0}$. We inductively construct all of the $\xi \in P_\leq (c)$ by extending the trivial homomorphism $1+\fP^{k_0}\to \widehat{T}$ backwards along the standard filtration.

Consider two base cases: $k_0=0$ and $k_0=1$. If $c$ is such that $k_0=0$ then $D_k(c)=\{1\}$ for all $k\in \N$ and $P_\leq( c) = \{1\}$, so the formula in the statement of the proposition holds. If $c$ is such that $k_0=1$ then $\xi(1+\fP)=\{1\}$ for all $\xi \in P_\leq( c)$, and the possible extensions of $\xi$ to $\cO_L^\times$ are parametrized by $$\Hom_G(\cO_L^\times / (1+\fP), D_0(c)) = \Hom_G(\ell^\times, D_0(c)).$$ So the formula in the statement of the proposition holds.

Now suppose as the induction hypothesis that \begin{equation}\label{inductionhyp}|P_\leq (c)| = \left|  \Hom_G(\ell^\times, D_0(c))\right| \prod_{k=1}^{k_0-1} \left| \Hom_G(\ell,D_k(c))\right|\end{equation} for all $c$ such that $k_0 \leq K$. Consider $c$ such that $k_0 = K+1$. Then all $\xi \in P_\leq( c)$ satisfy $\xi(1+\fP^{K+1}) = \{1\}$, and the possible extensions the trivial map $1+\fP^{K+1} \to \widehat{T}$ to elements of $\Hom_{G}(1+\fP^{K},D_{K}(c))$ are parameterized by $$\Hom_G((1+\fP^{K})/(1+\fP^{K+1}),D_{K}(c)) \simeq \Hom_G(\ell,D_{K}(c)),$$ since $L/F$ is unramified. Therefore \eqref{inductionhyp} holds for $c$ such that $k_0= K+1$. By induction, \eqref{inductionhyp} holds for all $c \in \N^M$.

By the normal basis theorem, there exists $\alpha \in \ell$ such that $$\{\alpha, \alpha^{q_F}, \alpha^{q_F^2},\ldots,\alpha^{q_F^{[L:F] -1}}\}$$ is a basis for $\ell$ over the residue field of $F$.  A $G$-equivariant homomorphism in $\Hom_G(\ell,D_k(c))$ is determined by its value on $\alpha$, which is of additive order $\mathop{\rm char}(\ell)$ in $\ell$.  Since $(q_F,\lambda)=1$, the element $\alpha$ cannot map non-trivially into the component group of any $D_k(c)$.  There are $$\mathop{\rm char}(\ell)^{\dim D_k(c)}$$ elements of order dividing $\mathop{\rm char}(\ell)$ in the connected component of the identity of $D_k(c)$.  Hence $$\left| \Hom_G(\ell,D_k(c)) \right| = \mathop{\rm char}(\ell)^{\dim D_k(c)},$$ and we have shown the first part of the Proposition.

If there exists $\xi \in P_\leq( c)$ such that $\chi_\xi(x) \neq 1$, then it immediately follows from orthogonality of characters that $\Pi_\leq( c,x) = 0$, hence the second part of the proposition. 
\end{proof}

Proposition \ref{MurP} is only valid for $G$-fixed $c\in \N^M$ (since otherwise $D_k(c)$ is not a $G$-module, and $G$-equivariant homomorphisms into $D_k(c)$ do not make any sense). However, we can always reduce to the case that $c$ is $G$-fixed by the following lemma.
 \begin{lemma}\label{reducedlemma} If $c\in \N^M$ is \emph{not} $G$-fixed, then $\Pi_=(c,x) = 0$.\end{lemma} 
\begin{proof}
Suppose $c$ is not fixed by $G$, so that $|M| \geq 2$. Without loss of generality suppose there exists $\sigma \in G$ such that $\mu^\sigma =\mu'$ but that $c_{\mu'}>c_\mu$.  Suppose for a contradiction that there exists $\xi \in \Hom_G(\cO_L^\times, \widehat{T})$ such that $c(\mu \circ \xi)=c_\mu$ and $c(\mu' \circ \xi)=c_{\mu'}$. If $z \in \cO_L^\times$ then the Galois equivariance of $\xi$ says $$\mu \circ \xi (\sigma z) = \mu' \circ \xi(z).$$ If $z \in 1+ \fP^{c_\mu}$, then we also have $\sigma z \in 1+ \fP^{c_\mu}.$ But then $c(\mu \circ \xi) \leq c_\mu$ implies that $c(\mu \circ \xi \circ \sigma) \leq c_\mu,$ and $c(\mu' \circ \xi) = c(\mu \circ \xi \circ \sigma), $ so $c_{\mu'}= c(\mu' \circ \xi)\leq c_\mu,$ contradiction.
\end{proof}

Before moving on, we include one more auxiliary result, which will be used in section \ref{unramifiedcounting} to show that only those $c \in \N^M$ with all entries either 0 or 1 will matter for the location and order of the rightmost pole of the global generating series $Z(s)$. For more details, see Lemma \ref{UU0}.
\begin{lemma}\label{dimDkgeq1}
If $c = (c_\mu) \in \N^M$ is such that $\max c_\mu \geq 2$ and $\Pi_=(c,x) \neq 0$ then $\dim D_k(c) \geq 1$ for all $k=0, \ldots, \max_\mu c_\mu-1$.
\end{lemma}
\begin{proof}
Let us choose an ordering of the $\mu \in M$, say $\mu_1, \ldots, \mu_m$, and write $c_i = c_{\mu_i}$. We choose the ordering such that $c_1$ is maximal among $c_1,\ldots,c_m$, thus $c_1\geq 2$.  By \eqref{protoincexc} $$\Pi_=(c,x) = \sum_{d_2 \leq c_2} \cdots \sum_{d_m \leq c_m} \mu(2^{d_2})\cdots \mu(2^{d_m}) \left( \Pi_\leq(c-(0,d_2,\ldots,d_m),x) - \Pi_\leq(c-(1,d_2,\ldots,d_m),x)\right).$$ Since $\Pi_=(c,x) \neq 0$ there exists $d_2,\ldots,d_m \in \{0,1\}$ such that \begin{equation}\label{Pineq}\Pi_\leq(c-(0,d_2,\ldots,d_m),x) \neq \Pi_\leq(c-(1,d_2,\ldots,d_m),x).\end{equation} We have by Proposition \ref{MurP} that $$\Pi_\leq(c,x) = \left| \Hom_{D_\fP} \left((\cO_L/\fP)^\times, D_0(c)\right)\right| \prod_{k=1}^\infty  \mathop{\rm char}(\cO/\fp)^{\dim D_k(c)}. $$ 
Thus, since $c_1 \geq 2$ we have  \begin{equation}\label{Pifrac} \frac{\Pi_\leq(c-(0,d_2,\ldots,d_m))}{\Pi_\leq(c-(1,d_2,\ldots,d_m))} = \operatorname{char}(\cO/\fp)^{\dim D_{c_1-1}(c-(0,d_2,\ldots,d_m)) - \dim D_{c_1-1}(c-(1,d_2,\ldots,d_m))} .\end{equation} By \eqref{Pineq} the quantity in \eqref{Pifrac} is $\neq 1$. 
Since $D_k(c)$ is monotonic in $c$, we have $$ \operatorname{char}(\cO/\fp)^{\dim D_{c_1-1}(c-(0,d_2,\ldots,d_m)) - \dim D_{c_1-1}(c-(1,d_2,\ldots,d_m))}>1,$$ from which we conclude \begin{multline*} 1 \leq \dim D_{c_1-1}(c-(0,d_2,\ldots,d_m)) - \dim D_{c_1-1}(c-(1,d_2,\ldots,d_m)) \\ \leq \dim D_{c_1-1}(c-(0,d_2,\ldots,d_m))  \leq \dim D_{c_1-1}(c) \leq \dim D_k(c)\end{multline*}
 for all $1 \leq k \leq c_1-1$.\end{proof}

We spend the rest of the section devoting particular attention to the case that all of the entries of $c$ are 0 or 1. Under this condition on $c$, the groups $D_k(c)=\{1\}$ for all $k\geq 1$ by the faithfulness of $r \vert_{\widehat{T}}$.  Therefore we restrict our attention to the case $k=0$.  
We make a change of variables, and instead consider subsets $S \subseteq M$ as in the introduction.  The change of variables is given by the $G$-equivariant bijection \begin{align} \{0,1\}^m & \simeq  2^{M} \label{ctoS1} \\ c & \leftrightarrow S= \{\mu: c_\mu=1 \},\label{ctoS2}\end{align} with $G$ acting on $2^{M}$ as in the introduction. Define the quantities $\Pi_\leq(S,x)$ and $\Pi_=(S,x)$ via the above bijection $c\leftrightarrow S$ in terms of $\Pi_\leq(c,x)$ and $\Pi_=(c,x)$, and define $D(S)= D_0(c)$ as in the introduction.

 Let $\Fr \in G$ denote the Frobenius element. By Lemma \ref{reducedlemma}, it is no loss of generality to suppose that $\Fr S= S$. Define an $F$-torus $T_S$ by taking its cocharacter lattice to be $\Z^{|S^c|}$ with coordinates indexed by $\mu \in S^c$, and $G$ acting by permuting these. We define a map of $F$-tori $\alpha: T_S \to T$ by the map of cocharacter lattices $$\alpha_*: \Z^{|S^c|} \to X_*(T)$$ \begin{equation}\label{alphadef_local} (0,\ldots, 1,  \ldots, 0) \mapsto \mu,\end{equation} where the $1$ is in the $\mu$-slot and the other coordinates are all 0, see Lemma \ref{equiv_of_cat_multi} and \eqref{char_cochar_PP}. The construction of $T_S$ and $\alpha$ is compatible with a global construction that we will be introduced in section \ref{sec:localtoglobal}.
 
Let $x \in T(\cO_F)$ and take the scheme-theoretic fiber $\alpha^{-1}(x)$ of $\alpha$ above $x$. By Lemma \ref{lem:fibermodel}, the fiber $\alpha^{-1}(x)$ has an integral model over $\cO_F$, which we write $\alpha^{-1}(x)_0$. We may take the base change of $\alpha^{-1}(x)_0$ to the residue field $f$ of $F$ to obtain a separated $f$-scheme of finite type $\alpha^{-1}(x)_f$ (which may fail in general to be reduced). Let us abuse notation by writing $\alpha^{-1}(x)(f)=\alpha^{-1}(x)_f(f)$ and recall the abuses of notation $T(f)$ and $T_S(f)$ from section \ref{norm}. Note by \eqref{stdisom1} that we also have $T(f) \simeq \Hom_G(X^*(T),\ell^\times)$ and $T_S(f) \simeq \Hom_G(X^*(T_S),\ell^\times)$ since $X^*(T) \simeq X^*(\cT_f)$ with the same action by $\Gal(L/F)\simeq \Gal(\ell/f)$ and similarly for $T_S$.

 Let $\Lambda$ be a finitely-generated subgroup of $T(\cO_F)$. (Later in section \ref{unramifiedcounting}, we will take $\Lambda$ to be a finite index subgroup of the global units of a torus over a number field.) Let  $L'/F$ be a Galois extension with $L\subseteq L'$ such that all geometric components of $\alpha^{-1}(x)$ for all $x \in \Lambda$ are defined over $L'$.
 Let $p(\alpha^{-1}(x))$ be the finite set of geometric components. There is a continuous action of $\Gal(L'/F)$ on $p(\alpha^{-1}(x))$.  Let $\Fr' \in \Gal(L'/F)$ denote a Frobenius automorphism, and write \begin{equation}\label{aSx} a(S,x)= \#\{y \in p(\alpha^{-1}(x)): \Fr' y = y\}.\end{equation} 
The number $a(S,x)$ does not depend on the choice of $\Fr'$, since the inertia subgroup of $\Gal (L'/F)$ acts trivially on $\alpha^{-1}(x)$. 

\begin{lemma}\label{Acalc}
Suppose $\Fr S =S$. Then $$\Pi_\leq(S,x) = \begin{cases} \left(a(S,x) + O_{T,r}(q_F^{-1/2})\right) q_F^{\dim D(S)} & \text{ if } \dim D(S)\geq 1\\ a(S,x)  & \text{ if } D(S) \text{ is finite.}\end{cases}$$ 
\end{lemma}

\begin{proof} 
By definition of $T$ and $T_S$, we have exact sequences $$\xymatrix{1 \ar[r] &  D(S) \ar[r] & \widehat{T} \ar[r] & \widehat{T}_S},$$ and $$\xymatrix{1 \ar[r] & \Hom_G(\ell^\times, D(S)) \ar[r] & \Hom_G(\ell^\times, \widehat{T}) \ar[r] & \Hom_G(\ell^\times, \widehat{T}_S)}.$$   
By Pontryagin duality applied to $\alpha:T_{S}(f) \to T(f)$ we have an exact sequence $$\xymatrix{1 \ar[r] &  (T(f) / \alpha(T_S(f)))^\wedge \ar[r] &  T(f)^\wedge \ar[r] &  T_S(f)^\wedge}.$$
The local Langlands correspondence for tori over finite fields Proposition \ref{LanglandsFF} asserts that $$T(f)^\wedge \simeq \Hom_G(\ell^\times, \widehat{T}),$$ and likewise for $T_S$, since they both split over $L/F$, which is unramified.  Therefore we have a commutative diagram $$ \xymatrix{ 1 \ar[r] &\Hom_G(\ell^\times, D(S))\ar[r] & \Hom_G(\ell^\times, \widehat{T}) \ar[r] \ar[d]^\simeq &  \Hom_G(\ell^\times, \widehat{T}_S)\ar[d]^\simeq \\ 1  \ar[r] & (T(f) / \alpha(T_S(f)))^\wedge \ar[r] & T(f)^\wedge \ar[r] &  T_S(f)^\wedge. }$$ By the five lemma of homological algebra, we conclude that $$\Hom_G(\ell^\times, D(S)) \simeq (T(f) / \alpha(T_S(f)))^\wedge.$$ By orthogonality of characters we have 
$$\Pi_\leq(S,x) = \sum_{\xi \in \Hom_G(\ell^\times, D(S))} \chi_\xi(x) = \begin{cases} | \Hom_G(\ell^\times, D(S)) | & \text{ if } x \in \alpha(T_S(f)) \subset T(f) \\ 0 & \text{ if } x \not \in \alpha(T_S(f)).\end{cases} $$
If $x \in \alpha(T_S(f))$, then $$\Pi_\leq(S,x)=  | \Hom_G(\ell^\times, D(S)) |  = \frac{|T(f)|}{| \alpha(T_S(f))|} = \frac{|T(f)|}{| T_S(f)|}  | \ker ( \alpha: T_S(f) \to T(f))|.$$ 
In either case of $x \in \alpha(T_S(f))$ or not, we have that \begin{equation}\label{eq:PiSx_as_pointcount}\Pi_\leq(S,x)= \frac{|T(f)|}{| T_S(f)|} |\alpha^{-1}(x)(f)|.\end{equation}

We use the Lang-Weil theorem in the form of Corollary \ref{LWcor} to count the number of points on the $\alpha^{-1}(x)$ over finite fields. 
The quantities $n$, $r$, and $d$ associated to $\alpha^{-1}(x)_{f}$ as in Theorem \ref{LWthm} are bounded uniformly as $x$ varies over $T(\cO_F)$, in terms of the degree of the equations cutting out $T_S$ and $\alpha$, and $\dim T$ and $\dim r$. Since there are only finitely many possibilities for $S$ for a given $T,r$, the error term in our application of Corollary \ref{LWcor} only depends on $T,r$.

By e.g. \cite[Thm. 1.72, Def. 1.73]{MilneAGS}, 
the map $\alpha$ factors as $\alpha: T_S \to \alpha(T_S) \to T$, with the first map faithfully flat and the second a closed immersion.  For any point $x \in \alpha(T_S)$, we have by e.g. \cite[A.73]{MilneAGS} that $$ \dim \alpha^{-1}(x) + \dim \alpha(T_S)  = \dim T_S,$$ 
and by \cite[Rem. 5.42]{MilneAGS} that $$ \dim D(S) = \dim T - \dim \alpha(T_S).$$ 
Combining these, we have $$ \dim T - \dim T_S + \dim \alpha^{-1}(x) = \dim D(S).$$ Applying the Lang-Weil theorem (Corollary \ref{LWcor}) to \eqref{eq:PiSx_as_pointcount}, we conclude the lemma.
\end{proof}
In the special case that $x = 1$ we state the leading constant in Lemma \ref{Acalc} in a more convenient fashion. Let $\Fr \in G$ and \begin{equation}\label{aS}a(S) = |\{ y \in \pi_0(D(S)): \Fr y^{q_F} = y\}|.\end{equation} 
\begin{lemma}\label{AcalcClassic}
Suppose $\Fr S =S$. Then $$\Pi_\leq(S,1) = \begin{cases} a(S)q_F^{\dim D(S)} \left(1 + O_{T,r}(q_F^{-1})\right)  & \text{ if } \dim D(S)\geq 1\\ a(S)  & \text{ if } D(S) \text{ is finite.}\end{cases}$$ 
\end{lemma}
\begin{proof}
We give an alternate computation of $|\Hom_G(\ell^\times, \widehat{T})|$. 
Let $x$ denote a generator for the cyclic group $\ell^\times$.  Then $\Hom_G(\ell^\times, D(S))$ is in bijection with the set $\{z \in D(S) : \Fr z = z^{q_F}\}$ of possible images of $x$ in $D(S)$.  This set is equal to the kernel $J$ of the $G$-equivariant homomorphism $$D(S) \to D(S)$$ given by $$ z \mapsto \frac{z^{q_F}}{\Fr z}.$$ We have an exact sequence of $G$-modules $$\xymatrix{ X^*(D(S)) \ar[r]^\varphi &  X^*(D(S)) \ar[r] &  X^*(J) \ar[r] &  1}.$$ The map $\varphi$ is given by $\varphi(\chi) = q_F \chi - \chi^{\Fr}$,  where we have written $X^*(D(S))$ in additive notation.  Our goal is to compute the cardinality of $X^*(J)$, which equals the cardinality of $J$ itself.

Write $X=X^*(D(S)) $, $X_t$ for the torsion subgroup, and $X_f = X/X_t$.  The map $\varphi: X \to X$ induces maps $X_t \to X_t$ and $X_f \to X_f$, both of which we also denote $\varphi$. We write $Q=X^*(J)$ for the cokernel of $\varphi:X \to X$, $Q_t$ for the cokernel of $\varphi: X_t \to X_t$, and $Q_f$ for the cokernel of $\varphi: X_f \to X_f$.  In summary, we have a commutative diagram $$\xymatrix{ 1 \ar[r] & X_t \ar[r] \ar[d]^\varphi & X \ar[r] \ar[d]^\varphi & X_f \ar[r] \ar[d]^\varphi & 1 \\
1 \ar[r] & X_t \ar[r] \ar[d] & X \ar[r]^{\pi} \ar[d] & X_f \ar[r] \ar[d]^{q} & 1 \\
 & Q_t \ar[r] & Q \ar[r]^{\pi_Q} & Q_f & }$$ 
 
 The map $\pi_Q$ is surjective since $\pi$ and $q$ are both surjective.  
 
 We show that the top right $\varphi$ is injective.  Indeed, let $\chi\in X_f$ satisfy $\varphi(\chi)=0$.  Since $\varphi$ is $G$-equivariant, we also have $\varphi(\chi^{\Fr^i}) = 0$, for all $ i $.  Since $\varphi(\chi)=0$ we have $q_F\chi = \chi^{\Fr}$, and so $\chi^{\Fr}\equiv 0 \mods {q_F}$.  But similarly, since $\varphi(\chi^{\Fr})=0$ we have that $\chi^{\Fr^{2}} \equiv  0 \mods{q_F^2}$.  Therefore $\chi \equiv 0 \mods{q_F^{|G|}}$.  Repeating this process ad infinitum, we conclude that $\chi=0 \in X_f$, so the top right $\varphi$ is injective. 
 
 Then by the snake lemma we have that $Q_t \hookrightarrow Q$, and so the bottom row of the diagram forms an exact sequence of finitely-generated abelian groups.  We have that $Q$ is finite if both $Q_t$ and $Q_f$ are, and in this case $|Q| = |Q_t||Q_f|$.  
   
Let us begin with $Q_f$. The map $\varphi$ on $X_f$ is given in matrices by $q_F I-A$, where $A$ is some matrix of integers for which $A^{|G|}=I$.  
Putting $q_F I- A$ in Smith normal form $q_FI-A = UDV$ with $U,V \in \GL_{\dim D(S)}(\Z)$, we have $$Q_f \simeq \frac{\Z}{d_1\Z} \times \cdots \times \frac{\Z}{d_{\dim D(S)} \Z},$$ with each $d_i$ is equal to $q_F \pm 1$, and so $|Q_f| = q_F^{\dim D(S)}(1+O(q_F^{-1}))$ if $\dim D(S) \geq 1$ and $|Q_f| = 1$ if $\dim D(S) = 0$.

Now we compute $|Q_t|$. For any endomorphism of a finite abelian group $f :A \to A$, we have that $|\ker f | = |\coker f|$.  Let $J_t$ be the kernel of $\varphi: X_t \to X_t$, which therefore has the same cardinality as $Q_t$.  But the cardinality of $J_t$ is exactly the quantity $a(S)$ defined above the statement of the lemma. 

We have shown that $|Q_t| = a(S)$ and $$|Q_f| = \begin{cases} q_F^{\dim D(S)} (1+O(q_F^{-1})) & \text{ if } \dim D(S) \geq 1\\ 1 & \text{ if } \dim D(S)=0.\end{cases}$$ Since $|Q| = |Q_t||Q_f|$, we conclude the lemma. 
\end{proof}

Finally, we apply the foregoing results on $\Pi_\leq(S,x)$ to derive the final results for $\Pi_=(S,x)$. Let $$\mu(i \in T) = \begin{cases} 1 & \text{ if } i \not \in T \\ -1 & \text{ if } i \in T.\end{cases}$$ The main tool is \eqref{protoincexc}, which we re-state for the sets $S$ as 
\begin{equation}\label{inclusionexclusion} \Pi_=(S,x) = \sum_{T\subseteq S} \mu(1 \in T)\cdots \mu(m \in T) \Pi_\leq(S - T,x).\end{equation} For a set $S$ we denote \begin{equation}\label{Sred}S_{\text{red}} = \{ \mu \in S : \sigma \mu \in S \text{ for all } \sigma \in G\}.\end{equation} The set $S_\text{red}$ is now $G$-fixed, and by \eqref{inclusionexclusion} and Lemma \ref{reducedlemma} we have \begin{equation}\label{Ared} \Pi_{\leq}(S,x) = \Pi_{\leq}(S_{\text{red}},x).\end{equation}
Recall from \eqref{Adef} that for faithful $r \vert_{\widehat{T}}$ we defined $A$ by 
$$A= \max\Big\{\frac{\dim D(S) + 1}{|S|} : S \subseteq M, D(S)\neq \{1\}\Big\}.$$

\begin{lemma}\label{A*calc} For any $\varnothing \neq S \subseteq M$ such that $\Fr S = S,$ and $$\frac{\dim D(S)+1}{|S|} \geq A,$$ we have \begin{equation}\label{eq:A*calc_statement}\Pi_=(S,x) = \begin{cases} \left(a(S,x)+O_{T,r}(q_F^{-1/2})\right)q_F^{\dim D(S)}& \text{ if }  \dim D(S)\geq 1 \\ 
a(S,x)  -1 & \text{ if }  \dim D(S) = 0 \text{ and } D(S)\neq \{1\} \\
0 & \text{ if } D(S)=\{1\}.
 \end{cases}\end{equation} 
 If $x=1$ then instances of $a(S,1)$ in \eqref{eq:A*calc_statement} may be replaced by $a(S)$.\end{lemma}
\begin{proof}
Suppose first that $\dim D(S) = 0$ and $a(S,x) \neq 0$.  Then for any $T \subseteq S$ we also have $ \dim D(T) = 0 $ since $D(T) \subseteq  D(S) $. By \eqref{inclusionexclusion} and \eqref{Ared} we have $$\Pi_{=}(S,x) = \sum_{T \subseteq S} \mu(1 \in T)\cdots\mu(m \in T) \Pi_{\leq}((S-T)_\text{red},x).$$ Since $a(S,x) \neq 0$, we have by Lemma \ref{Acalc} and Proposition \ref{MurP} that $\chi_\xi(x) = 1$ for all $\xi \in P_\leq(S)$. Then for any $S'\subseteq S$ we also have $\chi_\xi(x) =1$ for all $\xi \in P_{\leq}(S')$. 
Using Proposition \ref{MurP} and Lemma \ref{Acalc} again, we have $$\Pi_{=}(S,x)= a(S,x) +  \sum_{\varnothing \neq T \subseteq S} \mu(1 \in T)\cdots\mu(m \in T) a((S-T)_\text{red},x).$$
For any $\varnothing \neq T \subsetneq S$ we have $$\frac{\dim D(T) + 1}{|T|} = \frac{1}{|T|} > \frac{1}{|S|}= \frac{\dim D(S) + 1}{|S|}\geq A,$$ thus $D(T) = \{1\}$ by definition of $A$.   For all $T \neq \varnothing$, we have $(S-T)_{\text{red}} \subsetneq S$.  Thus, if $(S-T)_{\text{red}} \neq \varnothing$, then we have $D((S-T)_\text{red})=\{1\}$. On the other hand, if $(S-T)_{\text{red}} = \varnothing$, then we also have $D((S-T)_\text{red})=\{1\}$ by the faithfulness of $r\vert_{\widehat{T}}$.  Therefore \begin{align*} 
\Pi_{=}(S,x) & =  a(S,x) +  \sum_{\varnothing \neq T \subseteq S} \mu(1 \in T)\cdots\mu(m \in T) \\
& =  a(S,x) -1 +  \sum_{ T \subseteq S} \mu(1 \in T)\cdots\mu(m \in T) \\
& =  a(S,x) -1,\end{align*} since we assumed $S \neq \varnothing$. Note that $a(S,1)=a(S)$ by Lemmas \ref{Acalc} and \ref{AcalcClassic} when $\dim D(S)=0$.

Now suppose that $\dim D(S) \geq 1$ and $a(S,x) \neq 0$.  As above, we have by \eqref{inclusionexclusion}, \eqref{Ared}, Proposition \ref{MurP} and Lemma \ref{Acalc} that \begin{align}
\Pi_=(S,x) & =  \sum_{T \subseteq S} \mu(1 \in T)\cdots\mu(m \in T) \Pi_\leq((S-T)_\text{red},x) \nonumber \\
& =  q_F^{\dim D(S)}(1+ O(q_F^{-1/2}))  \sum_{\substack{T \subseteq S \\ \dim D((S-T)_\text{red}) = \dim D(S)}}  \mu(1 \in T)\cdots \mu(m \in T) a((S-T)_\text{red},x) .\label{eq:A*calc1}\end{align} If $x=1$ we may use Lemma \ref{AcalcClassic} in lieu of Lemma \ref{Acalc} to obtain \eqref{eq:A*calc1} with $a((S-T)_\text{red})$ in lieu of $a((S-T)_\text{red},1)$.

Suppose that $T$ is such that $\dim D( (S-T)_\text{red}) = \dim D(S)$.  If $(S-T)_\text{red}= \varnothing$, then $$1 \leq \dim D(S) = \dim D( (S-T)_\text{red}) = 0,$$ which contradicts the faithfulness of $r \vert_{\widehat{T}}$.  Therefore we may assume that $(S-T)_\text{red}\neq \varnothing$.  If $T \neq \varnothing$ then $$ \frac{\dim D((S-T)_\text{red}) + 1}{|(S-T)_\text{red}|} = \frac{ \dim D(S) + 1}{|(S-T)_\text{red}|} > \frac{\dim D(S) +1 }{|S|} \geq A.$$
Therefore $D((S-T)_\text{red}) =\{1\}$, and this is a contradiction with $\dim D( (S-T)_\text{red}) = \dim D(S)$.  Thus, the only $T\subseteq S$ which satisfies $\dim D( (S-T)_\text{red}) = \dim D(S)$ is $T=\varnothing$, from which we conclude the statement in the lemma.

Now suppose $a(S,x)=0$, in particular $x\neq 1$. If $\dim D(S)\geq 1$ then by \eqref{inclusionexclusion}, Lemma \ref{Acalc} and the triangle inequality, the statement of the lemma holds. 

To finish the proof of lemma, it remains to consider the case that $\dim D(S) = 0$.  If $D(S) = \{1\}$, then we must have $a(S,x)\neq 0$, so  
suppose that $\dim D(S) = 0$ and $D(S) \neq \{1\}$.  Suppose $|S|\geq 2$ and $S$ is maximal such that $\dim D(S) = 0$ and $D(S) \neq \{1\}$. There exists $\mu \not \in S$, since $D(M)= \widehat{T}$.  We claim that $\dim D(S \cup \mu) \geq 1$.  Indeed, by maximality, either $\dim D(S \cup \mu) \geq 1$ or $D(S \cup \mu)=\{1\}$. But the second of these can't happen since $D(S) \neq \{1\}$ already, and $D(\cdot)$ is monotonic. So $\dim D(S \cup \mu) \geq 1$.  But then $$ \frac{\dim D(S \cup \mu) + 1}{|S \cup \mu|} \geq \frac{2}{|S|+1} > \frac{1}{|S|} \geq A,$$ since $|S|\geq 2$. This is a contradiction with the definition of $A$. Hence, $|S|=1$. Then we have \begin{align*} \Pi_=(S,x) & =  \sum_{T \subseteq S} \mu(1 \in T) \cdots \mu(m \in T)\Pi_\leq((S-T)_\text{red},x) \\ & =  \Pi_\leq(S,x) -\Pi_\leq(\varnothing,x) \\ & =  -1.\end{align*} Here, $\Pi_\leq(S,x)=0$ by the assumption that there exists $\xi \in P_\leq(S)$ with $\chi_\xi(x)\neq1$, and $\Pi_\leq(\varnothing,x) = 1$ since only the trivial character appears in the definition of $\Pi_\leq (\varnothing, x)$.
\end{proof}

 \subsection{Local conductor zeta function, ramified case}\label{ramified}
In this section, $T$ is a torus over a non-archimedean local field $F$ and splitting over a finite Galois extension $L$ with ramification index $e=e_{L/F}$ not necessarily equal to $1$. Let $G=\Gal(L/F)$ be the corresponding group. Let $\fp, \fP$ be the maximal ideals of $\cO_F,\cO_L$ and choose a uniformizer $\pi$ of $\cO_F$. Let $NT(\cO_L)$ be the image of the norm map $N: T(\cO_L) \to T(\cO_F)$. Recall the definition of $N_F(s,x)$ from \eqref{NFsxdef}. 

\begin{theorem}\label{locaramthm} Suppose that $r \vert_{\widehat{T}}$ is faithful. 
For any $x \in NT(\cO_L)$ and character $\theta$ of $T(F)$ that is trivial on $NT(\cO_L)$, the series $N_F(s,x)$  
converges absolutely and uniformly on compacta in the region $$ \real(s) > \max \{ \frac{\dim D(S)}{|S|}: D(S)\neq \{1\}\},$$ where $S$ and $D(S)$ are as in section \ref{sec:Lgroupsandreps} and formula \eqref{DS}.
\end{theorem}

\begin{proof}
By Corollary \ref{MTcond}, Corollary \ref{ACcor}, and Proposition \ref{MP} we have \begin{align}\label{R}   \left| N_F(s,x)\right| & \leq \sum_{\chi \in NT(\cO_L)^\wedge} \frac{1}{|\tilde{\fc}(\chi\theta,r)|^{\real(s)/e}} \nonumber \\
& = \sum_{\chi \in NT(\cO_L)^\wedge} \frac{1}{|\tilde{\fc}(\chi,r)|^{\real(s)/e}} \nonumber \\
& \leq  (\delta(\psi) dx/dx')^m \sum_{ \xi \in \Hom_G(\cO_L^\times,\widehat{T})} q_F^{-\real(s)\sum_\mu c(\mu \circ \xi)/e}. \end{align}
The conductor $c$ appearing in the last line of \eqref{R} is in fact $c = c_\mathcal{U}$ as in Definition \ref{filtcond}, where $\mathcal{U}$ is the standard filtration on $\cO_L^\times$ (see Definition \ref{modifiedArtin}). Next, we construct yet another filtration and compare it to $\mathcal{U}$. 

By the normal basis theorem, there exists an element $\alpha \in L$ such that $\{\alpha^g: g \in G\}$ is a basis for $L/F$.  The $\{\alpha^g\}$ all have the same valuation (e.g. \cite[Ch.2 Cor 3]{SerreLF}), so by clearing numerators or denominators, there exists $\beta \in \cO_L^\times$ such that $\{\beta^g :g \in G\}$ is a basis for $L$.  We define an injective map of $\cO_F[G]$-modules $$f: \cO_F[G] \hookrightarrow \cO_L$$ by $f(1)=\beta$.  Its image has finite index in $\cO_L$ since $\{\beta^g\}$ span $L$.  

Let $\nu \geq 1$ be sufficiently large so that the $\fP$-adic exponential function $$\exp: \fP^{e\nu} \to 1 + \fP^{e\nu}$$ is well-defined and an isomorphism.  Then let $g: \cO_F[G] \hookrightarrow \cO_L^\times$ be defined as the composition of the following sequence of injective $\cO_F[G]$-module homomorphisms $$\xymatrix{g: \cO_F[G] \ar@{^{(}->}[r]_-f & \cO_L \ar[r]^-\simeq_-{\times \pi^\nu} & \fP^{e\nu} \ar[r]^-\simeq_-\exp & 1+ \fP^{e\nu} \ar@{^{(}->}[r] & \cO_L^\times}.$$ The homomorphism $g$ has finite cokernel.   
Let $V^n = g(\fp^n \cO_F[G])$ and $\mathcal{V} = (V^n)$ be the corresponding filtration of $\cO_L^\times$. We have for all $n\geq 0$ that \begin{equation}\label{VU}V^n \subseteq \cO_L^{(e \nu + en)}.\end{equation} 
Indeed, if $x \in \fp^n\cO_F[G]$ then we write $x$ as $$x=\sum_{g \in G} a_g g$$ with $a_g \in \fp^n$ for all $g \in G$.  So we have $$f(x) = \sum_{g \in G} a_g \beta^g \in \fp^n = \fP^{en},$$ so that $f(\fp^n\cO_F[G]) \subseteq \fP^{en}$. 

Now let us consider the conductor $c_\mathcal{V}$ of a character $\chi$ of $\cO_L^\times$ with respect to the filtration $\mathcal{V}$ (see Definition \ref{filtcond}) and compare $c_\mathcal{U}$ and $c_\mathcal{V}$. 
If $\chi \vert_{\cO^{(n)}_L} =1$ then $c_\mathcal{U}(\chi) \leq n$, and if $\chi \vert_{\cO^{(n)}_L} \neq 1$ then $c_\mathcal{U}(\chi) \geq n+1$.  Similarly, if $\chi \vert_{V^n} =1$ then $c_\mathcal{V}(\chi) \leq n$, and if $\chi \vert_{V^n} \neq 1$ then $c_\mathcal{V}(\chi)\geq n+1$. Therefore by \eqref{VU} we have $$c_\mathcal{U}(\chi) \geq e c_\mathcal{V}(\chi) + e \nu - e +1.$$

Then we have $$\sum_{ \xi \in \Hom_G(\cO_L^\times,\widehat{T})} q_F^{-{\real(s)} \sum_\mu c_\mathcal{U}(\mu \circ \xi)/e} \leq q_F^{-{\real(s)} m (\nu - 1 +1/e)} \sum_{\xi \in \Hom_G(\cO_L^\times, \widehat{T})} q_F^{-{\real(s)}  \sum_\mu c_\mathcal{V}(\mu \circ \xi)}.$$ 
But now the summand only depends on the restriction of $\xi$ to $V^0$.  We have an exact sequence $$ \xymatrix{ 1 \ar[r] & \Hom_G(\cO_L^\times / V^0, \widehat{T})  \ar[r] & \Hom_G(\cO_L^\times, \widehat{T}) \ar[r] & \Hom_G(V^0,\widehat{T}) \ar[r] & \cdots}$$ 
The kernel is a finite group, since $\cO_L^\times / V^0$ is a finite group and $\widehat{T}$ only has finitely many points of order dividing the cardinality of this group.  
Thus we have $$ \sum_{\xi \in \Hom_G(\cO_L^\times, \widehat{T})} q_F^{-{\real(s)} \sum_\mu  c_\mathcal{V}(\mu \circ \xi)} \leq  \left|  \Hom_G(\cO_L^\times / V^0, \widehat{T})  \right| \sum_{\xi \in \Hom_G(V^0, \widehat{T})} q_F^{-{\real(s)}  \sum_\mu c_\mathcal{V}(\mu \circ \xi)}.$$
But also $$\Hom_G(V^0, \widehat{T}) \simeq \Hom_G(\cO_F[G],\widehat{T}) \simeq \Hom(\cO_F,\widehat{T}).$$ If $\xi \leftrightarrow \tau \in \Hom(\cO_F,\widehat{T})$ across this isomorphism, then $c_\mathcal{V}(\mu \circ \xi ) = c_\mathcal{W}(\mu \circ \tau),$ where the latter is the conductor with respect to the filtration $\mathcal{W} = (\fp^n)$ of the additive group $\cO_F$.
  
Therefore $$\frac{|N_F(s,x) |}{(\delta(\psi) dx/dx')^m} \leq q_F^{-{\real(s)} m (\nu - 1 +1/e)}  \left|  \Hom_G(\cO_L^\times / V^0, \widehat{T})  \right| \sum_{\tau \in \Hom(\cO_F, \widehat{T})} q_F^{-{\real(s)} \sum_\mu c_\mathcal{W}(\mu \circ \tau)}.$$
Then one computes in a similar fashion as section \ref{unramified}. Let (cf.\ \eqref{Pleqc}) $$\Pi_\leq (c) =| \{\tau \in \Hom(\cO_F, \widehat{T}) : c_\mathcal{W}(\mu \circ \tau) \leq c_\mu \text{ for all } \mu \in M\} |. $$ Recall the complex diagonalizable groups $D_k(c)$ from \eqref{Dkc}. If $r \vert_{\widehat{T}}$ is faithful and $c\in \N^M$ is $G$-fixed then \begin{align*} \Pi_\leq (c)  & = \prod_{k=0}^\infty \left| \Hom(\fp^{k-1}/\fp^k,D_k(c))\right| \\
& = \prod_{k=0}^\infty  q_F^{\dim D_k(c)}.
\end{align*} By the faithfulness of $r \vert_{\widehat{T}}$, the product is actually a finite product.

We have therefore that $$\frac{|N_F(s,x)|}{(\delta(\psi) dx/dx')^m} \leq q_F^{-{\real(s)} m (\nu - 1 +1/e)}  \left|  \Hom_G(\cO_L^\times / V^0, \widehat{T})  \right| (1-q_F^{-{\real(s)/e}})^m \sum_{c \in \N^M } \frac{\prod_{k=0}^\infty  q_F^{\dim D_k(c)}}{q_F^{{\real(s)}  |c|}}.$$ 
Therefore, $R_v(s,x)$ converges absolutely and uniformly on compacts for all $$ \real(s) >  \limsup_{c \in \N^M} |c|^{-1} \sum_{k\geq 0} \dim D_k(c).$$
Now, recall that for any positive real $a,a', b,b'$, with say $a/b< a'/b'$ we have $a/b \leq (a+a')/(b+b') \leq a'/b'$. For any $c \in \N^M$ we write
$$|c| = \sum_{\mu \in M} c_\mu = \sum_{k \geq 0} \#\{ \mu: c_\mu >k\}.$$
Then,
$$|c|^{-1} \sum_{k\geq 0 } \dim D_k (c) = \frac{\sum_{k\geq  0 } \dim D_k(c) }{ \sum_{k\geq 0} \#\{ \mu: c_\mu >k\}} \leq \max_{k\geq 0} \frac{ D_k(c)}{ \#\{ \mu: c_\mu >k\}} = \max_{k\geq 0} \frac{D(\{ \mu: c_\mu >k\})}{ \#\{ \mu: c_\mu >k\}}.$$
Thus, $$\limsup_{c \in \N^M} |c|^{-1} \sum_{k\geq 0} \dim D_k(c) \leq   \max \{ \frac{\dim D(S)}{|S|}: D(S)\neq \{1\}\}.$$ 
\end{proof}

\section{Local archimedean theory}\label{archimedean_top}
\subsection{Local Langlands correspondence, local conductors}\label{archimedeanC}
We assume in this section that $F, L$ are \emph{archimedean} local fields, and $T$ is an $F$-torus splitting over $L$. 
Let $\Gamma_\R(s) = \pi^{-s/2}\Gamma(s/2)$ and $\Gamma_\C(s) = 2(2\pi)^{-s}\Gamma(s)$. If $(\rho,V)$ is a complex Galois representation of the group $W_F$, then the $L$ factor of $(\rho,V)$ is of the form $$ L(s,V) = \prod_i \Gamma_{F_i}(s+ \kappa_{\rho,i}),$$ where each $F_i= \R,$ or $\C$,  $\dim V = \sum_i [F_i:F]$, and $\kappa_{\rho,i} \in \C$, see \cite[\S 3.7]{DeligneAntwerpII}. Recall the discussion of the local Langlands correspondence for tori from section \ref{sec:LLC}.   
\begin{definition}\label{archdef}
Suppose $F$ is an archimedean local field. If $\varphi$ is the Langlands parameter associated to $\chi \in \Hom(T(F), \C^\times)$ by \eqref{LLCforTori}, then the quantity 
$$\fc (\chi, r) = \prod_i (|\kappa_{r \circ \varphi,i}|+1)^{[F_i: \R]}(\delta(\psi) dx/dx')^{\dim r}$$ is called the \emph{archimedean local analytic conductor} of $\chi$ with respect to a finite-dimensional complex representation $r$ of $\LT$.
\end{definition}
Note that our definition differs slightly from the standard definition in that the $+1$ above is sometimes replaced by a $+2$ or $+3$.  The reason some authors prefer $+2$ or $+3$ is to ensure that as $\dim V$ varies, there are only a finite number of representations of bounded conductor. Since we will always consider $\dim r$ to be fixed in this paper, we prefer $+1$ as it makes some of the computations in section \ref{archimedean} more elegant. 

The definition of the $L$-factor $L(s,V)$ for archimedean places is given in \cite{tateNTB} sections 3.1.1, 3.1.2 and 3.3.1 in terms of the classification of the finite dimensional irreducible representations of $W_F$ given in loc.\ cit.\ section 2.2.2. Therefore, we must make explicit parameterizations of the possible Langlands parameters $\varphi: W_{L/F} \to \LT$, as well as the possible representations $r: \LT \to \GL(V)$, in order to be able to find their compositions among the classification \cite[\S2.2.2]{tateNTB}. 

We assume until further notice that $F = \R$ and briefly discuss the easier case that $F= \C$ at the end of this section. 

An $F$-torus splits over a quadratic extension $L\simeq \C$,  and so $G = \Gal(L/F)$ is a group of order two, whose elements we write $\{1,\tau\}$. Let us recall the explicit description of the Weil groups for archimedean local fields. We have $$W_F = L^\times  \sqcup L^\times j,$$ and $$W_L = L^\times,$$ 
where we write elements of $W_F$ as words in $z$ and $j$ and we have the rules $jzj^{-1}=\tau z$ and $j^2=-1$.  We also have $$\pi: W_{F/F} = W_F^\text{ab} \simeq F^\times$$ where the isomorphism $\pi$ is given by $$\pi(z) = |z|^2, \quad \text{ and } \quad \pi(j) = -1.$$  Also note that  $$W_{L/F} = W_F.$$ 

Now we choose isomorphisms $L\simeq \C$ and \begin{equation}\label{TRisom}T(F) \simeq \Tt= (\R^\times)^{n_1}\times (S^1)^{n_2} \times (\C^\times)^{n_3}\end{equation} so that $\dim T= n=n_1+n_2+2n_3.$   
By computing with the inflation-restriction exact sequence and using facts about the group cohomology of finite cyclic groups, we have an explicit parameterization of the $L$-equivalence classes of Langlands parameters $\varphi: W_{L/F} \to \LT_u$. They are given by \begin{multline}\label{LPparameterization}\varphi(z)  = \Bigg( |z|^{w_1}, \ldots, |z|^{w_{n_1}}, \left(\frac{|z|}{z}\right)^{\alpha_1},\ldots, \left(\frac{|z|}{z}\right)^{\alpha_{n_2}}, \\ (\left(\frac{|z|}{z}\right)^{\alpha'_1}|z|^{w'_1-\alpha'_1},\left(\frac{|z|}{z}\right)^{-\alpha'_1}|z|^{w'_1-\alpha'_1}), \ldots, (\left(\frac{|z|}{z}\right)^{\alpha'_{n_3}}|z|^{w'_{n_3}-\alpha'_{n_3}},\left(\frac{|z|}{z}\right)^{-\alpha'_{n_3}}|z|^{w'_{n_3}-\alpha'_{n_3}})\Bigg)\rtimes 1\end{multline} and $$\varphi(j) = \left( (-1)^{\epsilon_1}, \ldots,(-1)^{\epsilon_{n_1}},1 \ldots, 1,(1,(-1)^{\alpha_1'}),\ldots,(1,(-1)^{\alpha_{n_3}'})\right) \rtimes \tau.$$  Here, $w_i \in i\R$, $\epsilon_i \in \{0,1\}$, $\alpha_i \in\Z$, $\alpha'_i \in \Z$, and $w'_i \in \C$ such that $w'_i-\alpha'_i\in i\R $. We write \begin{equation}\label{Ttwedgedef}\Tt^\wedge = (i\R^{n_1} \times (\Z/2\Z)^{n_1}) \times \Z^{n_2} \times (i\R^{n_3} \times \Z^{n_3}),\end{equation} so that Langlands parameters may be parametrized by $((w,\epsilon),\alpha, (w',\alpha')) \in \Tt^\wedge$.

We also need an explicit description of the representation $r$. The representation $r$ decomposes into irreducible representations, and we can parameterize all irreducible representations of $\LT$ by the set of orbits $G \backslash X^*(\widehat{T})$ using Mackey theory (see \cite[\S8.2]{SerreLinearReps}). We now study this parameterization explicitly. Corresponding to \eqref{TRisom} we have an isomorphism of $G$-modules \begin{equation}\label{X*Tisom} X^*(\widehat{T}) \simeq X = X_*(\G_m)^{n_1} \times X_*(S^1)^{n_2} \times X_*(\res_{\C/\R} \G_m)^{n_3}\end{equation} $$ \mu_x \leftrightarrow x,$$ where $X_*(\G_m) = \Z$ with $G$ acting trivially, $X_*(S^1) = \Z$ with $\tau \in G$ acting by sending $-1$ to $1$, and $X_*(\res_{\C/\R} \G_m)= \Z^2$ with $\tau \in G$ acting by swapping the two factors.  Each $x \in X$ is contained in a $G$-orbit of size $1$ or $2$. We have the following three types of isomorphism classes of irreducible representations of $\LT$:
\begin{enumerate}
\item[1a)] If $x$ is fixed by $G$, i.e.~is in an orbit of size one, then $\mu_x$ is an irreducible representation of $\LT$.  
\item[1b)] If $x$ is fixed by $G$, i.e.~is in an orbit of size one, then $\mu_x \otimes (\text{sign})$ is an irreducible representation of $\LT$.
\item[2)] If $x$ is \emph{not} fixed by $G$, i.e.~is in an orbit of size two, then $V_x=\Ind_{\widehat{T}}^{\LT} \mu_x$ is an irreducible representation of dimension two of $\LT$.  It only depends on the orbit of $x$.  That is, $V_x \simeq V_{\tau x}$ 
 and this representation is not isomorphic to any other $V_{x'}$, $x'\neq x,\tau x$.
\end{enumerate}
Therefore we get a decomposition $$r \vert_{\LT} = \bigoplus_{i=1}^{m_1} \mu_{x_i} \oplus \bigoplus_{i=1}^{m_2} \left(\mu_{x'_i}\otimes (\text{sign})\right) \oplus \bigoplus_{i=1}^{m_3} V_{x''_i},$$ for some $x_i,x_i'$ which are fixed by $G$ and some $x''_i $ which are not fixed by $G$, and where $m_1+m_2+2m_3=m = \dim r$.

To work out the archimedean $L$-factor for each Langlands parameter $\varphi$ (as in \eqref{LPparameterization}) and each irreducible representation of $\LT$ (as in (1a),(1b),(2), above), we must compute these representations of $W_F$ explicitly enough to be able to recognize them in the classification of irreducible representations given in \cite[\S2.2.2]{tateNTB}.
\begin{enumerate} 
\item[1a)] Suppose $x$ is fixed by $G$.  Then we have \begin{equation}\label{xexplicit}x = (a_1,\ldots,a_{n_1},0,\ldots,0,(b_1,b_1),\ldots,(b_{n_3},b_{n_3})).\end{equation} The representation of $W_{L/F}$ associated to $\mu_x,\varphi$ is $$(\mu_x \circ \varphi)(z) = \prod_{i=1}^{n_1} |z|^{a_iw_i} \prod_{i=1}^{n_3} |z|^{2b_i(w_i'-\alpha'_i)}= (\pi(z))^{\frac{1}{2}\sum_{i=1}^{n_1}a_iw_i + \sum_{i=1}^{n_3} b_i(w_i'-a_i')} $$ $$ (\mu_x \circ \varphi)(j) = (-1)^{\sum_{i=1}^{n_1} a_i \epsilon_i + \sum_{i=1}^{n_3} \alpha_i'b_i} = (\pi (j))^{\sum_{i=1}^{n_1} a_i \epsilon_i+ \sum_{i=1}^{n_3} \alpha_i'b_i}.$$ As a character of $F^\times$ this is $$ \left(\frac{|y|}{y}\right)^{{\sum_{i=1}^{n_1} a_i \epsilon_i+ \sum_{i=1}^{n_3} \alpha_i'b_i }} |y|^{\frac{1}{2}\sum_{i=1}^{n_1}a_iw_i + \sum_{i=1}^{n_3} b_i(w_i'-\alpha_1')}. 
$$ Following \cite[\S3.1.1]{tateNTB}, the $L$-function of this character of $W_{L/F}$ is $$L(s,\mu_x\circ \varphi) = \Gamma_{\R} \left( s + \frac{1}{2}\sum_{i=1}^{n_1}a_iw_i + \sum_{i=1}^{n_3} b_i(w_i'-\alpha_i')+\left(\sum_{i=1}^{n_1} a_i \epsilon_i + \sum_{i=1}^{n_3} \alpha_i'b_i\mods 2\right)\right).$$ 
Here and below, by $(n \mods 2)$ we mean the integer $0$ or $1$ according to the value of $n$ modulo $2$. 
\item[1b)]  Suppose $x $ is fixed by $G$.  Then $x$ is as in \eqref{xexplicit}, and 
we have the characters of $W_{L/F}$ $$(\mu_x \otimes (\text{sign})\circ \varphi)(z) = \prod_{i=1}^{n_1} |z|^{a_iw_i} \prod_{i=1}^{n_3} |z|^{2b_i(w_i'-\alpha'_i)}= (\pi(z))^{\frac{1}{2}\sum_{i=1}^{n_1}a_iw_i + \sum_{i=1}^{n_3} b_i(w_i'-\alpha_1')} $$ $$ (\mu_x \otimes (\text{sign}) \circ \varphi)(j) = -(-1)^{\sum_{i=1}^{n_1} a_i \epsilon_i + \sum_{i=1}^{n_3} \alpha_i'b_i } = (\pi (j))^{1+\sum_{i=1}^{n_1} a_i \epsilon_i + \sum_{i=1}^{n_3} \alpha_i'b_i}.$$ As a character of $F^\times$ this is $$\left(\frac{|y|}{y}\right)^{{1+\sum_{i=1}^{n_1} a_i \epsilon_i + \sum_{i=1}^{n_3} \alpha_i'b_i }} |y|^{\frac{1}{2}\sum_{i=1}^{n_1}a_iw_i + \sum_{i=1}^{n_3} b_i(w_i'-\alpha_1')}. 
$$
Following \cite[\S 3.1.1]{tateNTB}, the $L$-function of this character of $W_{L/F}$ is \begin{multline*} L(s,\mu_x \otimes (\text{sign}) \circ \varphi) \\ = \Gamma_{\R} \left( s + \frac{1}{2}\sum_{i=1}^{n_1}a_iw_i + \sum_{i=1}^{n_3} b_i(w_i'-\alpha_i')+\left(1 + \sum_{i=1}^{n_1} a_i \epsilon_i + \sum_{i=1}^{n_3} \alpha_i'b_i \mods 2\right)\right).\end{multline*} 
\item[2)] Suppose $x$ is \emph{not} fixed by $G$.  Then we have $$x = \left(a_1,\ldots, a_{n_1},c_1,\ldots, c_n,(b_1,b_1'), \ldots,(b_{n_3},b_{n_3}')\right),$$ where at least one of the $c_i \neq 0$ or one of the $b_i \neq b_i'$.   
Then we have $$(V_x \circ \varphi)(z) = \mx{(\mu_x \circ \varphi)(z)}{0}{0}{(\mu_{\tau x} \circ \varphi)(z)} $$ $$ (V_x \circ \varphi)(j) = \mx{(\mu_x \circ \varphi)(j)}{0}{0}{(\mu_{\tau x} \circ \varphi)(j)}\mx{0}{1}{1}{0}.$$ In order to find this representation in the classification of \cite[\S2.2.2]{tateNTB}, we must recognize it as the induction of some character.  
We have 
$$(\mu_x \circ \varphi)(z)  
=   \prod_{i=1}^{n_1} |z|^{a_iw_i} \prod_{i=1}^{n_2} z^{-\alpha_ic_i}|z|^{\alpha_ic_i} \prod_{i=1}^{n_3} z^{-\alpha_i'(b_i-b_i')}|z|^{w_i'(b_i+b_i')- 2b_i'\alpha_i'} $$
and
$$(\mu_{\tau x} \circ \varphi)(z) =  \prod_{i=1}^{n_1} |z|^{a_iw_i} \prod_{i=1}^{n_2} z^{\alpha_ic_i}|z|^{-\alpha_ic_i} \prod_{i=1}^{n_3} z^{\alpha_i'(b_i-b_i')}|z|^{w_i'(b_i+b_i')- 2b_i\alpha_i'} $$ 
The power of $z$ in $(\mu_x \circ \varphi)(z)$ is $$- \sum_{i=1}^{n_2} \alpha_i c_i - \sum_{i=1}^{n_3} \alpha'_i(b_i - b_i')$$ and the power of $z$ in $(\mu_{\tau x} \circ \varphi)(z)$ is $$ \sum_{i=1}^{n_1}\alpha_i c_i + \sum_{i=1}^{n_3} \alpha'_i(b_i - b_i').$$  Exactly one of these two is negative.  Following Tate, the rule for recognizing which character this representation is induced from is: choose $(\mu_x \circ \varphi)(z)$ or $(\mu_{\tau x} \circ \varphi)(z)$ according to which has a negative power of $z$.  Then the representation is induced from that character.  

The power of $|z|$ in $(\mu_x \circ \varphi)(z)$ is $$\sum_{i=1}^{n_1} a_i w_i + \sum_{i=1}^{n_2} \alpha_i c_i + \sum_{i=1}^{n_3} w'_i(b_i+b_i') - 2 \sum_{i=2}^{n_3} b_i' \alpha'_i$$ and the power of $|z|$ in $(\mu_{\tau x} \circ \varphi)(z)$ is $$\sum_{i=1}^{n_1} a_i w_i - \sum_{i=1}^{n_2} \alpha_i c_i + \sum_{i=1}^{n_3} w'_i(b_i+b_i') - 2 \sum_{i=2}^{n_3} b_i \alpha'_i.$$ Note that $$ \sum \alpha_i c_i - 2\sum b_i'\alpha_i' = \left( \sum \alpha_i c_i + \sum \alpha'_i(b_i-b_i') \right)- \sum \alpha'_ib_i - \sum \alpha_i'b_i'$$ and $$ -\sum \alpha_i c_i - 2\sum b_i\alpha_i' = \left(- \sum \alpha_i c_i - \sum \alpha'_i(b_i-b_i') \right)- \sum \alpha'_ib_i - \sum \alpha_i'b_i'.$$ Therefore if the power of $z$ in $(\mu_x \circ \varphi)(z)$ is negative we have that the power of $|z|$ in it is $$\sum_{i=1}^{n_1} a_i w_i  + \sum_{i=1}^{n_3} w'_i(b_i+b_i') + \left| \sum_{i=1}^{n_2} \alpha_i c_i + \sum_{i=1}^{n_3} \alpha'_i(b_i-b_i')\right| - \sum_{i=1}^{n_3} \alpha_i'(b_i+b_i'),$$ and if the power of $z$ in $(\mu_{\tau x} \circ \varphi)(z)$ is negative we get that the power of $|z|$ in it is given by exactly the same formula.  So we find the representation $V_x \circ \varphi$ of $W_{L/F}$ is induced from the character of $W_L$ given by $z^{-a}|z|^b$ where $$a = \left|  \sum_{i=1}^{n_2}\alpha_i c_i + \sum_{i=1}^{n_3} \alpha'_i(b_i - b_i') \right| \text{ and } \sum_{i=1}^{n_1} a_i w_i  + \sum_{i=1}^{n_3} (w'_i-\alpha_i')(b_i+b_i') + \left| \sum_{i=1}^{n_2} \alpha_i c_i + \sum_{i=1}^{n_3} \alpha'_i(b_i-b_i')\right| .$$ 
Then by \cite[\S 3.3.1]{tateNTB} we can conclude that $$L(s,\varphi, V_x) = \Gamma_\C\left( s+ \sum_{i=1}^{n_1} a_i w_i  + \sum_{i=1}^{n_3} (w'_i-\alpha_i')(b_i+b_i') + \left| \sum_{i=1}^{n_2} \alpha_i c_i + \sum_{i=1}^{n_3} \alpha'_i(b_i-b_i')\right|\right).$$
\end{enumerate}

We now collect the above results in a more compact form. 
Given a representation $r$ we determine a $3 \times 3$ block matrix $M=M(r)$ as follows.  Take a decomposition \begin{equation}\label{rdecomp}r \vert_{\LT} = \bigoplus_{i=1}^{m_1} \mu_{x_i} \oplus \bigoplus_{i=1}^{m_2} \left(\mu_{x'_i}\otimes (\text{sign})\right) \oplus \bigoplus_{i=1}^{m_3} V_{x''_i},\end{equation} where each $x_i, x_i' \in X$ are fixed by the action of $G$ and each $x_i'' \in X$ is in a $G$-orbit of cardinality $2$. We may write explicitly \begin{equation}\label{xiexplicit}x_i = (a_{i1},a_{i2},\ldots,a_{in_1},0, \ldots, 0, (b_{i1},b_{i1}), \ldots,(b_{in_3},b_{in_3})),\end{equation} and similarly for $x_i'$. Likewise we may write \begin{equation}\label{xi''explicit} x_i'' = ((a_{i1},a_{i2},\ldots,a_{in_1},c_{i1},\ldots,c_{in_{2}},(b_{i1},b_{i1}'), \ldots,(b_{in_3}, b_{in_3}')), \end{equation} with at least one $c_{ij} \neq 0$ or $b_{ij} \neq b_{ij}'$. Now define the matrix \begin{equation}\label{Mdef}M=M(r)=\left(\begin{array}{ccc} A_1 & 0 & B_1 \\ A_2 & 0 & B_2 \\ A_3 & C & B_3 \end{array}\right)\end{equation} where: $A_1 \in M_{m_1 \times n_1}(\Z)$ is given by $A_1 = ((a_{ij}))$ with $a_{ij}$ as in \eqref{xiexplicit}, $B_1\in M_{m_1 \times n_3}(\Z)$ is given by $B_1 = ((b_{ij}))$ where $b_{ij}$ is also as in \eqref{xiexplicit}. Next, $A_2$ and $B_2$ are defined similarly to $A_1$ and $B_1$ but using the coordinates for $x'_i$ instead of those of $x_i$ as in \eqref{xiexplicit}. Finally, $A_3 \in M_{m_3 \times n_1}(\Z)$ is given by $(a_{ij})$ where $a_{ij}$ are taken from \eqref{xi''explicit}, $C \in M_{m_3 \times n_2}(\Z)$ is given by $C= (c_{ij})$ where $c_{ij}$ are taken from \eqref{xi''explicit}, and $B_3\in M_{m_3 \times n_3}(\Z \times \Z)$ is given by $B_3 = ((b_{ij},b_{ij}'))$, where $(b_{ij},b_{ij}')$ is also taken from \eqref{xi''explicit}. 

We can write the elements $\varphi$ as length $n_1+n_2+n_3$ block column vectors, i.e.~as $$\left( \begin{array}{c} (w,\epsilon) \\ \alpha \\ (w',\alpha') \end{array} \right) \in  (\C^{n_1} \times (\Z/2)^{n_1}) \times \Z^{n_2} \times (  \C^{n_3} \times \Z^{n_3}).$$  We define block-matrix multiplication as follows. Matrices of the form $A$ multiplied on an element $(w,\epsilon) \in (i\R)^{n_1} \times (\Z/2)^{n_1} $ are defined to be $$A_1(w,\epsilon) = \frac{1}{2}A_1w + (A_1\epsilon \mods 2),$$ $$A_2(w,\epsilon) = \frac{1}{2}A_2w + (A_2\epsilon  \mods 2),$$  and $$A_3(w,\epsilon) = A_3w,$$ where the products on the right hand sides are the usual matrix multiplication. 
A matrix of the form $C$ multiplied on an element $(\alpha ) \in \Z^{n_2}$ is defined to by the standard matrix multiplication $C\alpha$.  We define also for elements $(w',\alpha') \in  \C^{n_3} \times \Z^{n_3}$ the multiplication $B_1(w', \alpha') = B_1(w'- \alpha')$ where on the right hand side we have usual matrix multiplication.  We also define the multiplication of $B_2$ in exactly the same way.  So, in summary, \begin{equation}\label{Mvarphi1}(A_1 | 0 | B_1) \left( \begin{array}{c} (w,\epsilon) \\ \alpha \\ (w', \alpha') \end{array} \right)=\frac{1}{2} A_1w +B_1(w'- \alpha') + (A_1\epsilon +B_1 \alpha' \mods 2) \in (i \R \times \{0,1\})^{m_1} \end{equation} and \begin{equation}\label{Mvarphi2}(A_2 | 0 | B_2) \left( \begin{array}{c} (w,\epsilon) \\ \alpha \\ (w', \alpha') \end{array} \right)=\frac{1}{2} A_2w +B_2(w'- \alpha') + (A_2\epsilon +B_2 \alpha' \mods 2)\in (i \R \times \{0,1\})^{m_2} .\end{equation}

Finally, we define the multiplication of $(A_3 | C | B_3)$ on $((w,\epsilon),\alpha,(w',\alpha'))$ as follows.  Let $B_3^+\in M_{m_3 \times n_3}(\Z) $ be the matrix $((b_{ij}+b'_{ij}))$ formed from the entries of $B_3$ and $B_3^-\in M_{m_3 \times n_3}(\Z)$ the matrix with entries $((b_{ij}-b_{ij}'))$.  Then we define \begin{equation}\label{Mvarphi3}(A_3 | C | B_3)\left( \begin{array}{c} (w,\epsilon) \\ \alpha \\ (w', \alpha') \end{array} \right)= A_3w + B_3^+(w'-\alpha') +  C\alpha + B_3^-(\alpha')\in (i\R \times \Z)^{m_3}.\end{equation}  
The above multiplication rules for $M$ define a continuous group homomorphism $$M: \Tt^\wedge \to (i \R \times \{0,1\})^{m_1+m_2} \times (i\R \times \Z)^{m_3}$$
 \begin{equation}\label{Mvarphi} \left( \begin{array}{c} (w,\epsilon) \\ \alpha \\ (w', \alpha') \end{array} \right)\mapsto M\varphi =\left(\begin{array}{ccc} A_1 & 0 & B_1 \\ A_2 & 0 & B_2 \\ A_3 & C & B_3 \end{array}\right) \left( \begin{array}{c} (w,\epsilon) \\ \alpha \\ (w', \alpha') \end{array} \right).\end{equation}  

Let $|\cdot|_{\mods 2}: i\R \times \Z \to \C$ be defined by $$|(it, n)|_{\mods 2} = \begin{cases} it & \text{ if } n \text{ is even, } \\ it +1 & \text{ if } n \text{ is odd}.\end{cases}$$ Let $|\cdot|_{\rm re} : i\R \times \Z \to \C$ be defined by $ |(it , n) |_{\rm re} = it +|n|$. If $M=M(r)$ is as above, and $\varphi$ is a Langlands parameter in the explicit form \eqref{LPparameterization}, let $(M\varphi)_i$ be the $i$th entry of $M\varphi$ as in \eqref{Mvarphi}.  
\begin{proposition}Let $F$ be a real archimedean local field. 
With notation and definitions as above, the archimedean local Langlands $L$-function $L(s, \chi,r)$ associated to a representation $r$ and a unitary character $\chi \in \Hom(T(F) , S^1)$ is given in explicit terms by  
\begin{multline}\label{Larchexplicit} L(s,r \circ \varphi) = \prod_{i=1}^{m_1}\Gamma_\R(s+|(M\varphi)_i|_{\mods 2} ) \prod_{i=m_1+1}^{m_1+m_2} \Gamma_\R(s+|(M\varphi)_i +(0,1)|_{\mods 2})\\ \times \prod_{i=m_1+m_2+1}^{m_1+m_2+m_3} \Gamma_\C(s+|(M\varphi)_i|_{{\rm re}}),\end{multline} where $\varphi$ is the Langlands parameter corresponding to $\chi$ by the local Langlands correspondence for tori \eqref{LLCforTori}.\end{proposition}

Let $x \in T(F)$, which according to our chosen isomorphism \eqref{TRisom} we can express as $$x \mapsto (\ldots,x_j,\ldots,x_j',\ldots,x_j'', \ldots ) \in (\R^\times)^{n_1} \times (S^1)^{n_2} \times (\C^\times)^{n_3},$$ where $x_j \in \R^\times$ for $j=1,\ldots, n_1$, $x_j' \in S^1$ for $j=1,\ldots, n_2$, and $x_j'' \in \C^\times$ for $j=1, \ldots,  n_3$.  Then if $\chi \in T(F)^\wedge$ corresponds to the Langlands parameter $\varphi$ with parametrization \eqref{LPparameterization} across the Langlands correspondence \eqref{langlandsunitary}, we have that $\chi$ is given explicitly by \begin{multline}\label{chiarchexplicit} \chi(x) = (\sgn x_1)^{\epsilon_1} \cdots (\sgn x_{n_1})^{\epsilon_{n_1}} |x_1|^{w_1}\cdots |x_{n_1}|^{w_{n_1}} {x'_{1}}^{\alpha_{1}} \cdots {x'_{n_2}}^{\alpha_{n_2}} \\ \times  \left(\frac{|x''_{1}|}{x''_{1}}\right)^{\alpha_1'} |x_1''|^{w_1'-\alpha_1'}\cdots \left(\frac{|x''_{n_3}|}{x''_{n_3}}\right)^{\alpha_{n_3}'} |x''_{n_3}|^{w_{n_3}'-\alpha_{n_3}'}.\end{multline}

We briefly discuss the case that $F$ is a complex archimedean local field. We have $W_F = F^\times$ and $W_{F/F} = W_F^{\text{ab}} = W_F$. Let us choose isomorphisms $F \simeq \C$ and $T(F) \simeq (\C^\times)^n$. The Langlands parameters $\varphi : W_{F} \to \widehat{T}_u$ are given explicitly by
\begin{equation}\label{LPparameterization_complex}
\varphi(z) = \left( \left( \frac{|z|}{z}\right)^{\alpha_1}|z|^{w_1-\alpha_1}, \ldots, \left( \frac{|z|}{z}\right)^{\alpha_n}|z|^{w_n-\alpha_n} \right)
\end{equation}
for some $w_j \in \C$ and $\alpha_j \in \Z$ with $w_j-\alpha_j \in i \R$. We write $\Tt^\wedge = (i \R \times \Z)^n$, so that Langlands parameters are parametrized by $(w,\alpha) \in \Tt^\wedge$. The irreducible algebraic representations are merely the characters of $\widehat{T}$, i.e.\ $X^*(\widehat{T}) \simeq X=X_*(\G_m)^n= \Z^n$. If $r \vert_{\widehat{T}}$ decomposes as
\begin{equation}\label{rvertThat_complex}
r \vert_{\widehat{T}} = \bigoplus_{i=1}^m \mu_{x_i},
\end{equation}
for $x_i = (b_{i1}, \ldots, b_{in})\in X$, then define the $m \times n$ matrix $M = (b_{ij})$.  Then $M$ is a continuous group homomorphism $M: \Tt^\wedge \to (i \R \times \Z)^m$ given by 
\begin{equation}\label{Mvarphi_complex}
(w,\alpha) \mapsto M\varphi = M(w,\alpha) = \left(\sum_{j=1}^n b_{ij}(w_j-\alpha_j) + \sum_{j=1}^n b_{ij} \alpha_j\right)_{i=1,\ldots m}.
\end{equation}
\begin{proposition}Let $F$ be a complex archimedean local field. 
With notation and definitions as above, the archimedean local Langlands $L$-function $L(s, \chi,r)$ associated to a representation $r$ and a unitary character $\chi \in \Hom(T(F) , S^1)$ is given in explicit terms by  
\begin{equation}\label{Larchexplicit_complex} L(s,r \circ \varphi) = \prod_{i=1}^{m} \Gamma_\C(s+|(M\varphi)_i|_{{\rm re}}),\end{equation} where $\varphi$ is the Langlands parameter corresponding to $\chi$ by the local Langlands correspondence for tori \eqref{LLCforTori}.\end{proposition}

Let $x \in T(F)$ corresponding to $(x_1, \ldots, x_n)\in (\C^\times)^n$ by the chosen isomorphism $T(F) \simeq (\C^\times)^n$. If $\chi \in T(F)^\wedge$ corresponds to a Langlands parameter $\varphi$ with parameter $(w,\alpha)$ across the Langlands correspondence \eqref{langlandsunitary}, then $\chi$ is given by 
\begin{equation}\label{chiarchexplicit_complex}\chi(x) = \varphi(x),\end{equation} where the latter is expressible in explicit terms by \eqref{LPparameterization_complex}.

\subsection{Local conductor zeta function, archimedean case}\label{archimedean}
In this section $F$ denotes an archimedean local field and $T$ and $F$-torus. Choose a Haar measure $\nu$ on the Pontryagin dual $T(F)^\wedge$. Let
\begin{equation}\label{AFdef}
A_F(s,x) = \int_{T(F)^\wedge} \frac{\chi(x)}{\fc(\chi,r)^s}\,d\nu(\chi).
\end{equation}

 We assume that $F\simeq \R$ until further notice and briefly discuss the simpler case that $F \simeq \C$ at the end of the section.
Recall \eqref{TRisom} that we have chosen an isomorphism 
\begin{equation}\label{TRisom2} T(k_v) \simeq \Tt = (\R^\times)^{n_1} \times (S^1)^{n_2} \times (\C^\times)^{n_3}\end{equation}
$$ x \mapsto (\ldots, x_j,\ldots, x_j', \ldots, x_j'',\ldots),$$
with $x_j \in \R^{\times}, x_j' \in S^1$, and $x_j'' \in \C^\times$.
\begin{theorem}\label{archthm}Suppose that $r \vert_{\widehat{T}}$ is faithful. \begin{enumerate}\item\label{archthm1} The integral $A_{F}(s,x)$ converges absolutely and uniformly on compacta in the domain $$\real(s) >\sigma_0=\max\{\frac{\dim D(S)}{|S|} : \dim D(S) \geq 1\}.$$ 
\item\label{archthm2} For $s$ in the above region of absolute convergence, we have $$A_{F}(s,x) \ll_{T,r,s} \prod_{1 \leq j \leq n_1} \frac{1}{1+|\log |x_j|| }\prod_{1 \leq j \leq n_3} \frac{1}{1+2|\log |x_j''|| },$$ with at most polynomial growth in $s$ in vertical strips.  
\item\label{archthm3} For any real $ \sigma_0 <\sigma \leq 2$, we have that $A_{F}(\sigma, x) $ is non-negative real.\end{enumerate}
\end{theorem}
From part \eqref{archthm1} of Theorem \ref{archthm} we extract the following corollary.
\begin{corollary}
Let $F$ be an archimedean local field and $\nu$ a Haar measure on the Pontryagin dual $T(F)^\wedge$ of the archimedean torus $T(F)$. Suppose that $r \vert_{\widehat{T}}$ is faithful. For any $\eps>0$ we have
\begin{equation}
\nu(\{ \chi \in T(F)^\wedge : \fc(\chi,r) \leq X \}) \ll_{T, r, \nu, \eps} X^{\sigma_0 + \eps}.
\end{equation}
\end{corollary}
The proof of Theorem \ref{archthm} will occupy the remainder of section \ref{archimedean} of this paper. In section \ref{analysis} we reduce assertion \eqref{archthm1} to a problem in combinatorial geometry (see Proposition \ref{polytopeconj}). The main input in the proof of assertion \eqref{archthm1} is a Brascamp-Lieb inequality, see Theorems \ref{BL} and \ref{BCCT10}. In section \ref{matroidbackground}, we give some background information on matroids and polymatroids, and in section \ref{pfofpolytopeconj} we solve the combinatorial geometry problem. Assertion \eqref{archthm2} follows immediately. Finally, in section \ref{positivity} we prove assertion \eqref{archthm3} of Theorem \ref{archthm}.

Recall that in section \ref{archimedeanC} we worked out explicitly the local Langlands correspondence for tori over archimedean local fields. Specifically, in \eqref{LPparameterization} we explicitly parameterized (with respect to choices $K_w \simeq \C$ and \eqref{TRisom2}) equivalence classes of Langlands parameters $\varphi: W_{L/F} \to \LT_u$ by 
\begin{equation}\label{Ttwedgedef2}((w,\epsilon),\alpha, (w',\alpha')) \in (i \R^{n_1} \times (\Z/2\Z)^{n_1}) \times \Z^{n_2} \times ( i \R^{n_3} \times \Z^{n_3})= \Tt^\wedge.\end{equation}
Recall the definition \eqref{Mdef} of the matrix $M=M(r)$, which gives by \eqref{Mvarphi} a map $M:\Tt^\wedge \to (i\R\times \{0,1\})^{m_1+m_2} \times (i \R \times \Z)^{m_3}$. By Definition \ref{archdef}, \eqref{Larchexplicit}, and \eqref{chiarchexplicit} we have for some constant $a$ depending only on the choice of $\nu$ that \begin{equation}\label{avsxGeneral} A_{F}(s,x) = a\sum_{\substack{\alpha \in \Z^{n_2} \\   \alpha' \in \Z^{n_3}}} \sum_{\epsilon \in \{0,1\}^{n_1}} \iint_{\substack{w \in i\R^{n_1} \\  w' \in i\R^{n_3} }} \frac{\chi_{w,\epsilon,\alpha, w',\alpha'} (x)}{\prod_{i=1}^{m_1+m_2} (|(M\varphi)_i|+1)^s \prod_{i=m_1+m_2+1}^{m_1+m_2+m_3} (|(M\varphi)_i|+1)^{2s}}\,dw\,dw',\end{equation}
where $\chi_{w,\epsilon,\alpha, w',\alpha'}$ is the unitary character of $\Tt = (\R^\times)^{n_1} \times (S^1)^{n_2} \times (\C^\times)^{n_3}$ that was given explicitly in terms of $w,\epsilon,\alpha, w',\alpha'$ in \eqref{chiarchexplicit}.

\subsubsection{Convergence}\label{analysis}
We apply the triangle inequality to $A_{F}(s,x)$. 
For $i=1, \ldots, m_1+m_2$, we have $(M\varphi)_i \in i\R \times \{0,1\}$ by inspection of \eqref{Mvarphi1}, \eqref{Mvarphi2}.  For such $i$ we apply the inequality $$ \frac{1}{\sqrt{x^2+1}+1} \leq \frac{1}{|x|+1}.$$  
Then we make the change of variables $w_j \mapsto iw_j$ and $w_j'-\alpha_j' \mapsto i w_j',$ so that $\chi \in T(k_v)^\wedge$ unitary implies that $w_j,w_j' \in \R$. 

We introduce some notation to record the result (see \eqref{avsxbound1}) of the aforementioned manipulations of $A_{F}(s,x)$. Let $M_\text{re}$ denote the $(m_1+m_2+m_3) \times (n_1+n_3)$ matrix $$M_\text{re} = \left( \begin{array}{cc} A_1 & B_1 \\ A_2 & B_2 \\ A_3 & B_3^+ \end{array}\right),$$ were $A_1,A_2,A_3,B_1,B_2,B_3^+$ were defined in section \ref{archimedeanC}. Such a matrix acts on $\overline{w} = (w,w') \in \R^{n_1+n_3}$ by the usual multiplication of matrices. Let also $$M_\text{int} = \left( \begin{array}{cc} C & B_3^-\end{array}\right),$$ where $C$ and $B_3^-$ were also defined in section \ref{archimedeanC}. The integral $m_3 \times (n_2+ n_3)$ matrix $M_{\text{int}}$ acts on $\overline{\alpha} = (\alpha, \alpha') \in \Z^{n_2 + n_3}$ by the usual multiplication of matrices. 

The result of our inequalities and changes of variable is \begin{multline}\label{avsxbound1} A_{F}(s,x) \ll \\ \sum_{\overline{\alpha} \in \Z^{n_2+n_3}}  \iint_{\overline{w} \in \R^{n_1+n_3}}  \prod_{i=1}^{m_1+m_2}\frac{1}{\big(|(M_{\text{re}}\overline{w})_i|+1\big)^{\real(s)}}\prod_{i=m_1+m_2+1}^{m_1+m_2+ m_3} \frac{1}{\big(( (M_\text{re}\overline{w})_i^2+(M_{\text{int}}\overline{\alpha})_i^2)^{1/2}+1 \big)^{2\real(s)}}\,d\overline{w} .\end{multline}

Before proceeding with the estimation of \eqref{avsxbound1}, we first describe a result in combinatorial geometry.  Let $M\in M_{m\times n}(\R)$ be an $m \times n$ matrix with real entries, $m\geq n$. \begin{definition}\label{biasdef} For any $\alpha \geq \beta \geq 1$ we say that $M$ is $(\alpha;\beta)$\emph{-biased} if there exist $\alpha$ rows of $M$ such that any basis of $\R^n$ formed from rows of $M$ contains at least $\beta$ of the distinguished $\alpha$ rows.\end{definition}
For example, note that any full-rank $m \times n$ matrix is $(m;n)$-biased. 

We also have the following minor variation on $(\alpha; \beta)$-bias. 
\begin{definition}\label{biasdef_variation}
Let $M$ be an $m \times n$ matrix with real entries along with a partition of its rows into two subsets $R_1 \sqcup R_2$. For any $\alpha_1, \alpha_2\geq 0$ and $\beta$ satisfying $\alpha_1+\alpha_2\geq \beta \geq 1$, we say that $M$ is $(\alpha_1,\alpha_2;\beta)$\emph{-biased} if there exists $\alpha_1$ rows from $R_1$ and $\alpha_2$ rows from $R_2$ such that any basis of $\R^n$ formed from rows of $M$ contains at least $\beta$ of the distinguished $\alpha_1+\alpha_2$ rows. 
\end{definition}

Recall the convex polytope $H_M$ associated to a matrix $M$ from section \ref{tools}. Write $\| \cdot \|_\infty$ for the $L^\infty$-norm on $\R^m$, i.e.~for $x\in\R^m$ we set $$\| x \|_\infty = \max(|x_1|,\ldots,|x_m|).$$ The norm $\| \cdot \|_\infty$ is a convex and piecewise-linear function on $\R^m$. Let \begin{equation}\label{B0def}B_\infty= B_\infty(M) = \inf\{ \| x\|_\infty: x \in H_M\}.\end{equation} \begin{proposition}\label{polytopeconj}
Let $M$ be any full-rank $m\times n$ matrix with real entries.  We have that $$B_\infty= \max\{\frac{\beta}{\alpha}: M \text{ is } (\alpha;\beta)\text{-biased}\}.$$ \end{proposition}
 The proof of Proposition \ref{polytopeconj} will be given in section \ref{pfofpolytopeconj}. 
 
 We also need a version of Proposition \ref{polytopeconj} for $(\alpha_1,\alpha_2;\beta)$-bias, and now spell this out in detail.
 For $x \in \R^m$ and a partition of the coordinates of $\R^m$ into two subsets $R_1 \sqcup R_2$, let 
 $$\| x \|_{\infty,1/2} = \max(\{|x_i|: i \in R_1\}\cup \{\frac{1}{2}|x_j|: j \in R_2\}).$$
 The function $\| \cdot \|_{\infty,1/2}$ is a convex and piecewise-linear function on $\R^m$. Let \begin{equation}\label{B0halfdef}B_{\infty,1/2}= B_{\infty,1/2}(M) = \inf\{ \| x\|_{\infty,1/2}: x \in H_M\}.\end{equation} 
 \begin{proposition}\label{polytopeconj_variation}
Let $M$ be any full-rank $m\times n$ matrix with real entries equipped with a partition of its rows into two subsets $R_1 \sqcup R_2$.  We have that $$B_{\infty,1/2}= \max\{\frac{\beta}{\alpha_1+2\alpha_2}: M \text{ is } (\alpha_1,\alpha_2;\beta)\text{-biased}\}.$$ \end{proposition}

 We now give the proof of assertion \eqref{archthm1} of Theorem \ref{archthm}, assuming Propositions \ref{polytopeconj} and \ref{polytopeconj_variation}.
 We begin to estimate $A_{F}(s,x)$ by applying the Brascamp-Lieb inequality (see section \ref{tools}) to the integral over $\overline{w}$ in \eqref{avsxbound1}.  
 Let $m=m_1+m_2+m_3$, $n=n_1+n_3$ and $M=M_{\text{re}}$. Partition the rows of $M$ into the first $m_1+m_2$ and the last $m_3$ rows, as in Definition \ref{biasdef_variation} and Proposition \ref{polytopeconj_variation}. 
The polytope $H_{M_\text{re}}$ is compact and non-empty, so the infimum in \eqref{B0halfdef} is attained, say by $\overline{B} = (B_1, \ldots, B_{m_1+m_2+m_3}) \in H_{M_\text{re}}$.    
 
We apply the Brascamp-Lieb Inequality Theorem \ref{BL} to the interior integral of \eqref{avsxbound1} with $a_i$ equal to the rows of $M = M_{\text{re}}$, $\overline{p} = \overline{B}$, $$f_1(x) = \cdots = f_{m_1+m_2}(x) = \frac{1}{(|x|+1)^{\real(s)}},$$ and $$f_i(x) = \frac{1}{\big((x^2+ (M_\text{int}\overline{\alpha})_i^2)^{1/2}+1\big)^{2\real(s)}}, \text{ for } i=m_1+m_2+1,\ldots, m_1+m_2+m_3.$$ With these choices, Theorem \ref{BL} shows that 
\begin{multline}\label{avsxbound3} A_{F}(s,x) \ll_{r,T} \\ \prod_{i=1}^{m_1+m_2} \left\| (|x|+1)^{-\real(s)}\right\|_{L^{B_i^{-1}}(\R)} \sum_{\overline{\alpha} \in \Z^{n_2+n_3}} \prod_{i=m_1+m_2+1}^{m_1+m_2+m_3} \left\|\big((x^2+ (M_\text{int}\overline{\alpha})_i^2)^{1/2}+1\big)^{-2\real(s)}\right\|_{L^{B_i^{-1}}(\R)} .\end{multline} 

Next we use the discrete Brascamp-Lieb inequality Theorem \ref{BCCT10} to bound the sum over $\overline{\alpha}$ in \eqref{avsxbound3}.  Since $M_{\text{int}}$ is full-rank, the infimum in \eqref{B0def} is attained, say by $\overline{B}' = (B_{m_1+m_2+1}', \ldots, B_{m_1+m_2+m_3}') \in H_{M_\text{int}}$. 
Let $a_i$, $i=1,\ldots, m_3$ denote the rows of $M_{\text{int}}$. We have that $x\in H_{M_\text{int}}$ if and only if \begin{equation}\label{CLLcondition3} \sum_{i=1}^{m_3} x_i = n_2+n_3\end{equation} and \begin{equation}\label{CLLcondition4}\sum_{i \in S} x_i \leq \rank(\mathrm{span}_\Z(\{a_i: i \in S\}))\end{equation} for every subset $S \subseteq \{1,\ldots, m_3\}$, by tensoring with $\R$. Let $\varphi_i:\Z^{n_2+n_3}\to \Z$ be given by $x \mapsto \langle a_i,x\rangle$. The discussion on \cite[p.~649]{BCCTintegral} shows that \eqref{CLLcondition3} and \eqref{CLLcondition4} imply that \begin{equation} \rank(H ) \leq \sum_{i=1}^{m_3} x_i \rank(\varphi_i(H)) \quad \text{for every subgroup } H \text{ of } \Z^{n_2+n_3}. \end{equation} 

We apply Theorem \ref{BCCT10} with $G=\Z^{n_2+n_3}$, $G_i = \Z$, $\varphi_i$ as above, $p_i = {B_i'}^{-1}$, and $$f_{i}(x) = \left\| \frac{1}{(\sqrt{x^2+\alpha^2}+1)^{2\real(s)}}\right\|_{L^{B_{m_1+m_2+i}^{-1}}(\R)} \text{ for } i=1, \ldots, m_3,$$ to obtain from \eqref{avsxbound3} that 
\begin{equation}\label{avsxbound4} A_{F}(s,x) \ll_{r,T} \prod_{i=1}^{m_1+m_2} \left\| (|x|+1)^{-\real(s)}\right\|_{L^{B_i^{-1}}(\R)}  \prod_{i=m_1+m_2+1}^{m_1+m_2+m_3}\left\| \left\|(\sqrt{x^2+ \alpha^2}+1)^{-2\real(s)}\right\|_{L^{B_i^{-1}}(\R)}\right\|_{\ell^{{B_i'}^{-1}}(\Z)} .\end{equation} 
The right hand side converges as soon as \begin{multline}\label{domainofac1}\real(s)> \max( \max_{i=1,\ldots, m_1+m_2}( B_i) , \max_{i = m_1+m_2+1, \ldots, m_1+m_2+m_3} \frac{1}{2} \max(B_i,B_i')) \\
= \max(B_{\infty,1/2}(M_{\text{re}}), \frac{1}{2} B_\infty(M_{\text{int}})).
\end{multline}

Since $r\vert_{\widehat{T}}$ is faithful, the matrices $M_\text{re}$ and $M_{\text{int}}$ are full-rank and so by Propositions \ref{polytopeconj} and \ref{polytopeconj_variation}, we have that $A_{F}(s,x)$ converges absolutely when 
\begin{equation}\label{domainofac3}\real(s) > \max( \max\{\frac{\beta}{\alpha_1+2\alpha_2}: M_\text{re} \text{ is } (\alpha_1,\alpha_2; \beta)\text{-biased}\}, \frac{1}{2} \max\{\frac{\beta}{\alpha}: M_\text{int} \text{ is } (\alpha;\beta)\text{-biased}\}).\end{equation} 
Let $$M' = \left(\begin{array}{cccc} A_1 & 0 & B_1 & B_1 \\
A_2 & 0 & B_2 & B_2 \\
A_3 & C & B_3^+ + B_3^- &  B_3^+ - B_3^- \\
A_3 & -C & B_3^+ - B_3^- &  B_3^+ + B_3^- \end{array}\right).$$ 
We claim that \begin{equation}\label{domainofac4}\max( \max\{\frac{\beta}{\alpha_1+2\alpha_2}: M_\text{re} \text{ is } (\alpha, \beta)\text{-biased}\}, \frac{1}{2} \max\{\frac{\beta}{\alpha}: M_\text{int} \text{ is } (\alpha,\beta)\text{-biased}\}) \leq B_\infty(M').\end{equation} Indeed, suppose the first maximum on the left hand side is larger.  Then there are $\alpha = \alpha_1+\alpha_2$ distinguished rows of $M_\text{re}$, $\alpha_1$ of which are among the first $m_1+m_2$ rows, and $\alpha_2$ are among the last $m_3$ rows. Choose the corresponding $\alpha_1$ rows of $M'$ among the first $m_1+m_2$ rows, and the corresponding $2\alpha_2$ rows, i.e. $\alpha_2$ pairs of rows of $M'$ from among the last $2m_3$ rows. This set of $\alpha_1+2\alpha_2$ rows of $M'$ shows that $M'$ is $(\alpha_1+2\alpha_2,\beta)$-biased. So by another application of Proposition \ref{polytopeconj} $$ \max\{\frac{\beta}{\alpha_1+2\alpha_2}: M_\text{re} \text{ is } (\alpha, \beta)\text{-biased}\} \leq B_\infty(M').$$ 
Similarly, suppose the second maximum on the left hand side of \eqref{domainofac4} is larger. Then there are $\alpha$ distinguished rows of $M_\text{int}$, and we choose the corresponding $2\alpha$ rows, i.e. $\alpha$ pairs of rows from among the last $2m_3$ rows of $M'$. This distinguished set of $2 \alpha$ rows of $M'$ shows that $M'$ is $(2\alpha,\beta)$-biased. So $$ \frac{1}{2} \max\{\frac{\beta}{\alpha}: M_\text{int} \text{ is } (\alpha,\beta)\text{-biased}\} \leq B_\infty(M'),$$ which finishes the proof of \eqref{domainofac4}. 

Finally, there is a bijection between the rows of $M'$ and the co-weights of $r$ via the isomorphism \eqref{X*Tisom}. Under this isomorphism, sets of rows of $M'$ correspond bijectively to subsets $S \subseteq M$ of co-weights of $r$, and $|S|=\alpha$ and $\dim D(S)=\beta$. This concludes the proof of assertion \eqref{archthm1} of Theorem \ref{archthm}. 

Assertion \eqref{archthm2} of Theorem \ref{archthm} follows immediately. Indeed, returning to equation \eqref{avsxGeneral}, we integrate by parts once in each variable $w_i, w_i'$, and apply part \eqref{archthm1} of the theorem. We artificially insert the factor of $2$ on the complex places as $| \cdot |^2$ is the natural absolute value on them from the point of view of algebraic number theory. This factor of $2$ serves to make the computations for final estimate in section \ref{conclusion} more elegant.

Before moving on to the purely matroid-theoretic sections \ref{matroidbackground} and \ref{pfofpolytopeconj}, we remark that Theorem \ref{archthm}\eqref{archthm1}\eqref{archthm2} also holds in the case that $k_v \simeq \C$ is a complex place upon taking $n_1=n_2=0$ and $m_1=m_2=0$, and $n=n_3 = \dim T$ and $m=m_3=\dim r$. Indeed, by \eqref{Larchexplicit_complex} and \eqref{chiarchexplicit_complex} in place of \eqref{Larchexplicit} and \eqref{chiarchexplicit}, the generating series $A_{F}(s,x)$ is equal to the expression in \eqref{avsxGeneral} with $n_1=n_2=0$ and $m_1=m_2=0$, the matrix $M$ as defined between \eqref{rvertThat_complex} and \eqref{Mvarphi_complex}.

Continuing as above, the integral that we need to bound is given by \eqref{avsxbound1} with $M_{\rm re}= M_{\rm int}=M$, $n_1=n_2=0$, $m_1=m_2=0$.  
Following the same steps as above, we have absolute convergence when $s$ satisfies \eqref{domainofac1} with $m_1=m_2=0$, 
which simply equals $\frac{1}{2}B_\infty(M)$. Since $r\vert_{\widehat{T}}$ is faithful, the matrix $M$ is full-rank, so Proposition \ref{polytopeconj} implies that $A_{F}(s,x)$ converges absolutely whenever $$\real(s) >  \max\{\frac{\beta}{\alpha}: M \text{ is } (\alpha;\beta)\text{-biased}\}).$$ Since there is a bijection between the rows of $M$ and the co-weights of $r$ via the isomorphism $X^*(\widehat{T})\simeq X_*(\G_m)^n = \Z^n$ (see the penultimate paragraph of section \ref{archimedeanC}), we conclude the first assertion of Theorem \ref{archthm}, as before. Note that we did not need to use the ``variations'' in Definition \ref{biasdef_variation} or Proposition \ref{polytopeconj_variation}, nor the paragraph containing the definition of $M'$ in the case that $k_v\simeq \C$.

\subsubsection{Background on matroids and polymatroids}\label{matroidbackground}
The key observation in the proof of Proposition \ref{polytopeconj} is that the definition of $(\alpha,\beta)$-bias makes sense more generally for \emph{matroids}, and $H_M$ is exactly the matroid base polytope associated to the matroid $M$. To this end, we next recall some background on matroids and polymatroids. The following exposition was communicated to the author by R.~Zenklusen. 

\begin{definition}[Matroid]\label{def:matroid}
A \emph{matroid} is a pair $(N,\cI)$ where $N$ is a finite set and $\cI \subset 2^N$ is a family of ``independent'' subsets of $N$ satisfying the following axioms. 
\begin{enumerate}
\item\label{matax1} $\cI \neq \varnothing$
\item\label{matax2} If $I \in \cI$ and $J \subset I$ then $J \in \cI$.
\item\label{matax3} If $I,J \in \cI$ and $|J| > |I|$ then there exists $e \in J \smallsetminus I $ such that $I \cup \{e\} \in \cI$. 
\end{enumerate}
\end{definition}
\begin{example}[Linear Matroid]
If $N$ is a finite set of vectors spanning a vector space, and $\cI$ is the set of linearly independent subsets of $N$, then $(N,\cI)$ is a matroid. One calls such a matroid a \emph{linear matroid}.
\end{example}

If $(N,\cI)$ is a matroid, then the set of \emph{bases} $\cB \subset \cI$ is the set of maximal subsets of $\cI$, ordered by inclusion. If $(N,\cI)$ is a linear matroid, then $\cB$ consists of subsets of vectors which form a basis. 

\begin{definition}\label{def:rank}
The \emph{rank function} of a matroid is the function $r:2^N \to \Z_{\geq 0}$ given by $$ r(S) = \max\{ |I| : I \in \cI,\, I \subset S\}.$$\end{definition} If $(N,\cI)$ is a linear matroid, then $r(S)$ is the dimension of the space spanned by the vectors in $S$.  By the definitions of $\cB$ and $r$ we have $$\cB = \{I \in \cI : r(I) = r(N)\}.$$
\begin{lemma} The rank function of a matroid $(N,\cI)$ satisfies the following properties.
\begin{itemize}
\item $r: 2^N \to \Z$
\item $r$ is submodular: $$r(A) + r(B) \geq r(A \cup B) + r(A \cap B)$$
\item $r$ is monotone: $ r(A) \geq r(B)$ for all $B \subseteq A \subseteq N$
\item $r$ is non-negative: $r(A)\geq 0$ for all $A \subseteq N$
\item $r$ satisfies $r(A \cup \{e\}) \leq r(A) +1$ for all $A \subseteq N$ and $e \in N$.
\end{itemize}
If $r:2^N \to \Z$ is any function enjoying these five properties, then there exists a unique $\cI \subset 2^N$ such that $(N,\cI)$ is a matroid whose rank function is $r$.
\end{lemma} \begin{proof} See \cite[\S39.7]{Schrijver}.  \end{proof}

Let $(N,\cI)$ be a matroid and let $\mathbf{1}_I \in \{0,1\}^N$ be the indicator function of $I$.  If $S$ is a finite set of points in $\R^N$, then we write $\conv(S)$ for the convex hull formed from those points. 
\begin{definition}\label{def:polytopes} The \emph{matroid polytope} of $(N,\cI)$ is $$P_\cI = \conv( \{\mathbf{1}_I : I \in \cI\})\subset \R^N$$ and the \emph{matroid base polytope} is $$P_\cB = \conv( \{\mathbf{1}_B : B\in \cB\})\subset \R^N.$$\end{definition}

We can also express the matroid polytope and matroid base polytope in terms of the rank function as follows.
  Let $x \in \R_{\geq 0}^N$, $e \in N$ and $x_e$ be the $e$-th component of $x$.  For a subset $S \subseteq N$ we set $x(S) = \sum_{e \in S} x_e$.  In terms of $x(S)$, we have (see \cite[Cor.~40.2b]{Schrijver}) $$P_\cI = \{ x \in \R_{\geq 0}^N : x(S) \leq r(S)  \text{ for all } S \subseteq N\}.$$  Then $P_\cB$ is one face of the matroid polytope $P_\cI$ given by a supporting hyperplane, i.e. we have (see \cite[Cor.~40.2d]{Schrijver})
\begin{equation}\label{basispolytope}P_\cB = P_\cI \cap \{x \in \R^N : x(N) = r(N)\}.\end{equation}   

\begin{theorem}[Matroid Intersection]
Let $(N,\cI_1)$ and $(N,\cI_2)$ be two matroids on the same ground set.  Then we have $$\max \{ |I| : I \in \cI_1 \cap \cI_2 \} = \min_{A\subseteq N } \{ r_1(A) + r_2(N \smallsetminus A)\}.$$\end{theorem}  \begin{proof} See \cite[Thm.~41.1]{Schrijver}. \end{proof} 
In fact, we shall need the following polyhedral generalization of matroids.
\begin{definition}[Polymatroid]A \emph{polymatroid} on $N$ is a polytope $$P_f= \{x \in \R^N_{\geq 0} : x(S) \leq f(S)\, \text{ for all } S \subseteq N\}$$ where $f:2^N \to \R_{\geq 0}$ is a submodular and monotone function.\end{definition}  
\begin{theorem}[Polymatroid intersection]\label{polyinter}
Let $f_1,f_2: 2^N \to \R_{\geq 0}$ be two submodular and monotone functions.  We have $$\sup\{ x(N) : x \in P_{f_1} \cap P_{f_2}\} = \min_{A \subseteq N} \{f_1(A) + f_2(N \smallsetminus A)\}.$$
\end{theorem}
\begin{proof}  See \cite[Cor.~46.1b]{Schrijver}.  \end{proof}

The following definition generalizes Definition \ref{biasdef} from linear matroids to matroids. \begin{definition}[$(\alpha;\beta)$-bias for matroids] We say a matroid $(N,\cI)$ is $(\alpha;\beta)$-biased if there exists $S\subseteq N$ with $|S|=\alpha$ such that $$|B \cap S | \geq \beta$$ for all bases $B \subseteq N$.\end{definition}   

Similarly, if $N$ is equipped with a partition $N=R_1\sqcup R_2$, then we say that a matroid $(N,\cI)$ is $(\alpha_1,\alpha_2;\beta)$-biased if there exists $S\subseteq N$ with $|S\cap R_1|=\alpha_1$, $|S \cap R_2| = \alpha_2$ and such that $|B \cap S | \geq \beta$ for all bases $B \subseteq N$.

\begin{lemma} A subset $S \subseteq N$ satisfies $r(N) - r(N \smallsetminus S) \geq \beta$ if and only if for any basis $B \subseteq N$ we have $|B \cap S|\geq \beta$.\end{lemma}
\begin{proof} 
``Only if'': Let $B \subseteq N$ be any basis. By the sub-modularity of the rank function we have $$r(B \cap S) + r(N \smallsetminus S ) \geq r((B \cap S) \cup (N\smallsetminus S)).$$ Since $B \subseteq (B \cap S) \cup (N\smallsetminus S)$, we have $$r((B \cap S) \cup (N\smallsetminus S))= r(N).$$ However, $B \cap S$ is independent, so we have $$|B\cap S| = r(B \cap S) \geq r(N)-r(N \smallsetminus S) \geq \beta.$$

``If'': Suppose that $B\in \cB$ is such that $|B \cap S|$ is minimal as we range over all bases. Equivalently, $B$ is such that $|B \smallsetminus S|$ is maximal. We claim that $B \smallsetminus S$ is maximal by inclusion among independent sets which are disjoint from $S$. From this claim it follows by definition of the rank function that $$|B \smallsetminus S| = r(N\smallsetminus S),$$ and so $r(N)-r(N \smallsetminus S) = |B \cap S|\geq \beta$. 

If the claim were false, then there would exist $e \not \in S \cup B$ such that \begin{equation*}\label{matroidadde}(B \smallsetminus S ) \cup \{e\}= (B \cup \{e\}) \smallsetminus (B \cap S)\end{equation*} is an independent set, by matroid axiom \eqref{matax3}.  Since the set $(B \smallsetminus S ) \cup \{e\}$ is independent and $r(B \cup \{e\})=r(N)$, we can complete $(B \smallsetminus S ) \cup \{e\}$ to a basis $\widetilde{B} \subseteq B \cup \{e\}$. But then we have $$(B \cup \{e\}) \smallsetminus (B \cap S) = \widetilde{B} \smallsetminus S,$$ from which it follows that $$ |\widetilde{B} \smallsetminus S| = |(B \cup \{e\}) \smallsetminus (B \cap S)| = |B \smallsetminus S| + 1.$$ This contradicts the minimality of $|B \cap S|$ among all bases $B \in \cB$. Therefore the claim is true. 
\end{proof}
\begin{corollary}\label{abbiasdef2} A matroid $(N,\cI)$ is $(\alpha;\beta)$-biased if and only if there exists $S\subseteq N$ with $|S|=\alpha$ and $r(N)-r(N \smallsetminus S)\geq \beta$. A matroid $(N, \cI)$ with a partition $N = R_1 \sqcup R_2$ is $(\alpha_1,\alpha_2;\beta)$-biased if and only if there exists $S\subseteq N$ with $|S\cap R_1|=\alpha_1$, $|S\cap R_2|=\alpha_2$ and $r(N)-r(N \smallsetminus S)\geq \beta$. \end{corollary}

\subsubsection{Proof of Propositions \ref{polytopeconj} and \ref{polytopeconj_variation}}\label{pfofpolytopeconj}
We begin with the proof of Proposition \ref{polytopeconj}. 
Considering the level sets of the $L^\infty$ norm, the $B_\infty$ defined in \eqref{B0def} becomes $$B_\infty=\inf\{\lambda \geq 0: P_\cB \cap[0,\lambda]^N \neq \varnothing\}.$$ (Aside: compare this and \eqref{basispolytope} to the discussion of the Manin conjecture in section \ref{sec:manin}.) From the description \eqref{basispolytope} of the matroid basis polytope in terms of the rank function, we have $$B_\infty=\inf\{ \lambda\geq 0: \sup\{ x(N) : x \in P_\cI \cap [0,\lambda]^N\} =r(N)\}.$$  Now we re-interpret $[0,\lambda]^N$ as the polymatroid defined by the function $$f:2^N \to \R_{\geq 0}$$ $$ f(S) = \lambda |S|.$$ Then, by the polymatroid intersection theorem (Theorem \ref{polyinter}), we have $$\sup\{ x(N) : x \in P_\cI \cap [0,\lambda]^N\} = \min_{A \subseteq N} \{ r(A) + f(N \smallsetminus A)\} = \min_{A \subseteq N} \{ r(A) + \lambda| N \smallsetminus A|\}.$$  Since this last min is always $\leq r(N)$, to characterize $B_\infty$ it suffices to find the smallest $\lambda\geq 0$ such that for all $A\subseteq N$\begin{equation}\label{penultimate}r(A) + \lambda |N \smallsetminus A|\geq r(N),\end{equation}  that is to say
$$B_\infty = \inf \{\lambda \geq 0 : r(A) + \lambda |N \smallsetminus A|\geq r(N)\text{ for all } A\subseteq N\}.$$ It changes nothing to swap $A$ with $N \smallsetminus A$, so $$ B_\infty = \inf \{\lambda \geq 0 : r(N\smallsetminus A) + \lambda |A|\geq r(N)\text{ for all } A\subseteq N\}.$$ If $A=\varnothing$ then the inequality is satisfied for all $\lambda$, so suppose not.  We then have by Corollary \ref{abbiasdef2} $$ B_\infty =   \max_{A \subseteq N} \{\frac{r(N)- r(N \smallsetminus A)}{|A|}: A \neq \varnothing \} =  \max_{\alpha, \beta}\{ \frac{\beta}{\alpha}: (N,\cI) \text{ is }(\alpha;\beta)\text{-biased}\} .$$ 
This concludes the proof of Proposition \ref{polytopeconj}. 

The proof of Proposition \ref{polytopeconj_variation} is identical to that of Proposition \ref{polytopeconj} with the following substitutions. The polytope $[0,\lambda]^N$ should be replaced by the polytope
$$ \prod_{e\in N} [0, c_e\lambda], \quad \text{ where } \quad c_e = \begin{cases} 1 & \text{ if } e \in R_1 \\ 2 & \text{ if } e \in R_2.\end{cases}$$
The corresponding function $f: 2^N \to \R_{\geq 0}$ is given by $f(S) = \lambda \sum_{e\in S} c_e$. Instances of $|A|$ or $|N\smallsetminus A|$ should be replaced by $\sum_{e\in A} c_e$ or $\sum_{e\in N \smallsetminus A} c_e$, respectively. 

\subsubsection{Positivity}\label{positivity}
To prove assertion \eqref{archthm3} of Theorem \ref{archthm} we first establish one-variable versions of the result. \begin{lemma}\label{positivein1dim}
For all real $0 < \sigma \leq 2$ the Fourier transforms of the following functions are positive or $+\infty$: $$ f(x)= \frac{1}{(\sqrt{x^2+1}+1)^\sigma}, \quad g(x)= \frac{1}{(|x|+1)^\sigma} - \frac{1}{(\sqrt{x^2+1}+1)^\sigma}.$$ 
\end{lemma}
\begin{proof}

When $\xi>0$ we have by contour shifting   \begin{align*} \widehat{f}(\xi) = \int_{-\infty}^\infty \frac{e(-\xi x)}{(\sqrt{x^2+1}+1)^\sigma}\,dx & =  i \int_1^\infty e^{-2 \pi \xi x}\left( \frac{1}{(i \sqrt{x^2-1} +1)^\sigma} -  \frac{1}{(-i \sqrt{x^2-1} +1)^\sigma}\right)\,dx \\
& =   i \int_1^\infty \frac{e^{-2 \pi \xi x}}{x^{2\sigma}}\left((-i \sqrt{x^2-1} +1)^\sigma -  (i \sqrt{x^2-1} +1)^\sigma \right)\,dx \\
& = 2 \int_1^\infty \frac{\sin (\sigma \arctan \sqrt{x^2-1})}{x^\sigma} e^{-2 \pi \xi x}\,dx>0.\end{align*}
The value $\widehat{f}(0)$ is clearly positive if it converges, and if $\xi<0$ we follow the same steps as above, shifting the contour up instead of down.

Now we show $\widehat{g}(\xi)$ is positive or $+\infty$. 
For a real parameter $0 \leq \beta \leq 1$ define \begin{equation}\label{ghatbetaxi}\widehat{g}_\beta(\xi) = \int_{-\infty}^\infty \frac{e(-\xi x)}{(\sqrt{x^2+\beta^2}+1)^\sigma}\,dx,\end{equation} so that $\widehat{g}(\xi) = \widehat{g}_0(\xi) - \widehat{g}_1(\xi).$  
We have \begin{equation}\label{ghateq}\widehat{g}(\xi) =  \int_1^0 \frac{d}{d\beta} \widehat{g}_\beta(\xi)\,d\beta  =  \sigma \int_0^1 \beta \int_{-\infty}^\infty \frac{e(-\xi x) }{\sqrt{x^2+\beta^2} (\sqrt{x^2+\beta^2}+1)^{\sigma+1}}\,dx\,d\beta. \end{equation}
Suppose that $\xi>0$. The interior integral has a branch cut from $-i\beta$ to $-i \infty$. To evaluate the integral, we shift the contour around this branch. We have \begin{align*} & \int_{-\infty}^\infty \frac{e(-\xi x) }{\sqrt{x^2+\beta^2} (\sqrt{x^2+\beta^2}+1)^{\sigma+1}}\,dx\\
= & \int_\beta^\infty e^{-2\pi \xi x} \left( \frac{1}{\sqrt{x^2-\beta^2} (i \sqrt{x^2-\beta^2}+1)^{\sigma+1}} + \frac{1}{ \sqrt{x^2-\beta^2} (-i \sqrt{x^2-\beta^2}+1)^{\sigma+1}} \right)\,dx \\
=& 2 \int_\beta^\infty \frac{ e^{-2\pi \xi x}}{\sqrt{x^2-\beta^2} (x^2-\beta^2+1)^{\frac{\sigma+1}{2}}} \cos ( (\sigma+1) \arctan \left(  \sqrt{x^2-\beta^2}\right)) \, dx \\
= & 2 \int_0^{\pi/2} \frac{e^{-2\pi \xi \sqrt{\tan^2 \theta + \beta^2}}}{\sqrt{\tan^2 \theta + \beta^2}}(\cos \theta)^{\sigma-1} \cos ((\sigma+1) \theta) \,d\theta.\end{align*}
Now we return to \eqref{ghateq}, change order of integration, and change variables $\sqrt{\tan^2 \theta +\beta^2} \to y$ to find
\begin{align*}
\widehat{g}(\xi) 
& = 2 \sigma  \int_0^{\pi/2} \int_{\tan \theta}^{\sec \theta} e^{-2\pi \xi y}\,dy  (\cos \theta)^{\sigma-1} \cos ((\sigma+1) \theta) \,d\theta .
\end{align*}
Let $h_\sigma(\theta)$ be the anti-derivative of $(\cos \theta)^{\sigma-1} \cos((\sigma+1)\theta)$. By integrating by parts we have 
$$\widehat{g}(\xi) = - 2\sigma \int_0^{\pi/2} \frac{1}{\cos^2 \theta} ( (\sin \theta)e^{-2\pi \xi \sec \theta} - e^{-2\pi \xi \tan \theta}) h_\sigma(\theta)\,d\theta.$$
Observe that $(\sin \theta)e^{-2\pi \xi \sec \theta} - e^{-2\pi \xi \tan \theta} \leq 0$, while $h_\sigma(\theta)\geq 0$ for all $0<\sigma \leq 2$ and $0 \leq \theta\leq \pi/2$. Thus $\widehat{g}(\xi)\geq 0$ for $\xi >0$. 
The case $\xi = 0$ is obvious and $\xi<0$ follows by a similar calculation. 
\end{proof}
It follows from Lemma \ref{positivein1dim} that the Fourier transforms of $$\frac{1}{(|x|+1)^\sigma}, \quad \text{ and } \quad \frac{1}{(|x|+1)^\sigma} + \frac{1}{(\sqrt{x^2+1}+1)^\sigma} $$ are also everywhere positive or $+\infty$. 

\begin{lemma}\label{positivein1dim2}For all real $0 < \sigma \leq 2$, $a \in \Z_{\geq 1}$ and $\xi \in \R/\Z$, the Fourier series $$\sum_{\beta \in \Z} \frac{e(\beta \xi)}{(1+a|\beta|)^\sigma}$$ is positive or $+\infty$.\end{lemma}
\begin{proof} Recall the Dirichlet and Fej\'er kernels $$D_u(x) = \sum_{|n| \leq u} e(nx) \quad \text{ and } \quad F_n(x) = \frac{1}{n} \sum_{k=0}^{n-1}D_k(x).$$ By summing by parts twice we have $$\sum_{\beta \in \Z} \frac{e(\beta \xi)}{(1+a|\beta|)^\sigma} = \sum_{n=0}^\infty \left( \frac{1}{(a(n+2)+1)^\sigma} - 2 \frac{1}{(a(n+1)+1)^\sigma} + \frac{1}{(an+1)^\sigma} \right) (n+1) F_{n+1}(\xi).$$ The factor inside the parentheses above is positive by the mean value theorem. The Fej\'er kernel is also positive, and therefore the series is positive wherever it converges. \end{proof}

\begin{lemma}\label{positivein1dim3}
 For all real $0 < \sigma \leq 2$, $a \in \Z_{\geq 1}$ and $\xi \in \C^\times$, the function $$ \sum_{\beta \in \Z} \left( \frac {|\xi|}{\xi}\right)^\beta \int_{\R} \frac{e(-|\xi|x)}{(\sqrt{x^2+ (a\beta)^2} + 1)^{\sigma}}\,dx$$  is positive or $+\infty$.\end{lemma}
 \begin{proof} 
By summation by parts twice and the positivity of the F\'ejer kernel, it suffices to show that the second forward difference in $\beta$ of $$f(\beta,\xi) = \int_{\R} \frac{e(-|\xi|x)}{(\sqrt{x^2+ (a\beta)^2} + 1)^{\sigma}}\,dx$$ is positive. Recall the definition of $\widehat{g}_\beta(\xi)$ from \eqref{ghatbetaxi}, in terms of which we have $$f(\beta+2,\xi) - 2 f(\beta+1,\xi)+f(\beta,\xi) = \int_{\beta}^{\beta+1} \frac{d}{d\gamma} \int_\gamma^{\gamma+1} \frac{d}{d\alpha} \widehat{g}_{a\alpha}(|\xi|)\,d\alpha\,d \gamma.$$ As in the proof of Lemma \ref{positivein1dim}, we have \begin{align*} & f(\beta+2,\xi) - 2 f(\beta+1,\xi)+f(\beta,\xi) \\
& = -\sigma a \int_{\beta}^{\beta+1} \frac{d}{d\gamma} \int_{\gamma}^{\gamma+1}  a \alpha \int_{-\infty}^\infty \frac{e(-|\xi| x)}{\sqrt{x^2 + ( a \alpha)^2} (\sqrt{x^2+( a \alpha)^2}+1)^{\sigma+1}}\,dx \,d\alpha\,d\gamma \\ 
& = -2 \sigma  \int_{\beta}^{\beta+1} \frac{d}{d\gamma}\int_0^{\pi/2} \int_{\sqrt{\tan^2\theta + ( a \gamma)^2}}^{\sqrt{\tan^2 \theta + ( a \gamma + 1)^2}} e^{-2\pi |\xi| y}\,dy  ( \cos \theta)^{\sigma-1} \cos((\sigma+1)\theta)\,d\theta\,d \gamma \\
& = -2 \sigma  a  \int_0^{\pi/2} \int_{\beta}^{\beta+1} \left( \frac{( a \gamma+1) e^{-2\pi |\xi | \sqrt{\tan^2 \theta + ( a \gamma + 1)^2}}}{\sqrt{\tan^2 \theta + ( a \gamma + 1)^2}} - \frac{ a \gamma e^{-2\pi |\xi| \sqrt{\tan^2 \theta + ( a \gamma)^2}}}{\sqrt{\tan^2 \theta + ( a \gamma)^2}} \right)\,d\gamma \\
& \hspace{6cm} \times ( \cos \theta)^{\sigma-1} \cos((\sigma+1)\theta)\,d\theta .\end{align*}
Recall $h_\sigma(\theta)\geq 0$ is defined to be the anti-derivative of $( \cos \theta)^{\sigma-1} \cos((\sigma+1)\theta)$. Let also $$H(\beta,\theta, \xi) =  \int_{\beta}^{\beta+1} \frac{d}{d\theta}\left( \frac{(a\gamma+1) e^{-2\pi |\xi| \sqrt{\tan^2 \theta + (a\gamma + 1)^2}}}{\sqrt{\tan^2 \theta + (a\gamma + 1)^2}} - \frac{a\gamma e^{-2\pi |\xi| \sqrt{\tan^2 \theta + (a\gamma)^2}}}{\sqrt{\tan^2 \theta + (a\gamma)^2}} \right)\,d\gamma .$$
Then, by integrating by parts we have $$ f(\beta+2,\xi) - 2 f(\beta+1,\xi)+f(\beta,\xi) = 2 \sigma  a  \int_0^{\pi/2} H(\beta, \theta, \xi)h_\sigma(\theta)\,d\theta.$$
To prove the lemma, it suffices to show that $H(\beta,\theta, \xi)$ is non-negative for all $\beta \in \N,$ $a \in \Z_{\geq 1}$, $ \xi \in \C^\times $, and $0 \leq \theta \leq \pi/2$. Set $$N(\gamma, \theta, \xi) = \gamma e^{-2\pi |\xi| \sqrt{\tan^2 \theta + \gamma^2}} \tan \theta \sec^2 \theta \left( \frac{2 \pi |\xi | \sqrt{\tan^2 \theta + \gamma^2}  +  1}{ \sqrt{\tan^2 \theta + \gamma^2}^3}\right)$$  
so that
\begin{align}H(\beta,\theta,\xi) & = \int_{\beta}^{\beta+1} (N( a \gamma, \theta,\xi)-N( a \gamma+1,\theta,\xi))\,d\gamma \nonumber \\
& =  a^{-1}  \tan \theta \sec^2 \theta \left( \int_{\sqrt{\tan^2 \theta + ( a \beta)^2}}^{\sqrt{\tan^2 \theta + ( a \beta+1)^2}} - \int_{\sqrt{\tan^2 \theta + ( a (\beta+1))^2}}^{\sqrt{\tan^2 \theta + ( a (\beta+1)+1)^2}} \right) e^{-2 \pi |\xi|y} \left( \frac{2 \pi |\xi| y +1}{y^2}\right) \,dy. \label{3rd_ft_lem_diff_of_ints} \end{align} 
Note that $$\frac{d}{dy} e^{-2 \pi |\xi|y} \left( \frac{2 \pi |\xi| y +1}{y^2}\right) = -\frac{e^{-2\pi |\xi| y}}{y},$$ 
so that the difference of integrals in \eqref{3rd_ft_lem_diff_of_ints} equals by the mean value theorem 
\begin{equation}\label{3rd_ft_lem_lasteq} \left. e^{-2 \pi |\xi| \sqrt{\tan^2 \theta + u^2}} \left( \frac{2 \pi |\xi| \sqrt{\tan^2 \theta + u^2} + 1}{\tan^2 \theta + u^2} \frac{y_1}{\sqrt{ \tan^2 \theta + u^2}}\right)\right|_{y_1}^{y_2}\end{equation}
for some $y_2 \in [a \beta,a\beta+1]$ and $y_1 \in [a(\beta+1), a(\beta+1)+1]$.
Since the function of $u$ in \eqref{3rd_ft_lem_lasteq} is decreasing for all $u>0$ we have that $H(\beta,\theta,\xi)\geq 0$.
\end{proof}

We assume that $k_v$ is real until further notice. Consider the Fourier dual pair $$  \Tt = (\R^\times)^{n_1} \times (S^1)^{n_2} \times (\C^\times)^{n_3} \leftrightarrow  (i\R^{n_1} \times (\Z/2\Z)^{n_1}) \times \Z^{n_2} \times (i\R^{n_3} \times \Z^{n_3}) = \Tt^\wedge $$ and the spaces of tempered distributions $\cS'(\Tt)$ and $\cS'(\Tt^\wedge)$. 
Recall $M$ and $M\varphi$ from section \ref{archimedeanC}. Let $0<\sigma \leq 2$ and $f_i \in \cS'(\Tt^\wedge)$ be given for $i=1, \ldots, m_1+m_2$ by the function $$f_i(w, \epsilon, \alpha, w',\alpha') = \frac{1}{(|(M \varphi)_i|+1)^\sigma},$$ and for $i=m_1+m_2+1,\ldots, m_1+m_2+m_3$ by the function $$f_i(w, \epsilon, \alpha, w',\alpha')=  \frac{1}{(|(M \varphi)_i|+1)^{2\sigma}},$$
where $(M\varphi)_i$ denotes the $i$th entry of $M\varphi$ (see \eqref{Mvarphi}).

Let $\cF_G(f)$ denote the Fourier transform of a bounded measure $f$ on a locally-compact abelian group $G$ following \cite[Ch.\ II \S1 2.\ D\'ef.\ 3]{BourbakiTheoriesSpectrales}.
We have that the Fourier transform $\cF_{\Tt^\wedge}(f_i)$ of each of the functions $f_i \in \cS'(\Tt^\wedge)$ is also a tempered distribution, i.e.\ $\cF_{\Tt^\wedge}(f_i) \in \cS'(\Tt).$  Recall that a tempered distribution $L \in \cS'(\Tt)$ is called positive if for all non-negative-valued $\psi \in \cS(\Tt)$ one has $L(\psi)\geq 0$. We shall next show that $\cF_{\Tt^\wedge}(f_i) \in \cS'(\Tt)$ are positive distributions. 

Let $K\subseteq \Tt^\wedge$ be the kernel of the homomorphism $$(M\varphi)_i: \Tt^\wedge \to \begin{cases} i\R \times \{0,1\} & \text{ if } i=1, \ldots, m_1+m_2 \\ i\R \times \Z & \text{ if } i = m_1+m_2+1 , \ldots, m_1+m_2+m_3.\end{cases}$$    
Then, for any non-negative-valued $\psi \in \cS(\Tt)$ we have $$\cF_{\Tt^\wedge}(f_i)(\psi) := \int_{\Tt^\wedge} f_i(\varphi) \cF_{\Tt}(\psi)(\varphi)\,d\varphi = \int_{\Tt^\wedge / K} f_i(\overline{\varphi})\int_K \cF_{\Tt}(\psi)(\overline{\varphi}k)\,dk \,d\overline{\varphi}.$$ 
Let $\Tt_K:= \{ x \in \Tt: k(x)=1 \text{ for all } k \in K\}$ so that by Poisson summation (see the remark following Lemma \ref{Psum}) we have for almost every $\overline{\varphi} \in \Tt^\wedge/K$ that $$ \int_K \cF_{\Tt}(\psi)(\overline{\varphi}k)\,dk = \int_{\Tt_K} \psi(t) \overline{\overline{\varphi}(t)}\,dt = \cF_{\Tt_K}( \psi)(\overline{\varphi}).$$

Now let us consider the group $\Tt^\wedge/ K$. It is canonically isomorphic to the dual of $\Tt_K$ by Pontryagin duality; also the map $(M\varphi)_i$ identifies $\Tt^\wedge/ K$ with a subgroup of $i\R \times \{0,1\}$ or $i\R \times \Z$. The Fourier transform $\cF_{\Tt^\wedge/K}(f_i)$ may then be considered as a tempered distribution on $\Tt_K$. Since $|f_i(\overline{\varphi})| \leq 1$ for all $\overline{\varphi} \in \Tt^\wedge/K$, we have $|\cF_{\Tt^\wedge/K}(f_i)(g)| \leq \|g\|_{L^\infty(\Tt_K)}$ for all $g \in \cS(\Tt_K)$. By the Hahn-Banach theorem $\cF_{\Tt^\wedge/K}(f_i)$ extends to a linear functional on the continuous functions vanishing at infinity $C_0(\Tt_K)$ satisfying $|\cF_{\Tt^\wedge/K}(f_i)(g)| \leq \|g\|_{L^\infty(\Tt_K)}$ for all $g \in C_0(\Tt_K)$. By the Riesz-Markov-Kakutani theorem (see e.g.\ \cite[6.19 Thm.]{RudinRealandComplex}), the tempered distribution $\cF_{\Tt^\wedge/K}(f_i)$ is given by integration against a regular Borel measure on $\Tt_K$, which we continue to write $\cF_{\Tt^\wedge/K}(f_i)$.  Lemmas \ref{positivein1dim}, \ref{positivein1dim2}, and \ref{positivein1dim3} now show that the measure $\cF_{\Tt^\wedge/K}(f_i)$ is non-negative and has total mass $1$. By the Plancherel theorem (see \cite[Ch.\ II \S1 5.\ Prop.\ 13]{BourbakiTheoriesSpectrales}) we have that $$\cF_{\Tt^\wedge}(f_i)(\psi) = \int_{\Tt^\wedge/ K} f_i(\overline{\varphi}) \cF_{\Tt_K}( \psi)(\overline{\varphi}) \, d\overline{\varphi} = \int_{\Tt_K} \cF_{\Tt^\wedge/K}(f_i)(t) \psi(t)\,dt.$$  Therefore, the tempered distribution $\cF(f_i)$ itself is a positive distribution.

By applying the Hahn-Banach and Riesz-Markov-Kakutani theorems again to $\cF_{\Tt^\wedge}(f_i)$, we obtain a regular non-negative Borel measure $\mu_i$ on $\Tt$ for each $i=1, \ldots m_1+m_2+m_3$ such that $\cF_{\Tt^\wedge}(f_i)(g) = \int_{\Tt} g(x)\,d\mu_i(x)$ for all $g\in \cS(\Tt)$. 
 By assertion \eqref{archthm1} of Theorem \ref{archthm} and the formula \eqref{avsxGeneral} for $A_{F}(s,x)$, the $(m_1+m_2+m_3)$-fold convolution of the measures $\mu_i$ is given by a function and in particular we have $$A_{F}(\sigma, x) = a\left(\bigast_{i=1}^{m_1+m_2+m_3}\mu_i\right)(x)$$ for all $x \in \Tt$, where $a$ is the positive real factor appearing in \eqref{avsxGeneral}. Since the convolution of non-negative measures is non-negative, we have that $A_{F}(\sigma,x)$ takes non-negative values for $\sigma_0<\sigma\leq 2$ and all $x \in T(k_v)$, as was to be shown.

The case that $k_v\simeq \C$ follows in an identical way taking $n_1=n_2=m_1=m_2=0$, $n_3=\dim T$, and $m_3=\dim r$.

\section{Global theory}\label{sec:localtoglobal} 
\subsection{Global analytic conductor}\label{sec:globalconductor}
In this section, we return to the notation of the introduction. That is, we are given a number field $k$ and a $k$-torus $T$ with splitting field $K$ and Galois group $G=\Gal(K/k)$. We choose a finite-dimensional complex algebraic representation $r$ of the $L$-group $\LT$ (see section \ref{sec:Lgroupsandreps} for definitions).
Let $v$ be a place of $k$ and denote by $k_v$ the completion of $k$ at $v$. 
Let $T_v = T \times_k \Spec k_v$ be the base-change of $T$ to $k_v$. For each $v$ there exists a valuation $w$ of $K$ extending $v$ such that $K_w$ is the splitting field of $T_v$. We have that $X^*(T_v)=X^*(T)$ as abelian groups, but where the Galois action on $X^*(T_v)$ is given by restricting the action of $G$ to the embedded copy of $\Gal (K_w/k_v)$. Thus, we obtain an embedding of $L$-groups  $\LT_v \hookrightarrow \LT$ for each place $v$ of $k$. The representation $r$ of $\LT$ then determines representations of each $\LT_v$ by restriction.

An automorphic character $\chi \in \cA(T)$ admits a factorization $$\chi= \bigotimes_v \chi_v$$ in which all but finitely many of the $\chi_v$ are trivial on the maximal compact subgroup of $T(k_v)$.
\begin{definition}\label{GlobalAC}
The analytic conductor $\fc(\chi, r)$ is an invariant of $\cA(T)$ defined by \begin{equation}\label{analc1}\fc(\chi,r) = \prod_v \fc_v(\chi_v,r\vert_{\LT_v}),\end{equation} where $\fc_v(\chi_v,r\vert_{\LT_v})$ are local analytic conductors (see Definitions \ref{artinc1} and \ref{archdef}), all but finitely many of which are equal to $1$.  \end{definition}

\subsection{Algebraic number theory for tori}\label{subsec:localtoglobal}
We next review the standard theorems of algebraic number theory in the context of tori. Let $\cO_w$ and $\cO_v$ be the rings of integers in $K_w$ and $k_v$. Let $S_\infty$ be the set of archimedean places of $k$. Recall the notation
$$T_{\A} :=  \prod_{v \in S_\infty} T(k_v) \times \prod_{v \not \in S_{\infty}} T(\cO_{v})$$ and the global units $U(T)$ of the torus $T$ defined by $$U(T):= T(k) \cap T_{\A}.$$ According to Dirichlet's units theorem \cite{ShyrDsThm}, the abelian group $U(T)$ is finitely generated. We also have a short exact sequence of locally compact Hausdorff abelian groups 
\begin{equation}\label{classnodef} \xymatrix{1 \ar[r] & T(k) T_{\A} \ar[r] & T(\A) \ar[r] & \cl(T) \ar[r] & 1,}\end{equation}
where the cokernel is by definition the (finite) class group of $T$ \cite[Thm.~3.1]{OnoTori}. 
 
Instead of $U(T)$ and $\cl(T)$, we need the following minor variants of these objects. Let $NT(\cO_w)$ be the image of the norm map $N:T(\cO_w) \to T(\cO_v)$.  Let \begin{equation}\label{TNAdef}T_{N,\A} := \prod_{v \in S_\infty} T(k_v) \times \prod_{v \not \in S_{\infty}} NT(\cO_{w}).\end{equation} We call the set  \begin{equation}\label{UNTdef}U_N(T) := T(k) \cap T_{N,\A} \end{equation} the global norm-units of the torus $T$.  Recall that for all unramified non-archimedean places of $k$ we have $NT(\cO_w) = T(\cO_{v})$ by Lemma \ref{Amano}, and for all ramified non-archimediean places we have that $NT(\cO_w)$ is a finite index subgroup of $T(\cO_v)$ by Lemma \ref{keylemma}.  Thus, global norm-units $U_N(T)$ is a finite index subgroup of the global units of the torus $U(T)$ and thus a finitely-generated abelian group itself. We also have the short exact sequence \begin{equation}\label{beginninges} \xymatrix{1 \ar[r] & T(k) T_{N,\A} \ar[r] & T(\A) \ar[r] & \cl_N(T) \ar[r] & 1}\end{equation} with finite cokernel $\cl_N(T)$. We call $\cl_N(T)$ the norm-class group of $T$.

\subsection{Global conductor zeta function}\label{counting}
Recall the (twisted) local generating series $N_F$ and $A_F$ from \eqref{NFsxdef} and \eqref{AFdef}. Let 
\begin{equation}\label{Ysdef}
Y(s):= \frac{1}{\Vol(U_N(T)^\wedge)}\sum_{\theta \in \cl_N(T)^\wedge} \sum_{x \in U_N(T)} \prod_{v \in S_\infty}A_{k_v}(s,x) \prod_{v \not \in S_\infty} N_{k_v}(s,x). 
\end{equation}
In section \ref{locglobandconclusion} we will prove that $Y(s)$ equals the global generating series $Z(s)$ of Theorem \ref{Zsthm} up to a constant depending only on the choices of Haar measures on $\cA(T)$ and $T(k_v)^\wedge$ for $v$ archimedean.  

Recall $A, \widetilde{G}, \widetilde{\Sigma}_0, m=\dim r$ from in section \ref{statementofresults} and $\mathcal{R}=\mathcal{R}(c)$ from \eqref{calRdef}. The goal of this section is the following theorem (cf.\ Theorem \ref{Zsthm}).

\begin{theorem}\label{Ysthm}
Suppose that $r\vert_{\widehat{T}}$ is faithful. The series $Y(s)$ converges absolutely for $\real(s)>A$ and extends to a meromorphic function in the open half plane $\real(s)>A-\min(2^{-1}, m^{-2})$. There exists $c=c(T,r)>0$ such that the function $Y(s)$ \begin{itemize} \item has a pole at $s=A$ of order $|\widetilde{G}\backslash \widetilde{\Sigma}_0|$ and no other poles in $\mathcal{R}(c)$ (respectively, the half-plane $\real(s)>A-\min(2^{-1}, m^{-2})$ if the  Artin conjecture holds), \item grows slowly in $\mathcal{R}$; i.e.\ there exists $J=J(T,r)>0$ and $0<c'=c'(T,r)\leq c$ such that for any $s=\sigma+it \in \mathcal{R}(c')$ avoiding any small neighborhood $U$ of $A$ we have $$Y(\sigma + it) \ll_{T,r,U} (\log (|t|+3))^J,$$ 
and \item  has moderate growth in a vertical strip if the  Artin conjecture holds, i.e.\ there exists $K=K(T,r)>0$ such that for any $s=\sigma+it$ with $\sigma > A-\min(2^{-1},m^{-2})$ avoiding any small neighborhood $U$ of $A$ we have $Z(\sigma+it) \ll_{T,r,\sigma, U} (1+|t|)^K.$ 
\end{itemize}
\end{theorem}
Convention: In Theorem \ref{Ysthm} and throughout the paper, whenever we say ``the  Artin conjecture holds'', we mean that the finitely many (depending on $T,r$) Artin $L$-functions associated to the complex Galois representations in the hypothesis of Theorem \ref{unrthmwithArtin} are entire (up to the possibility of a pole at $s=1$ in the case that the representation is trivial).

Following Tate \cite[\S3.5]{tateNTB} we take $\psi$ to be a non-trivial additive character of $\A/k$ and $dx$ the Haar measure on $\A$ such that $\int_{\A/k} \,dx = 1$.  Let $\psi_v$ be the local component of $\psi$ at a place $v$, $dx = \prod_v dx_v$ be any factorization of $dx$ into a product of local measures such that the ring of integers $\cO_v$ at all but finitely many $v$ gets measure 1, and $\delta(\psi_v)$ be the function defined in \cite[\S3.4.5]{tateNTB}. 

Recall the positive integer $\lambda$ from \eqref{lambda}.  We call the set of finite places $v$ of $k$ for which either \begin{itemize}
\item $v$ is ramified in $K/k$, or 
\item $(\lambda, q_{k_v})\neq 1$
\end{itemize}
 the set of ``bad'' places, and denote that set by $B$. 
Let $$U(s,x) = \prod_{v \not \in B \cup S_\infty}  \left( \delta(\psi_v) dx_v/dx_v'\right)^{sm}N_{k_v}(s,x), \quad R(s,x) =  \sum_{\theta \in \cl_N(T)^\wedge} \prod_{v \in B}\left( \delta(\psi_v) dx_v/dx_v'\right)^{sm} N_{k_v}(s,x),$$ and $$  A(s,x) = \prod_{v \in S_\infty} \left( \delta(\psi_v) dx_v/dx_v'\right)^{sm}A_{k_v}(s,x).$$  
Recalling \eqref{ur1} that $N_{k_v}(s,x)$ does not depend on $\theta_v$ when $v$ is unramified, we have by definition that \begin{equation}\label{Zexpression2} Y(s) = \frac{1}{\Vol(U_N(T)^\wedge)}\prod_{v }  \left( \delta(\psi_v) dx_v/dx_v'\right)^{-sm}  \sum_{x \in U_N(T)} A(s,x)U(s,x)R(s,x).\end{equation}
  By \cite[\S3.5]{tateNTB} the product $$ \prod_{v } \left( \delta(\psi_v) dx_v/dx_v'\right)$$ does not depend on the factorization of the global $dx$, nor on the choice of global additive character $\psi$, but only on $k$. 
To determine the analytic properties of $Y(s)$, it suffices therefore to determine the analytic properties of $A(s,x)$, $U(s,x)$, and $R(s,x)$, and to show that sum over $x \in U_N(T)$ in \eqref{Zexpression2} converges absolutely.

\subsection{Unramified places}\label{unramifiedcounting}
Set \begin{equation*}
U_{\rm max}(s):= \sup_{x \in U_N(T)}|U(s,x)|.
\end{equation*}
\begin{theorem}\label{unrthm}  Suppose that $r\vert_{\widehat{T}}$ is faithful. 
\begin{enumerate}
\item\label{unrthm1} The series $U(s,x)$ converges absolutely and uniformly on compacta in the right half-plane $\real(s)>A$. It admits a meromorphic continuation to $\{s:\real(s)>A-\min(2^{-1}, m^{-2})\}$ and $U_{\rm max}(s)$ is finite whenever $s$ is not equal to a pole of some $U(s,x)$ for any $x \in U_N(T)$.
\item\label{unrthm1prime} There exists an effective constant $0<c \leq \min(2^{-1}, m^{-2})$ depending on $r,T$ but independent of $s,x$ such that the only pole of $U(s,x)$ in the region $\mathcal{R}(c)$ is at $s=A$. 
\item\label{unrthm4prime} There exist effective constants $J>0$ and $0<c'\leq c$ depending on $r,T$ but independent of $s$ such that for all $s \in \mathcal{R}(c')$ avoiding a small neighborhood $N$ surrounding $A$ we have $$U_{\rm max}(s) \ll_{T,r,N}  (\log (|\Imag(s)|+3))^{J}.$$
\item\label{unrthm2} The series $U(s,1)$ has a pole at $s=A$ of order $ | \widetilde{G} \backslash \widetilde{\Sigma}_0 |$ with positive leading constant in its Laurent series expansion.
\item\label{unrthm3} The possible pole of $U(s,x)$ at $s=A$ is of order $\leq | \widetilde{G} \backslash \widetilde{\Sigma}_0 |$, has positive leading coefficient, and Laurent series expansion bounded by that of $U(s,1)$ at $s=A$. 
\end{enumerate}
\end{theorem}
\begin{theorem}\label{unrthmwithArtin}
Suppose that $r \vert_{\widehat{T}}$ is faithful and that the Artin $L$-functions of the irreducible representations $V_i$ of $\Gamma =\Gal(K'/k)$ that arise in \eqref{Videf} are entire up to a possible pole at $s=1$, where $K'$ is given in Definition \ref{K'def}. 
\begin{enumerate}
\item\label{unrthm1primeArtin} The only pole of $U(s,x)$ when $\real(s)>A-\min(2^{-1},m^{-2})$ is at $s=A$. 
\item\label{unrthm4} There exists $K=K(T,r)>0$ such that for all $s$ with $\real(s)>A-\min(2^{-1},m^{-2})$ and avoiding a small neighborhood $N$ surrounding $A$ we have $$U_{\rm max}(s) \ll_{T,r,N,\real(s)} (1+|\Imag(s)|)^{K}.$$ 
\end{enumerate}
\end{theorem}
We devote the rest of this section to proving Theorem \ref{unrthm} and Theorem \ref{unrthmwithArtin}.

First, let us reintroduce some of the structures of section \ref{unramified} in our global context.  
Let $M$ be the multi-set of co-weights of $r$ (recall definition \ref{def_coweights}) and $S$ be a subset of $M$. Let $G_S = \Stab_G S \subseteq G= \Gal(K/k).$ Then $G_S$ acts on $S$, and also on its complement $S^c$. Let $K_S$ be the intermediate field in the extension $K/k$ corresponding to $G_S$ under the Galois correspondence. 

Define the $K_S$-torus $T_S$ by taking its cocharacter lattice to be $\Z^{|S^c|}$, where the coordinates are indexed by $\mu \in S^c$ with $G_S$ acting by permuting these.  Let $\alpha: T_S \to T$ be the map of $K_S$-tori given by the map of cocharacter lattices $$\alpha_*: \Z^{|S^c|} \to X_*(T)$$ \begin{equation}\label{alphadef_global} (0,\ldots, 1,  \ldots, 0) \mapsto \mu,\end{equation} where the $1$ is in the $\mu$th entry and the other coordinates are all 0, see Lemma \ref{equiv_of_cat_multi} and the evaluation pairing \eqref{char_cochar_PP}. Given a $K_S$-point $x$ of $T$, let $\alpha^{-1}(x)$ be the fiber of $\alpha$ over $x$ and let $p(\alpha^{-1}(x))$ denote its set of irreducible components. 
Recall the constant $\lambda$ from \eqref{lambda}. 
\begin{definition}\label{K'def}
Let $K'/k$ be the minimal Galois extension over which all geometric components of $\alpha^{-1}(x)$ for all $S$ and $x \in U_N(T)$ are defined and which contains the $\lambda$th roots of unity.
\end{definition}
Note that Lemma \ref{agbnd} guarantees the existence of $K'$, and that the Galois extension $\widetilde{K}=K(\mu_\lambda)$ is contained in $K'$, where $K$ is the splitting field of $T$ and $\mu_\lambda$ is the group of $\lambda$th roots of unity.
Let $\Gamma =  \Gal(K'/k) .$ The finite group $\Gamma$ acts on the set of irreducible components $p(\alpha^{-1}(x))$ for any $x$. We will show later that if $x=1$, the action of $\Gamma$ on $p(\alpha^{-1}(1))=\pi_0(\ker \alpha)$ factors through $\Gamma \to \widetilde{G}$.

 Now let us consider a place $v$ of $k$, which we assume for the rest of this section satisfies $v \not \in S_\infty \cup B$. To each such place there is associated a prime ideal $\fp$ of $k$, and we have $q_{k_v}=N(\fp)$, the absolute ideal norm of $\fp$.  
The valuation $w$ extending $v$ introduced in section \ref{sec:globalconductor} determines a unique prime $\fP$ of $K$ lying over $\fp$. If $\Fr_\fP \in G$ fixes $S$, then since $\fp$ is unramified we may base change the $K_S$-tori $T_S$ and $T$ to $k_v$, so that the results of section \ref{unramified} apply to $T_{S, k_v}$, $T_{k_v}$, and $\alpha^{-1}(x)_{k_v}$ with the decomposition group $D_\fP\subseteq G$ at $\fp$ playing the role of the local Galois group $\Gal(L/F)$.

For each such $\fP$ we also choose a prime $\fP'$ of $K'$ lying over $\fP$.  
The decomposition group $D_{\fP'} \subseteq \Gamma$ at $\fp$ plays the role of the local Galois group $\Gal(L'/F)$ from section \ref{unramified}. 
All of the geometric components of $\alpha^{-1}(x)_{k_v}$ for all $x \in U_N(T)$ are defined over $\cO_{\fP'}$ and can be base-changed to the finite residue field of the completion $K'_{\fP'}$.

We have the restriction map $\Gamma \to G$ under which $ \Fr_{\fP'} \mapsto \Fr_{\fP}.$ When $x=1\in U_N(T)$ we will also use the map \begin{equation}\label{Gtilde_injection_body} \widetilde{G} \hookrightarrow G \times \Gal(k(\zeta_\lambda)/k)\end{equation} $$\Fr_{\fP'} \mapsto (\Fr_{\fP},N(\fp)),$$  where $N(\fp) \in \Gal(k(\zeta_\lambda)/k)$ denotes the automorphism of $k(\zeta_\lambda)$ sending a primitive $\lambda$th root of unity $\zeta_\lambda$ to $\zeta_\lambda^{N(\fp)}$  (see \eqref{Gtildeinclusion}).  In particular, the value of $N(\fp)$ modulo $\lambda$ is determined by $\Fr_{\fP'}\in \widetilde{G}$.

By \eqref{rhotoc} we have $$U(s,x) 
 =  \prod_{\fp \not \in B} \sum_{c\in \N^M} \frac{\Pi_=(c,x)}{N(\fp)^{s|c|}} .$$ 
\begin{definition}\label{approxdef}
If $F_1(s)$ and $F_2(s)$ are meromorphic functions defined in $\real(s)>\sigma_i$, $i=1,2$ and there exists an analytic function $G(s)$ given by an absolutely and uniformly convergent Euler product in $\real(s)> \sigma_0$ such that $F_1(s) = G(s)F_2(s)$, then we say that $F_1$ equals $F_2$ \emph{up to an absolutely convergent Euler product} in $\real(s) > \sigma_0$ and write $F_1 \approx F_2$ in $\real(s) > \sigma_0$.\end{definition}
\begin{lemma}\label{approxlem}
Suppose $F_1$ and $F_2$ are as in Definition \ref{approxdef}. 
\begin{enumerate}
\item If $F_1 \approx F_2$ in $\real(s)>\sigma_0$, then $F_1$ admits a meromorphic continuation to $\real(s)>\max(\sigma_0, \min(\sigma_1,\sigma_2)).$ 
\item The relation $\approx$ defines an equivalence relation on meromorphic functions on $\real(s)>\max(\sigma_0, \min(\sigma_1,\sigma_2)).$
\item If $F_1 \approx F_2$ in $\real(s)>\sigma_0$, then $F_1$ and $F_2$ have the same poles to the same orders in the domain $\real(s)>\max(\sigma_0, \min(\sigma_1,\sigma_2)).$
\end{enumerate}
\end{lemma}
The proofs are easy exercises, so we omit them. Since we do not give an expression for the leading constant in Theorem \ref{MT1}, it suffices to study $U(s,x)$ up to the $\approx$ equivalence.

Recall the change of variables $c \leftrightarrow S$ given by \eqref{ctoS1}, \eqref{ctoS2}, and let $$U_0(s,x)  =  \prod_{\fp \not \in B } \sum_{ S\subseteq M} \frac{\Pi_=(S,x)}{ N(\fp)^{s|S|}}.$$ 
\begin{lemma}\label{UU0} We have $U(s,x) \approx U_0(s,x)$ in $\real(s)>A-(2m)^{-1}$.  \end{lemma}
\begin{proof}
It suffices to show that the series \begin{equation}\label{01prodratio2}\prod_{\fp \not \in B}  \frac{ \sum_{c \in \N^M} \Pi_=(c,x)N(\fp)^{-s|c|}}{\sum_{c \in \{0,1\}^M} \Pi_=(c,x)N(\fp)^{-s|c|}}\end{equation} converges absolutely and uniformly on compacta in $\real(s) >A - \frac{1}{2m}$.

Consider $\fp \not \in B$ with $N(\fp)$ sufficiently large.  The factor of the product \eqref{01prodratio2} at $\fp$ is \begin{equation}\label{01prod3} 
1+  \sum_{\substack{c \in \N^M \\ \max c_\mu \geq 2}} \frac{\Pi_=(c,x)}{N(\fp)^{s|c|}} - \left( \sum_{\substack{c \in \N^M \\ \max c_\mu \geq 2}} \frac{\Pi_=(c,x)}{N(\fp)^{s|c|}}\right) \left(\sum_{\substack{c \in \{0,1\}^M\\ c \neq 0}} \frac{\Pi_=(c,x)}{N(\fp)^{s|c|}}\right) + \cdots\end{equation}

We may assume that any $c\in \N^M$ appearing in one of the sums in \eqref{01prod3} is $D_\fP$-fixed, since otherwise $\Pi_=(c,x)= 0$ by Lemma \ref{reducedlemma}.
 We have the trivial bounds $|\Pi_=(c,x)| \leq \Pi_=(c,1) \leq \Pi_\leq(c,1)$ and by Proposition \ref{MurP} and Lemma \ref{Acalc} the bound $$ \Pi_\leq(c,1) \ll_{T,r} \prod_{k=0}^\infty N(\fp)^{\dim D_k(c)}.$$ Note that $$|c| = \sum_{k=0}^\infty |\{\mu \in M : c_\mu > k\}|,$$ so $$ \frac{\Pi_=(c,x)}{N(\fp)^{s|c|}} \ll \prod_{k=0}^\infty N(\fp)^{\dim D_k(c) -s|\{\mu: c_\mu > k\}|}.$$ Therefore we have e.g.~for the first sum in \eqref{01prod3} that $$\sum_{\substack{c \in \N^M \\ \max c_\mu \geq 2}} \frac{\Pi_=(c,x)}{N(\fp)^{s|c|}} \ll \sum_{\substack{c\, :\, \max c_\mu \geq 2 \\ \Pi_=(c,x) \neq 0}} \prod_{k=0}^\infty N(\fp)^{\dim D_k(c) -s|\{\mu: c_\mu > k\}|}.$$ By Lemma \ref{dimDkgeq1} we have the product here has at least two non-one factors for each $c$ in the outer sum. 
 Now we take the product of \eqref{01prod3} over $\fp \not \in B$, and find that \eqref{01prodratio2} converges absolutely and uniformly on compacta in the region \begin{equation}\label{repititions}\real(s)> \sup_{i\geq 2} \max_{\substack{1 \leq j \leq i \\ S_j \subseteq M}} \{ \frac{ \dim D(S_1) + \cdots + \dim D(S_i) + 1}{|S_1| + \cdots + |S_i|} :  D(S_j) \neq \{1\}\}.\end{equation}
 
\begin{lemma}\label{aa'cc'}
Let $i \geq 2$. For any real numbers $a_1, \ldots, a_i$, and any $1\leq c_1, \ldots, c_i\leq m$ one has $$ \frac{a_1+\cdots a_i+1}{c_1+\cdots + c_i} \leq  \max_{j=1,\ldots,i} \Big\{ \frac{a_i+1}{c_i}\Big\} - (2m)^{-1}.$$
\end{lemma}
\begin{proof} It suffices to show the case that $i=2$. 
Suppose without loss of generality that $\frac{a_1+1}{c_1}\geq \frac{a_2+1}{c_2}$, i.e.\ $c_2a_1-c_1a_2+c_2 \geq c_1$. Then,
$$\frac{a_1+1}{c_1} -  \frac{a_1+a_2+1}{c_1+c_2} = \frac{c_2+a_1c_2-c_1a_2}{c_1(c_1+c_2)} \geq \frac{c_1}{c_1(c_1+c_2)}\geq \frac{1}{2m}.$$
\end{proof}
By Lemma \ref{aa'cc'}, we conclude that \eqref{01prodratio2} converges absolutely and uniformly in the region $\real(s)>A-(2m)^{-1}$.
\end{proof}

We arrange the product $U_0(s,x)$ over (finitely many) conjugacy classes $C \subseteq \Gamma$: 
\begin{equation}\label{splitovercc}U_0(s,x)  = \prod_{C\subseteq \Gamma} \prod_{\substack{\fp \not \in B \\ \Fr_{\fP'} \in C}} \sum_{ S\subseteq M} \Pi_=(S,x) N(\fp)^{-s|S|}.\end{equation}

\begin{lemma}\label{F1F2F3}Let \begin{equation}\label{SigmaP}{\Sigma^\fP} = \{S \subseteq M : S \neq \varnothing, \quad \Fr_\fP S = S, \quad \frac{\dim D(S) + 1}{|S|}\geq A\}.\end{equation} 
We have $$ U_0(s,x) \approx \prod_{C \subseteq \Gamma} \prod_{\substack{\fp \not \in B \\ \Fr_{\fP'} \in C}}  \left( 1+\sum_{S \in \Sigma^{\fP}} \Pi_=(S,x) N(\fp)^{-s|S|} \right)$$ in $\real(s)> A-\min(2^{-1}, m^{-2})$.\end{lemma}

\begin{proof}  By Lemma \ref{reducedlemma} we have if $\Fr_{\fP} S \neq S$ then $\Pi_=(S,x)=0$. Consider the sets $\varnothing \neq S\subseteq M$ which satisfy $\Fr_{\fP} S = S$, and take 
$$ (2^M \smallsetminus \{\varnothing\})^{D_\fP} = {\Sigma^\fP} \sqcup {\Sigma^\fP}^c.$$ 
Let
$$F_1(s) = \prod_{C \subseteq \Gamma}\prod_{\substack{\fp \not \in B \\ \Fr_{\fP'} \in C}}  \left( 1+ \sum_{S \in {\Sigma^\fP}} \Pi_=(S,x) N(\fp)^{-s|S|}\right),$$ and 
   $$F_2(s) = \prod_{C \subseteq \Gamma}\prod_{\substack{\fp \not \in B \\ \Fr_{\fP'} \in C}} \left(1+ \sum_{S \in {\Sigma^\fP}^c}  \Pi_=(S,x) N(\fp)^{-s|S|}\right).$$
 By Lemma \ref{aa'cc'} we have $$\prod_{C \subseteq \Gamma}\prod_{\substack{\fp \not \in B \\ \Fr_{\fP'} \in C}}  \sum_{ S\subseteq M} \Pi_=(S,x) N(\fp)^{-s|S|} \approx F_1(s)F_2(s)$$ in $\real(s)>A-(2m)^{-1}$.
   Note that for any $S\in {\Sigma^\fP}^c$ we have \begin{equation}\label{largerdomain}A-\frac{\dim D(S) + 1}{|S|} \geq m^{-2},\end{equation} since $A$ and $\frac{\dim D(S) + 1}{|S|}$ are two elements of the Farey sequence of order $m$.
   We have by the estimate $|\Pi_=(S,x)| \leq \Pi_\leq(S,1)$, Lemma \ref{Acalc}, and \eqref{largerdomain} that $F_2(s)$ converges absolutely and uniformly in the region $\real(s) \geq A-\min(2^{-1}, m^{-2})$. Thus, $U_0(s,x) \approx F_1(s)$ in $\real(s)>A-\min(2^{-1}, m^{-2})$. 
   \end{proof}

Recall the quantity $a(S,x)$ from \eqref{aSx}. Similarly, let \begin{equation}\label{aCSx}a_C(S,x) = \#\{y \in p(\alpha^{-1}(x)): \Fr_{\fP'} y = y\},\end{equation} which only depends on the conjugacy class $C\subseteq \Gamma$ of $\Fr_{\fP'}$. Likewise, recall $a(S)$ from \eqref{aS} and let \begin{equation}\label{aCS}a_C(S) = \#\{ y \in \pi_0(D(S)): \Fr_{\fP} y^{N(\fp)} = y\},\end{equation} which as well only depends on the conjugacy class $C \subseteq \Gamma$ of $(\Fr_{\fP},N(\fp))$ via \eqref{Gtilde_injection_body}. 
We have by Lemmas \ref{UU0}, \ref{F1F2F3} and \ref{A*calc} \begin{equation}\label{F1}U(s,x)  \approx \prod_{C\subseteq \Gamma}\prod_{\substack{\fp \not \in B \\ \Fr_{\fP'} \in C}}  \left( 1+\sum_{S \in {\Sigma^\fP}} N(\fp)^{\dim D(S) - s|S|} \begin{cases}a_C(S,x) & \text{ if } \dim D(S) \geq 1 \\  a_C(S,x) -1 & \text{ if } \dim D(S) =0 \end{cases} \right) \end{equation} in $\real(s)>A-\min(2^{-1}, m^{-2})$, where we may replace $a_C(S,1)$ by $a_C(S)$ when $x=1$. Note that if $\frac{\dim D(S) +1}{|S|} > A$ then $D(S) = \{1\}$ by the definition of $A$, and so $a_C(S,x)=1$ and the term corresponding to $S$ above vanishes.

 We split up the sum in \eqref{F1} over the possible values of $\dim D(S), |S|$. The parameter space is $$P= \{(a,b): 0 \leq a \leq n, 1 \leq  b \leq m, \frac{a+1}{b}= A\}.$$ Since the case $a=0$ is different in \eqref{F1}, so we also introduce $$ P_0: =\{(a,b): 1 \leq a \leq n, 1 \leq  b \leq m, \frac{a+1}{b}= A\} .$$ If $A=1$, then $P = P_0\cup \{(0,1)\}$, and otherwise $P=P_0$.  Let also \begin{equation}\label{Sigmaab}\Sigma_{a,b} = \{S \in \Sigma : \dim D(S) =a, \quad |S|=b\},\end{equation} so that $$\Sigma = \bigsqcup_{(a,b)\in P} \Sigma_{a,b}.$$ Note that $\Sigma_{0,1/A}$ is non-empty only if $A=1$. With this notation, we have that \begin{equation}\label{usxbeforereptheory} U(s,x)  \approx \prod_{C\subseteq \Gamma}\prod_{\substack{\fp \not \in B \\ \Fr_{\fP'} \in C}}    \sum_{\substack{S \in \Sigma_{0,1} \\ \Fr_{\fP} S=S}} \big(1 + N(\fp)^{-s}\big)^{a_C(S,x) -1}  \prod_{(a,b) \in P_0} \sum_{\substack{S \in \Sigma_{a,b} \\ \Fr_{\fP} S=S}} \big(1 + N(\fp)^{a-bs}\big)^{a_C(S,x)} \end{equation} 
 in $\real(s)> A-\min(2^{-1}, m^{-2})$, with optionally $a_C(S)$ in place of $a(S,1)$ if $x=1$.
 Let \begin{equation}\label{Sigmaabtilde}\widetilde{\Sigma}_{a,b} = \begin{cases} \{(S,y): S \in \Sigma_{a,b}, \quad y \in p(\alpha^{-1}(x))\} & \text{ if } (a,b) \neq (0,1) \\
 \{(S,y): S \in \Sigma_{0,1}, \quad y \in p(\alpha^{-1}(x)), \quad y \neq 1\} & \text{ if } (a,b) = (0,1). \end{cases}\end{equation}

The group $\Gamma$ acts on each $\widetilde{\Sigma}_{a,b}$ through $G$ on $S$ and through its Galois action on $p(\alpha^{-1}(x))$.  If $x=1$, the action factors through the action of $(\overline{g},\gamma) \in G \times (\Z/\lambda \Z)^\times$  on elements $(S,y)$ with $y\in \pi_0(D(S))$ and is given by $(\overline{g},\gamma) .y = \overline{g}y^{\gamma}.$ 

Let us further decompose the action of $\Gamma$ on $\widetilde{\Sigma}_{a,b}$ into orbits $\widetilde{\Sigma}_{a,b} = \bigsqcup \cO.$ Let $V_\cO$ be the permutation representation of $\Gamma$ acting (transitively) on $\cO$. Let $\psi_\cO$ be its character and $C$ a conjugacy class of $\Gamma$. Then $\psi_{\cO}(C)$ is the number of $\Fr_{\fP'}$-fixed points on $\cO$. In these terms, we have $$ U(s,x)  \approx \prod_{C\subseteq \Gamma} \prod_{\substack{\fp \not \in B \\ \Fr_{\fP'} \in C}}   \prod_{(a,b) \in P} \prod_{\cO \subseteq \widetilde{\Sigma}_{a,b}}  \left(1 + N(\fp)^{a-bs}\right)^{\psi_{\cO}(C)}. $$

Now we decompose $V_{\cO}$ into irreducible representations $V_i$ of $\Gamma$ so that \begin{equation}\label{Videf}V_{\cO} = \bigoplus_i V_{i}^{\oplus m_{\cO,i}} \quad \text{ and } \quad \psi_{\cO} = \sum_i m_{\cO,i} \psi_i \end{equation} for some $m_{\cO,i} \in \N$, where $\psi_{i}$ is the character of $V_i$. Then we have $$  U(s,x)  \approx \prod_{C\subseteq \Gamma} \prod_{\substack{\fp \not \in B \\ \Fr_{\fP'} \in C}}  \prod_{(a,b) \in P} \prod_{\cO \subseteq \widetilde{\Sigma}_{a,b}} \prod_i  \left(1 + \psi_i( C)N(\fp)^{a-bs}\right)^{m_{\cO,i}} $$ in $\real(s)>A-\min(2^{-1}, m^{-2})$.

Moving the products over $C$ and $\fp$ to the inside, we have\begin{equation}\label{nearfinal}U(s,x) \approx \prod_{(a,b) \in P}  \prod_{\cO \subseteq \widetilde{\Sigma}_{a,b}} \prod_i L^{(B)}(bs-a, V_i)^{m_{\cO,i}}\end{equation} in $\real(s)>A-\min(2^{-1}, m^{-2})$, where $L^{(B)}(bs-a, V_i)$ denotes the Artin $L$-function attached to $V_i$ with archimedean primes and primes in $B$ omitted. 

 By Lemma \ref{agbnd}, the group $\Gamma$ is finite and does not depend on $x\in U_N(T)$. Therefore there are only finitely many irreducible representations $V_i$ of $\Gamma$, and so the product over $i$ in \eqref{nearfinal} has finitely many factors, bounded uniformly in terms of $x$. By the first part of Lemma \ref{agbnd}, the dimension of $V_{\cO}$ is uniformly bounded in terms of $x$. Therefore, the multiplicities $m_{\cO,i}$ as well as the degree and conductor of the (partial) Artin $L$-functions on the right hand side of \eqref{nearfinal} remain uniformly bounded as $x$ varies over $U_N(T)$ (but depend on $r$ and $T$, of course). 

By the Brauer induction theorem and class field theory, for each $V_i$ there exist finitely many intermediate fields $k \subseteq K_j \subseteq K'$, Hecke characters $\chi_j$ of $K_j$ and multiplicities $m_j \in \Z$, $j=1, \ldots$ such that \begin{equation}\label{BrauerInduction}L(s,V_i) = \prod_j L(s,\chi_j)^{m_j}.\end{equation} Part \eqref{unrthm1} of Theorem \ref{unrthm} now follows from the analytic continuation of Hecke $L$-functions and the uniformity statements of the preceding paragraph. 

Although it is expected that the critical zeros of any $L(s,\chi_j)$ appearing in \eqref{BrauerInduction} with $m_j<0$ cancel with the critical zeros of some $L(s,\chi_{j'})$ in \eqref{BrauerInduction} with $m_{j'}>0$ (the  Artin conjecture), at present we cannot exclude the possibility that those $j$ with $m_j<0$ contribute infinitely many poles to $L(s,V_i)$ inside the critical strip. The hypothesis in Theorem \ref{unrthmwithArtin} asserts that precisely such a possibility does not occur, in which case assertion \eqref{unrthm1primeArtin} of Theorem \ref{unrthmwithArtin} is immediate.

 Applying \eqref{BrauerInduction} to \eqref{nearfinal}, assertion \eqref{unrthm1prime} of Theorem \ref{unrthm} follows from the zero-free region for Hecke $L$-functions in Lemma \ref{ZFHforHeckeL} upon taking \begin{equation}\label{cconstruction}c = \min \{\min_{(a,b) \in P} \frac{1}{b(\log b+1)}\min_{\cO,i} \min_{j: m_j<0} c(\chi_j), 2^{-1}, m^{-2}\},\end{equation} which is in particular independent of $x$. Moreover, applying the lower bounds of Lemma \ref{HeckeLowerBound} to any $L(s,\chi_j)^{m_j}$ with $m_j<0$ and the upper bounds of \cite[Thm.\ 1]{Coleman} to any $L(s,\chi_j)^{m_j}$ with $m_j>0$ in \eqref{BrauerInduction}, we obtain assertion \eqref{unrthm4prime} of Theorem \ref{unrthm} with $c'$ constructed from $c'(\chi_j)$ as in \eqref{cconstruction}.

Since $L(1, \chi) \neq 0$ for any Hecke characters $\chi$ (see Lemma \ref{ZFHforHeckeL}), the number of poles at $s=A$ appearing in \eqref{nearfinal} is equal to the number of trivial representations appearing among the representations $V_i$, counted with multiplicity. By e.g.~Serre \cite[\S 2.3 exercise~2.6]{SerreLinearReps}, the number of trivial characters is equal to the number of orbits $$\sum_{a,b \in P} | \Gamma \backslash \widetilde{\Sigma}_{a,b}|.$$ If $x=1$ then this matches $| \widetilde{G} \backslash \widetilde{\Sigma}_0|$ as defined in the introduction. Thus we have established part \eqref{unrthm2} of Theorem \ref{unrthm}. 

For $s$ with $\real(s)>A$ we have that $|U(s,x)| \leq U(\real(s), 1)$.  Therefore $$ \sum_{a,b \in P} | \Gamma \backslash \widetilde{\Sigma}_{a,b}| \leq | \widetilde{G} \backslash \widetilde{\Sigma}_0|$$ for any $x\in U_N(T)$, and the Laurent series expansion of $U(s,x)$ around $s=A$ is bounded in absolute value by that of $U(s,1)$. By e.g. \eqref{usxbeforereptheory}, the leading constant in the Laurent expansion is positive real. This establishes part \eqref{unrthm3} of Theorem \ref{unrthm}.

Lastly, if the  Artin conjecture holds, the Phragmen-Lindel\"of convexity principle applied to the strip $-1\leq \real(s) \leq 2$, say, asserts that the $L$-functions $L(s,V_i)$ have at most polynomial growth in their critical strips. Therefore, there exists $K= K(r,T)>0$ independent of $x$ so that when $s$ has $\real(s)>A-\min(2^{-1}, m^{-2})$ and $s$ avoids a small neighborhood $N$ around $s=A$, we have \begin{equation}\label{hooray}U(s,x)\ll_{T,r,N,\real(s)} (1+|\Imag(s)|)^{K},\end{equation} 
establishing assertion \eqref{unrthm4} of Theorem \ref{unrthmwithArtin}.

\subsection{Ramified places}\label{finitelymany}
\begin{lemma}\label{RamLemma}
Suppose that $r\vert_{\widehat{T}}$ is faithful. The function $R(s,x)$ converges absolutely and uniformly on compacta in the region $\real(s)>A-m^{-2}$. It takes a positive value at the point $s=A$. \end{lemma}
\begin{proof} Recall that \begin{equation}\label{Rstart} R(s,x)= \sum_{\theta \in \cl_N(T)^\wedge} \prod_{v \in B}\left( \delta(\psi_v) dx_v/dx_v'\right)^{sm} N_{k_v}(s,x).\end{equation}  We saw in section \ref{ramified} for each $v \in B$ that $\left( \delta(\psi_v) dx_v/dx_v'\right)^{sm} N_{k_v}(s,x)$ converges absolutely and uniformly on compacta in $$ \real(s) > \max \{ \frac{\dim D(S)}{|S|}: D(S)\neq \{1\}\} \geq A-m^{-2}, $$ a region which includes the point $s=A$. Since the sum over class group characters in \eqref{Rstart} is finite and $B$ consists of finitely many places, the function $R(s,x)$ converges absolutely absolutely and uniformly on compacta in the region $\real(s)>A-m^{-2}$. The first statement of  \eqref{RamLemma} holds.

We now show the second assertion of Lemma \ref{RamLemma}. The product over $v \in B$ is a finite product. Let us enumerate the places appearing as $v_1,\ldots,v_s$.  Let us denote by $$ c \mid |B|^\infty$$ the set of all positive integers $c$ of the form $q_{k_{v_1}}^{n_1} \cdots q_{k_{v_s}}^{n_s}$, for $n_1,\ldots, n_s \in \N$.  In this paragraph, we write $\mu(d)$ for the M\"obius function defined with respect to numbers $d \mid |B|^\infty$, that is, where $q_{k_{v_i}}$ plays the role of the primes. For $c \mid |B|^\infty$ we let $$ H_c^*= \{(\theta, \chi_{v_1}, \ldots, \chi_{v_s}) \in \cl_N(T)^\wedge \times \prod_{v  \in B} NT(\cO_w)^\wedge : \prod_{i=1}^s q_{k_{v_i}}^{c_r(\chi_{v_i}\theta_{v_i})} = c\},$$ where $c_r$ is the Artin conductor associated to $\chi_{v_i}\theta_{v_i}$ via $r$. We also define $$ H_c= \{(\theta, \chi_{v_1}, \ldots, \chi_{v_s}) \in \cl_N(T)^\wedge \times \prod_{v  \in B } NT(\cO_w)^\wedge : \prod_{i=1}^s q_{k_{v_i}}^{c_r(\chi_{v_i}\theta_{v_i})} \mid c\}.$$ The set $H_c$ is a group.  We have the relations $$H_c = \bigcup_{d\mid c} H_d^*, \quad \text{ and } \quad H_c^* = \bigcup_{d \mid c} \mu(d) H_{c/d},$$ where the unions are over integers $d | |B|^\infty$ which also divide $c$, and $\mu(d)$ is defined as before.  

In terms of these definitions, we have  \begin{align*} R(s,x) = & \sum_{c \mid |B|^{\infty}} \frac{1}{c^s} \sum_{(\theta, \chi_{v_1}, \ldots, \chi_{v_s}) \in H_c^*} \chi_{v_1}(x) \cdots \chi_{v_s}(x) \\
= &  \sum_{c \mid |B|^{\infty}} \frac{1}{c^s} \sum_{d \mid c} \mu(d) \sum_{(\theta, \chi_{v_1}, \ldots, \chi_{v_s}) \in H_{c/d}} \chi_{v_1}(x) \cdots \chi_{v_s}(x) \\
= &  \left( \sum_{c \mid |B|^{\infty}} \frac{1}{c^s}\right)^{-1} \sum_{c \mid |B|^{\infty}} \frac{1}{c^s} \sum_{(\theta, \chi_{v_1}, \ldots, \chi_{v_s}) \in H_{c}} \chi_{v_1}(x) \cdots \chi_{v_s}(x).\end{align*}
The first factor is evidently positive real for all $s>0$, since it is a product over finitely many ``primes''.  By orthogonality of characters, the second is a Dirichlet series with non-negative integer coefficients, hence it takes a positive real value wherever it converges absolutely.  
\end{proof}

\subsection{Global convergence}\label{conclusion}
In this section, we prove Theorem \ref{Ysthm}. 
First we show that the sum in \eqref{Zexpression2} converges absolutely and uniformly on compacta in $\real(s)> A-\min(2^{-1}, m^{-2})$. 

For each $v \in S_\infty$, let $$\sigma_v: T(k)\hookrightarrow \Tt = \begin{cases} (\R^\times)^{n_1} \times (S^1)^{n_2} \times (\C^\times)^{n_3} & \text{ if } v \text{ real} \\ (\C^\times)^{n} & \text{ if } v \text{ complex}\end{cases}$$ be the corresponding embedding. Let $S_{1,\infty}$ and $S_{2,\infty}$ be the set of real and complex archimedean places, respectively. By formula \eqref{Zexpression2}, Theorems \ref{archthm}, \ref{unrthm}, and Lemma \ref{RamLemma} we have for any $s$ not coinciding with any pole of any $U(s,x)$ and with $\real(s)>A-\min(2^{-1}, m^{-2})$ that \begin{multline}\label{secondtolast}  Y(s)  \ll_{r,T, \nu} \sum_{x \in U_N(T)} |U(s,x)||R(s,x)|  \\ \times \prod_{v \in S_{1,\infty}} \prod_{j=1}^{n_1} \frac{1}{1+ | \log |(\sigma_v x)_j| |}\prod_{j=n_1+n_2+1}^{n_1+n_2+n_3} \frac{1}{1+ 2| \log |(\sigma_v x)_j| |}\prod_{v \in S_{2,\infty}} \prod_{j=1}^{n} \frac{1}{1+ 2| \log |(\sigma_v x)_j| |} .\end{multline} 
Let $J$ denote the set of pairs $(v,j)$ appearing on the right hand side of \eqref{secondtolast}, i.e.\ $$J = \bigcup_{v \in S_{1, \infty}} \{ (v,j): j=1,  \ldots, n_1,n_1+n_2+1, \ldots, n_1+n_2+n_3\} \cup \bigcup_{v \in S_{2,\infty}} \{ (v,j): j=1, \ldots, n\}.$$ 
For any $(v,j)\in J$, we set
$$ \xi_j^v = {\rm pr}_{j} \circ \sigma_v : T(k) \to \begin{cases} \R^\times & \text{ if } v \in S_{1, \infty} \text{ and } j=1, \ldots, n_1 \\ 
\C^\times & \text{ if } v \in S_{1, \infty} \text{ and } j=n_1+n_2+1, \ldots, n_1+n_2+n_3 \\ 
\C^\times & \text{ if } v \in S_{2, \infty} \text{ and } j=1, \ldots, n\end{cases}$$  to be the composition of $\sigma_v$ with the projection to the $j$th entry of $\Tt$. 

Since $R(s,x)$ is a finite sum, we have the bound $|R(s,x)|\leq R(\sigma,1)$, where $s=\sigma+it$. 
By positivity, we extend the sum over $U_N(T)$ in \eqref{secondtolast} to $U(T)$. Since roots of unity have absolute value 1, the summand only depends on $U(T)/U(T)_{\rm tors}$, so we reduce to this at the cost of a factor depending only on $T$. Thus for $s$ with $\real(s)> A-\min(2^{-1},m^{-2})$ \begin{equation}\label{last}Y(s)  \ll_{r,T, \nu} U_{\rm max}(s) R(\sigma,1) \sum_{x \in U(T)/U(T)_{\rm tors}}  \prod_{(v,j) \in J} \frac{1}{1+|\log | \xi_j^v (x)|_v|}. \end{equation}

The set $\{\xi_j^v : (v,j)\in J\}$ forms a basis for $\prod_{v \in S_\infty} X^*(T)^{G_{k_v}}$. Following standard notation, we write $r_\infty= |J|$ for the rank of this finite-rank free abelian group, i.e.\ $$r_\infty = \rank \prod_{v \in S_\infty} X^*(T)^{G_{k_v}} = \sum_{v \in S_\infty} \rank X^*(T)^{G_{k_v}}= |S_{1,\infty}|(n_1+n_2) + |S_{2, \infty}|n.$$ Similarly, let $r_k =\rank X^*(T)^{G_k}$. By Dirichlet's units theorem for tori \cite{ShyrDsThm}, we have $\rank U(T) = r_\infty-r_k$. Accordingly, let $\{\epsilon_1, \ldots, \epsilon_{r_\infty-r_k}\}$ be a $\Z$-basis for $U(T)/U(T)_{\rm tors}$. The $r_\infty \times (r_\infty - r_k)$ matrix 
$$\Phi = ( \log | \xi_j^v(\epsilon_i)|_v)_{(v,j)\in J, i=1, \ldots, r_\infty-r_k}$$ 
is called the \emph{regulator matrix} of $T$.

It is a key fact in what follows that \emph{any} choice of $r_\infty-r_k$ rows among the $r_\infty$ rows of the regulator matrix $\Phi$ yields a non-singular square matrix. As observed by M.H.\ Tran in his Thesis \cite[Def.\ 7.2.2]{TranThesis}, this fact follows from the proof of Dirichlet's units theorem for tori \cite{ShyrDsThm}, and goes back to \cite[\S 3.8]{OnoTori} in the case $k=\Q$. The absolute value of the determinant of any such square submatrix of $\Phi$ is by definition the regulator $R_T$ of $T$ as appears in the class number formula for $T$, see \cite[Thm.\ 1.3]{TranPaper}.

Writing $\Phi_{j}^v$ for the $(v,j)$th row of the regulator matrix $\Phi$, the bound \eqref{last} becomes 
\begin{equation}\label{lastlast}
Y(s)  \ll_{r,T, \nu} U_{\rm max}(s)  R(\sigma,1) \sum_{n \in \Z^{r_\infty-r_k}}  \prod_{(v,j) \in J} \frac{1}{1+|\langle \Phi_j^v,n\rangle|}
\end{equation}
with $\langle \cdot, \cdot \rangle$ the standard inner product on $\R^{r_\infty-r_k}$. 

If $T$ is anisotropic, then $U(T)$ is finite so the sum/product in \eqref{lastlast} is trivial, and the first assertion of Theorem \ref{Ysthm} follows immediately. 

Suppose then that $T$ is isotropic. For $x\in \R$ one has $(1+|x|)^{-1} \leq (1+x^2)^{-1/2}$ and the latter is smooth. Since the summand
$$  \prod_{(v,j) \in J} \frac{1}{\sqrt{1+\langle \Phi_j^v,n\rangle^2}}$$
has well-controlled partial derivatives, we may bound the sum in \eqref{lastlast} by an integral to obtain 
\begin{equation}\label{lastlastlast} Y(s)  \ll_{r,T, \nu} U_{\rm max}(s) R(\sigma,1) \int_{x \in \R^{r_\infty-r_k}}  \prod_{(v,j) \in J} \frac{1}{\sqrt{1+\langle \Phi_j^v,x\rangle^2}}.\end{equation}

Now, we are in a position to apply the Brascamp-Lieb inequality to \eqref{lastlastlast}. We apply Theorem \ref{BL} with $m=r_\infty$, $n=r_\infty-r_k$, the matrix $M=\Phi$ the regulator matrix with rows $a_j^v=\Phi_j^v$, and $f_j^v(x)=(1+x^2)^{-1/2}$ for all $(v,j)\in J$. Recall the definition of the polytope $H_\Phi$ from \eqref{CLLcondition1} and \eqref{CLLcondition2}. Since $\Phi$ is full-rank, $H_\Phi$ is compact. Then, there exists a point in $H_\Phi$, say $(d_j^v)_{(v,j)\in J}\in H_\Phi$, for which the infimum 
$$B_\infty(\Phi)= \inf\{\| x\|_\infty : x \in H_\Phi\}$$ 
is attained (recall \eqref{B0def}). Applying Theorem \ref{BL} with $\overline{p}= (d_j^v)_{(v,j)\in J}$ and the other parameters chosen as above, we obtain
\begin{equation}\label{lastlastlastlast}\int_{x \in \R^{r_\infty-r_k}}  \prod_{(v,j) \in J} \frac{1}{\sqrt{1+\langle \Phi_j^v,x\rangle^2}} \ll_{T,r} \prod_{(v,j) \in J} \left(\int_\R \frac{1}{(1+x^2)^{1/2d_j^v}}\right)^{d_{j}^v}.\end{equation}

Now, to show that the sum over $U_N(T)$ in \eqref{Zexpression2} converges absolutely it suffices now to show that $\min\{1/d_j^v: (v,j) \in J\}>1$. Since $\min\{1/d_j^v: (v,j) \in J\} = B_\infty(\Phi)^{-1}$ by definition of $(d_j^v)_{(v,j)\in J}$, it suffices to show that $B_\infty(\Phi)<1$. We have (recall Proposition \ref{polytopeconj} and Definition \ref{biasdef}) that $$B_\infty(\Phi) = \max\{ \frac{\beta}{\alpha} : \Phi \text{ is } (\alpha;\beta)\text{-biased}\}.$$ Suppose that $\alpha$, $\beta$ are such that $\Phi$ is $(\alpha;\beta)$-biased, i.e.\ there is a distinguished set, say $A$, of rows of $\Phi$ with $|A|=\alpha$ such that any basis of $\R^{r_\infty-r_k}$ contains at least $\beta$ of the $\alpha$ distinguished rows. We aim to show that $\beta<\alpha$ for any such configuration. 

 If $\alpha> r_\infty-r_k$, then we would have $\alpha>r_\infty-r_k\geq \beta$. Suppose $\alpha \leq r_\infty-r_k$. We shall construct a basis of $\R^{r_\infty-r_k}$ from the rows of $\Phi$ that uses at most $\max(0, \alpha-r_k)$ of the rows in $A$. Thus, we will have that $\beta\leq \max(0, \alpha-r_k),$ and since $T$ is isotropic we get $\beta<\alpha$. Recall the key fact that any $r_\infty-r_k$ rows of the regulator matrix $\Phi$ form a basis for $\R^{r_\infty-r_k}$. So, to construct the promised basis, choose the $\min(r_\infty-r_k,r_\infty-\alpha)$ rows of $\Phi$ that are not in $A$ along with any $r_\infty-r_k - \min(r_\infty-r_k,r_\infty-\alpha) = \max(0, \alpha-r_k)$ rows from $A$. This proves the first assertion of Theorem \ref{Ysthm} in the isotropic case.

By Theorem \ref{archthm}\eqref{archthm3}, Theorem \ref{unrthm}\eqref{unrthm2}\eqref{unrthm3}, and Lemma \ref{RamLemma}, the leading coefficient in the Laurent series expansion of $U(s,x)R(s,x)A(s,x)$ at $s=A$ is positive for each $x \in U_N(T)$. Therefore, there can be no cancellation in the leading order Laurent coefficients in the sum over $U_N(T)$ of these in formula \eqref{Zexpression2}.  The assertion in the first bullet point of Theorem \ref{Ysthm} follows from this along with Theorems \ref{unrthm}\eqref{unrthm1prime} and \ref{unrthmwithArtin}\eqref{unrthm1primeArtin}.

The second and third bullet points of Theorem \ref{Ysthm} follow from Theorems \ref{unrthm}\eqref{unrthm4prime} and \ref{unrthmwithArtin}\eqref{unrthm4}, and the previously established absolute convergence of the integral in \eqref{lastlastlastlast}.

\section{Final counting}\label{final}
\subsection{Local to global and proof of Theorem \ref{Zsthm}}\label{locglobandconclusion}
\begin{proposition}\label{localtoglobal}
Suppose that $r\vert_{\widehat{T}}$ is faithful. 
There exists $c \in \R_{>0}$ depending only on choices of Haar measures so that $$ Z(s):= \int_{\cA(T)} \frac{1}{\fc(\chi,r)^s}\,d\nu(\chi) = \frac{c}{\Vol(U_N(T)^\wedge)} \sum_{\theta \in \cl_N(T)^\wedge} \sum_{x \in U_N(T)} \prod_{v \in S_\infty}A_{k_v}(s,x) \prod_{v \not \in S_\infty} N_{k_v}(s,x)$$ for all $s$ with $\real(s)$ sufficiently large, where $A_{k_v}(s,x)$ and $N_{k_v}(s,x)$ are local archimedean and non-archimedean generating series defined in   \eqref{AFdef} and \eqref{NFsxdef}, respectively.\end{proposition}
\begin{proof}  
Recall (see section \ref{subsec:localtoglobal}) the short exact sequence of  locally compact Hausdorff topological abelian groups
\begin{equation} \xymatrix{1 \ar[r] &  U_N(T) \big\backslash T_{N,\A} \ar[r] & T(k) \big\backslash T(\A) \ar[r] & \cl_N(T) \ar[r] & 1.}\end{equation}
Let \begin{equation}\label{Vwedge}V^\wedge = \{\chi \in T_{N,\A}^\wedge: \chi(x) = 1 \text{ for all } x \in U_N(T)\}.\end{equation} By Pontryagin duality, we have the dual short exact sequence
\begin{equation}\label{besdual} \xymatrix{1 \ar[r] &  \cl_N(T)^\wedge \ar[r] & \cA(T) \ar[r] & V^\wedge \ar[r] & 1.}\end{equation}
Recall that we have chosen a Haar measure $\nu$ on $\cA(T)$. The finite group $\cl_N(T)^\wedge$ naturally takes the counting measure, so these determine a quotient Haar measure $\overline{\nu}$ on $V^\wedge$.  Let us write $T_\infty = T(k_\infty) = \prod_{v \in S_\infty} T(k_v)$ and $T_\infty^\wedge = T(k_\infty)^\wedge = \prod_{v \in S_\infty} T(k_v)^\wedge$. 

We work a bit more generally than necessary for the time being. Let $c$ be an integrable function on $\cA(T)$ that satisfies the following factorization property: there exist functions $c_\infty$ on $T_\infty^\wedge$ and $c_f$ on $T(\A_{\rm fin})^\wedge$ such that if $\chi= \chi_\infty \otimes \chi_{f}$ with $\chi_\infty \in T(F_\infty)^\wedge$ and $\chi_f \in T(\A_{\rm fin})^\wedge$, then $$c(\chi) = c_f(\chi_f)c_\infty(\chi_\infty).$$
For such a function $c$ we decompose its integral over $\cA(T)$ using the quotient measure, i.e.\ we apply e.g.\ \cite[Ch.\ VII \S 2 7.\ Prop.\ 10]{BourbakiIntegration} with $G = \cA(T)$, $G'=\cl_N(T)^\wedge$, $G''= V^\wedge$, $\pi$ the restriction map $\pi: \cA(T) \to V^\wedge$, $\alpha = \nu$, $\alpha'$ equal to counting measure, and $\alpha''= \overline{\nu}$ to obtain
\begin{equation}
\int_{\cA(T)} c(\chi) \,d\nu(\chi) = \int_{V^\wedge} \sum_{\theta \in \cl_N(T)^\wedge} c(\theta \chi) \,d\overline{\nu}(\chi).
\end{equation}
Let $$\overline{c}(\chi) = \sum_{\theta \in \cl_N(T)^\wedge} c(\theta \chi),$$ which only depends on $\pi(\chi) \in V^\wedge$. Since $\theta$ is trivial on $T_\infty^\wedge$ (see \eqref{beginninges}), we have  for $\chi = \chi_\infty \otimes \chi_f$  by the factorization property of $c$ that 
\begin{equation}
 \overline{c}( \chi) = c_\infty(\chi_\infty) \sum_{\theta \in \cl_N(T)^\wedge} c_f(\theta \chi_f)=: c_\infty(\chi_\infty) \overline{c}_f(\chi_f).
 \end{equation}
Now let $NT_f^\wedge = \prod_{v \nmid \infty} NT(\cO_w)^\wedge$. It is a discrete group, so we give it the counting measure. We decompose the integral of $\overline{c}$ over $V^\wedge$ as an iterated integral
\begin{equation}
\int_{\cA(T)} c(\chi) \,d\nu(\chi) = \int_{V^\wedge} \overline{c}(\chi)\,d\overline{\nu}(\chi) = \sum_{\chi_f \in NT_f^\wedge}  \overline{c}_f(\chi_f) \int_{ \substack{\chi_\infty \in T_\infty^\wedge \\ \chi_f\chi_\infty(x) = 1 \,\forall \, x \in U_N(T)}} c_\infty(\chi_\infty)\,d\tilde{\nu}(\chi_\infty).
\end{equation}
where $\tilde{\nu}$ is the quotient Haar measure of $\overline{\nu}$ by the counting measure on $NT_f^\wedge$. Let 
\begin{equation}\label{Vinftywedge} 
V_\infty^\wedge = \{\chi_\infty \in T_\infty^\wedge: \chi_\infty(x) = 1 \text{ for all } x \in U_N(T)\}
\end{equation}
so that 
\begin{equation}\label{chi_fV_infty^wedge}\chi_f^{-1}V_\infty^\wedge = \{\chi_\infty \in T_\infty^\wedge : \chi_f\chi_\infty(x) = 1 \text{ for all } x \in U_N(T),\end{equation}
where on the left hand side of \eqref{chi_fV_infty^wedge} $\chi_f^{-1}$ means any element of $T_\infty^\wedge$ that takes the same values as $\chi_f^{-1}$ on all $x \in U_N(T)$. Then, by a change of variables (using the invariance of the Haar measure) we have
\begin{equation}\label{readyforPsum}
\int_{\cA(T)} c(\chi) \,d\nu(\chi) =\sum_{\chi_f \in NT_f^\wedge}  \overline{c}_f(\chi_f)  \int_{V_\infty^\wedge} c_\infty(\chi_f^{-1}\chi_\infty)\,d\tilde{\nu}(\chi_\infty).
\end{equation}
We want to apply Poisson summation (Lemma \ref{Psum}) to the interior integral on the right hand side of \eqref{readyforPsum}. We have the following two dual short exact sequences of locally compact Hausdorff topological abelian groups:
$$\xymatrix{ 1 \ar[r] & U_N(T) \ar[r] & T_\infty \ar[r] & V_\infty^{\wedge\wedge} \ar[r] & 1}$$ and by Pontryagin duality $$\xymatrix{ 1 \ar[r] & V_\infty^\wedge \ar[r] & T_\infty^\wedge \ar[r] & U_N(T)^{\wedge} \ar[r] & 1}.$$
Recall (section \ref{subsec:localtoglobal}) that $U_N(T)$ is discrete so that $U_N(T)^{\wedge}$ is compact. It is also convenient to note that the interior integral on the right hand side of \eqref{readyforPsum} only depends the image of $\chi_f^{-1} \in T_\infty^\wedge$ in $U_N(T)^\wedge$. 
 
 Now we invoke Lemma \ref{Psum} with $G = T_\infty^\wedge$, $ H = V_\infty^\wedge$, $f=\fc_\infty$ where $$\fc_\infty(\chi_\infty) := \prod_{v \in S_\infty} \frac{1}{\fc_v(\chi_v,r \vert{\LT_v})^s} \quad \text{ for } \quad \chi_\infty = \bigotimes_{v \in S_\infty}\chi_v,$$
 and $x = \chi_f^{-1} \in U_N(T)^\wedge$.  Note that $\fc_\infty \in L^1(T_\infty^\wedge)$ by Theorem \ref{archthm}\eqref{archthm1}.
 
We check the hypotheses \eqref{Psum1},\eqref{Psum2} and \eqref{Psum3} of Lemma \ref{Psum}. 
 In section \ref{conclusion} we only used the trivial bounds $U(s,x)\ll (\log (|t|+3))^J$ and $|R(s,x)|\leq R(\sigma,1)$, so that the proof there in fact shows that $$\sum_{x \in U_N(T)}\left| \int_{T_\infty^\wedge} \fc_\infty(\chi_\infty) \overline{\chi_\infty(x)}\,d\chi_\infty\right| < \infty$$
 for all $s\in \C$ with $\real(s)$ sufficiently large and any choice of Haar measure, so that 
 hypothesis \eqref{Psum1} of Lemma \ref{Psum} is satisfied. 
Next, recall that $T(k_v)^\wedge \simeq \Tt^\wedge$ where $\Tt^\wedge$ is given explicitly in \eqref{Ttwedgedef}. Then, indexing several copies of $\Tt^\wedge$ by $v \in S_\infty$,  $T_\infty^\wedge$ is isomorphic to $\prod_{v\mid \infty} \Tt^\wedge_v$ and $V_\infty^\wedge$ is a subgroup of this with compact quotient. For $\chi' \in T_\infty^\wedge$ \begin{equation}\label{verifyPsum2} \int_{V_\infty^\wedge} |\fc_\infty(\chi_\infty \chi')|\,d\tilde{\nu}(\chi_\infty)\end{equation} is an integral of the form $\prod_{v \in S_\infty}A_{k_v}(\sigma,1)$ with $\sigma \in \R_{>0}$ sufficiently large and $A_{k_v}(\sigma,1)$ given as in \eqref{avsxGeneral}, but with the integration restricted to the aforementioned compact quotient coset isomorphic to $\chi'V_\infty^\wedge$. Applying the inequality $(1+|x|)^{-1} \leq (1+x^2)^{-1/2}$ if necessary, one sees that the integrand has well-controlled derivatives, and thus \eqref{verifyPsum2} converges for any $\chi'\in T_\infty^\wedge$ by integral comparison with the integral over $T_\infty^\wedge$ and Theorem \ref{archthm}\eqref{archthm1}, so hypothesis \eqref{Psum2} of Lemma \ref{Psum} is satisfied. 
Lastly, hypothesis \eqref{Psum3} of Lemma \ref{Psum} follows directly from hypothesis \eqref{Psum2} by the dominated convergence theorem since $\fc_\infty$ is itself a continuous function on $T_\infty^\wedge$. 

The result of Lemma \ref{Psum} is that 
\begin{equation}\label{eq:Psumresult}
 \int_{V_\infty^\wedge} \fc_\infty(\chi_f^{-1}\chi_\infty)\,d\tilde{\nu}(\chi_\infty) = \frac{c}{\vol(U_N(T)^\wedge)} \sum_{x \in U_N(T)} \prod_{v \in S_\infty} A_{k_v}(s,x) \overline{\chi_f(x)},
\end{equation}
for some constant $c$ depending only on the choices of Haar measures involved. Finally, the function $\fc(\chi,r)^{-s}$ on $\cA(T)$ given in Definition \ref{GlobalAC} clearly satisfies the factorization property and is integrable for $\real(s)$ sufficiently large by the convergence of \eqref{eq:Psumresult}, Theorem \ref{unrthm}\eqref{unrthm1},  Lemma \ref{RamLemma}, and \eqref{readyforPsum}. Combining \eqref{eq:Psumresult} and formula \eqref{readyforPsum}, rearranging and pulling the finite sum over $\cl_N(T)$ to the outside, we obtain the formula in Proposition \ref{localtoglobal}.
\end{proof}

Finally, Proposition \ref{localtoglobal} and Theorem \ref{Ysthm} together imply Theorem \ref{Zsthm}, which in term implies the first assertion of Theorem \ref{MT1} following \cite[Ch.\ III \S 11]{InghamPrimeNumbers} with instances of $-\frac{\zeta'}{\zeta}(s)$ replaced by $Z(s)$ and \cite[Thm.\ 20]{InghamPrimeNumbers} replaced by the second bullet point of Theorem \ref{Zsthm}. The second assertion of Theorem \ref{MT1} on the power-saving error term follows from Theorem \ref{Zsthm} by \cite[Thm.\ A.1]{CLTsch}.

\subsection{Proof of last assertion of Theorem \ref{MT1}}
\begin{theorem}\label{MT2ndhalf}
If $r \vert_{\widehat{T}}$ is not faithful, then $\nu(\{\chi \in \cA(T): c(\chi,r)\leq X\}) = \infty$ for some finite $X$.\end{theorem}
\begin{proof}
We use the notation introduced in the course of the proof of Proposition \ref{localtoglobal}. Recall the exact sequence \eqref{besdual}, and in particular the cokernel $$V^\wedge = \{ (\chi_\infty, \chi_f): \chi_\infty(x) \chi_f(x) =1 \text{ for all } x \in U_N(T)\},$$ where $\chi_\infty \in T_\infty^\wedge$ and $\chi_f \in NT_f^{\wedge}$. For prove Theorem \ref{MT2ndhalf} it suffices to construct a subset of $V^\wedge$ of infinite Haar measure on which the analytic conductor remains bounded. 

By hypothesis we have that $\{1\} \subsetneq \ker r \vert_{\widehat{T}}$. Let $S_0$ denote the set of unramified places of $k$ which split completely in $K/k$, and for which $q_{K_w} \equiv 1 \pmod {|\pi_0(\ker r\vert_{\widehat{T}})|}.$ By the Chebotarev density theorem, $|S_0|=\infty$. For any $v \in S_0$ we have by Lemma \ref{Amano} and Proposition \ref{PairingProp} that $$NT(\cO_w)^\wedge = T(\cO_v)^\wedge \simeq \Hom_G(\cO_w^\times, \widehat{T}).$$ Write $\ell$ for the residue field of $K_w$. By construction of $S_0$, we have $|\Hom_G(\ell^\times, \ker r\vert_{\widehat{T}})|\geq 2$ for any $v \in S_0$, and therefore $|\Hom_G(\cO_w^\times, \ker r\vert_{\widehat{T}})|\geq 2$ as well. All $\xi \in \Hom_G(\cO_w^\times, \ker r\vert_{\widehat{T}}) \subseteq \Hom_G(\cO_w^\times, \widehat{T})$ have $r \circ \varphi$ of trivial Artin conductor, where  $\varphi \in \Phi(T)$ is such that $\xi = \varphi \vert_{\cO_w^\times}$ as in \eqref{bij_LLP_to_coh}. For any $v \in S_0$ let $T(\cO_v)^\wedge_0$ be the subset of $T(\cO_v)^\wedge$ corresponding to $\Hom_G(\cO_w^\times, \ker r\vert_{\widehat{T}})$ across the restricted Langlands perfect pairing \eqref{unrampairing}. Let $NT_{f,0}^\wedge  \subseteq NT_f^\wedge$ be given by $$ NT_{f,0}^\wedge= \prod_{v \in S_0} T(\cO_v)^\wedge_0 \times \prod_{v \not \in S_0 \cup S_\infty} \{1\}.$$ Then $NT_{f,0}^\wedge$ is an infinite set such that every $\chi_f = (\chi_v)_{v \not \in S_\infty} \in NT_{f,0}^\wedge$ satisfies $$ \prod_{v\not \in S_\infty} \fc_v(\chi_v,r) = 1.$$ 

To construct a subset of $V^\wedge$ of infinite measure on which $\fc(\chi,r)$ is bounded, it suffices to extend each $\chi_f \in NT_{f,0}^\wedge$ to $V(T)$. Recall the notation $$\chi_f^{-1}V_\infty^\wedge = \{\chi_\infty \in T_\infty^\wedge : \chi_\infty(x) \chi_f(x)= 1 \text{ for all } x \in U_N(T)\}.$$

\begin{lemma}\label{Keps} There exist constants $K,\eps>0$ and a subset $X(\chi_f) \subseteq  \chi_f^{-1}V_\infty^\wedge$ for each $\chi_f \in NT_f^\wedge$ such that $\nu(X(\chi_f))>\eps$ and $$\sup\{ \prod_{v \in S_\infty} \fc_v(\chi_v,r): \chi_\infty \in X(\chi_f)\}\leq K$$ for all $\chi_f \in NT_f^\wedge$.  
\end{lemma}
\begin{proof}
We give an explicit description of the sets $\chi_f^{-1}V_\infty^\wedge$ in terms of the corresponding Langlands parameters across \eqref{langlandsunitary}. 
Let $x_1,\ldots, x_s$ be generators for $U_N(T)$. Recall $\Tt$ from \eqref{TRisom}. For $v \in S_\infty$ let $\sigma_v : U_N(T) \hookrightarrow T(k_v) \simeq \Tt$ be the corresponding embedding. As in section \ref{archimedeanC}, we write $$\sigma_v x_i = (\ldots, x_{vij},\ldots,x_{vij}',\ldots, x_{vij}'',\ldots),$$ where $x_{vij} \in \R^\times$, $x_{vij}' \in S^1$, and $x_{vij}'' \in \C^\times$. 

Let us index several copies of $\Tt^\wedge$ (see \eqref{Ttwedgedef}) by $v \in S_\infty$, so that $T_\infty^\wedge \simeq \prod_{v\mid \infty} \Tt_v^\wedge$. Consider the image of $\chi_f^{-1} V_\infty^\wedge$ across the map  \begin{equation}\label{chifVinfty} \prod_{v \in S_\infty} T(k_v)^\wedge \to \prod_{v \in S_\infty} \Tt^\wedge_{v}.\end{equation}
We write elements of $\prod_{v  \mid \infty} \Tt^\wedge_{v}$ as $((w_v,\epsilon_v),\alpha_v,(w'_v,\alpha'_v))_{v \in S_\infty}$. Then the image of $\chi_f^{-1}V_\infty^\wedge$ across \eqref{chifVinfty} is an affine hyperplane in $\prod_{v  \mid \infty} \Tt_{v}^\wedge$ cut out by \begin{equation}\label{hyperplane} \prod_{v \in S_\infty} \prod_{j=1}^{n_1} (\sgn x_{vij})^{\epsilon_{vj}} |x_{vij}|^{w_{vj}} \prod_{j=1}^{n_2} {x_{vij}'}^{\alpha_{vj}} \prod_{j=1}^{n_3} \left( \frac{ | x''_{vij}|}{x''_{vij}}\right)^{ \alpha'_{vj}} |x''_{vij}|^{w'_{vj}-\alpha'_{vj}} = \chi_f(x_i)^{-1},\end{equation} for all generators $x_i$, $i=1,\ldots, s$ of $U_N(T)$. Since $\chi_f(x_i) \in S^1$ for all $\chi_f$ and $x_i$, the affine hyperplane in $\prod_{v \mid \infty} \Tt_{v}^\wedge$ described by \eqref{hyperplane} intersects a fixed (independent of $\chi_f$) compact set around the origin, say $U_0$, in a set of positive measure bounded below independently of $\chi_f$. 

For each $v \in S_\infty$, the set $\{\chi \in T(k_v)^\wedge: \fc_v(\chi,r)\leq X\}$ is in bijection under the local Langlands correspondence with \begin{equation}\label{boundedcondset} H_v = \{\varphi : \prod_{i=1}^{m_1+m_2}(|(M\varphi)_i|+1)\prod_{i=m_1+m_2+1}^{m_1+m_2+m_3} (|(M\varphi)_i|+1)^2 \leq X\}\end{equation} by the results of section \ref{archimedeanC}. Let $$L_v(X)=\{ \varphi: \sum_{i=1}^{m_1+m_2}(|(M\varphi)_i|+1)\sum_{i=m_1+m_2+1}^{m_1+m_2+m_3} (|(M\varphi)_i|+1)^2 \leq (m_1+m_2+m_3)X^{1/(m_1+m_2+m_3)}\}.$$ By the am-gm inequality, we have $L_v(X) \subseteq H_v(X)$. If $X$ is sufficiently large, then $L_v(X)$ contains any fixed compact set in $\Tt^\wedge$, and in particular $U_0 \subseteq \prod_{v \in S_\infty}L_v(X).$ Taking $K= X^{|S_\infty|}$, the lemma is proved. \end{proof}

The fibered set $NT_{f,0}^\wedge \times X(\chi_f) \subseteq V^\wedge$ is our candidate for a set of infinite measure and bounded analytic conductor. 
By Lemma \ref{Keps} and additivity of measure we have $$\nu(NT_{f,0}^\wedge \times X(\chi_f)) = \nu \left(\bigcup_{\chi_f} \{(\chi_\infty,\chi_f) : \chi_\infty\in X(\chi_f)\}\right) = \sum_{\chi_f} \nu\left(  \{(\chi_\infty,\chi_f) : \chi_\infty\in X(\chi_f)\}\right) \geq \infty,$$ yet $\fc(\chi,r)$ is uniformly bounded for any $\chi = (\chi_\infty,\chi_f) \in  NT_{f,0}^\wedge \times X(\chi_f)$. 
\end{proof}

\def\cprime{$'$}

 \end{document}